\documentclass{amsart}

\usepackage{graphicx}

\newtheorem{theorem}{\noindent{\bf Theorem}}[section]
\newtheorem{definition}[]{\noindent{\bf Definition}}[section]
\newtheorem{proposition}[]{\noindent{\bf Proposition}}[section]

\newtheorem{lemma}[]{\noindent{\bf Lemma}}[section]
\newtheorem{corollary}[]{\noindent{\bf Corollary}}[section]
\newtheorem{example}[]{\noindent{\bf Example}}[section]
\newtheorem{remark}[]{\noindent{\bf Remark}}[section]

\begin{document}

\title[Geometric structure and existence]{Geometric structure and existence of reducible spherical conical metrics}
\author { Zhiqiang Wei,~~Yingyi Wu, ~~Bin Xu}


\begin{abstract}
 A conformal metric ${\rm d}s^{2}$ with finitely many conical singularities of constant Gaussian curvature $K=1$ on a compact Riemann surface is referred to as a spherical conical metric.  When the associated monodromy group of ${\rm d}s^{2}$ is diagonalizable, we refer to ${\rm d}s^{2}$ as a reducible  spherical conical metric. The simplest case of a reducible spherical conical metric is a `football', which denotes a 2-sphere with a spherical conical metric that has precisely two singularities separated by a distance of $\pi$. This study delves into the intrinsic geometric structure and existence of reducible spherical conical metrics on compact Riemann surfaces. We demonstrate that any such spherical surface can be divided into a finite number of pieces by cutting along a set of suitable geodesics, which connect the conical singularities and some smooth points of the metric. Especially, each piece is isometric to a portion obtained by cutting a football along a geodesic that joins the two conical singularities. As an application, an angle condition for the existence of such a metric is presented. Most notably, our study demonstrates the existence of a reducible spherical conical metric where all the saddle points of a Morse function are located on the same geodesic.
\vspace*{2mm}

\noindent{\bf Key words}\hskip3mm  Spherical conical metric, Conical singularity, Killing vector field, Abelian differential.
\vspace*{2mm}\\
\noindent{\bf 2020 MR Subject Classification:}\hskip3mm 30F45, 53C23.
\thispagestyle{empty}

\end{abstract}
\maketitle

\section{Introduction}
\setcounter{equation}{0}
We define a metric ${\rm d}s^{2}$ as a conical metric on a compact Riemann surface $\mathcal{M}$ if ${\rm d}s^{2}$ is a conformal Riemannian metric on $\mathcal{M}$, except at a finite number of points labeled $p_{1},\ldots,p_{N}$.  For each point $p_{l}$, there exists a complex coordinate $z_{l}$ with $z_{l}(p_{l})=0$ and a real number $0< \alpha_{l}\neq 1$ such that ${\rm d}s^{2}=e^{2\varphi_{l}}|{\rm d}z_{l}|^{2}$, where
$$\varphi_{l}-(\alpha_{l}-1)\ln|z_{l}|$$
can be continuously extended to $0$. The numbers $\alpha_{l}$ and $2\pi\alpha_{l}$ are referred to as the order and the conical angle of ${\rm d}s^{2}$ at $p_{l}$, respectively. These points $p_{1},\ldots,p_{N}$ are known as the conical singularities of ${\rm d}s^{2}$. For brevity, we say that ${\rm d}s^{2}$ represents the real divisor $D:=\sum\limits_{l=1}^{N}(\alpha_{l}-1)P_{l}$ and refer to the K-surface as $\mathcal{M}_{\{\alpha_{1},\ldots,\alpha_{N}\}}$. This implies that there is a conical metric on $\mathcal{M}$ which represents some real divisor $D=\sum\limits_{l=1}^{N}(\alpha_{l}-1)P_{l}$.\par

A conical metric with constant Gaussian curvature $K=-1,0$ or $1$ is referred to as a hyperbolic, flat or spherical conical metric, respectively. The problem of investigating the existence and uniqueness of such a metric for a given real divisor $D=\sum\limits_{l=1}^{N}(\alpha_{l}-1)P_{l}$, where $1\neq \alpha_{l}>0$ for all $l$, on a compact Riemann surface $\mathcal{M}$, is a fundamental problem in complex analysis and differential geometry. This problem is an extension of the classical uniformization theorem, which is applied to surfaces with boundaries and remains an open question. The history of studying hyperbolic conical metrics  dates back
 at least to Poincar\'{e} \cite {HP98} and Picard \cite{EP05}. This paper focuses on the study of spherical conical metrics, as the existence of hyperbolic and flat conical metrics has already been established \cite{Hei62,McO88,Tr91}. For other issues related to conical metrics, readers are referred to \cite{BD13,Wu11,Ch99,Ch00,JLP13,LZ02,MYN15,MZ20,MZ22,WZ00} and references cited therein.\par

Troyanov \cite{Tr89} demonstrated that a spherical conical metric with two singularities of orders $\alpha$ and $\beta$ exists on the Riemann sphere $S^{2}$  if and only if $\alpha=\beta$. In a subsequent work \cite{Tr91}, he provided a sufficient condition for the existence of a spherical conical metric on a compact Riemann surface $\mathcal{M}$. Luo and Tian \cite{LT92}, when $\mathcal{M}=S^{2}$ and all conical angles are within $(0,2\pi)$, proved that Troyanov's condition is both sufficient and necessary for the existence of a unique metric. Chen and Li \cite{CL91} identified necessary conditions for the existence of spherical conical metrics on a compact Riemann surface under certain constraints. Umehara and Yamada \cite{UY00} in their study of constant mean curvature surfaces in hyperbolic 3-space,  not only provided a necessary and sufficient condition for the existence of spherical conical metrics on $S^{2}$ with 3 conical singularities, but also introduced the concepts of reducible and irreducible spherical conical metrics.  Eremenko \cite{Er04} also studied the existence of spherical conical metrics on $S^{2}$ with 3 conical singularities. Chen et al. \cite{Xu15} through their investigation of the developing map of a spherical conical metric, established a link between the existence of reducible metrics and the presence of a specific type of meromorphic 1-forms on a compact Riemann surface $\mathcal{M}$. More recently, Mondello and Panov \cite{MD16} provided angle constraints for the existence of irreducible spherical conical metrics on $S^{2}$, while Eremenko \cite{Er20} outlined angle constraints for reducible spherical conical metrics. We refer the reader to a comprehensive survey by Eremenko \cite{Er21} for further details. For more information on spherical conical metrics, readers are referred to \cite{DFA11,Br88,Er20b,GT23,Er23,LSX21,MD19,Tah22} and references cited therein.\par

Let's explore the concept of reducible spherical conical metrics as discussed in \cite{Xu15,UY00}. Suppose ${\rm d}s^{2}$ is a spherical conical metric on a compact Riemann surface $\mathcal{M}$. For each smooth point of ${\rm d}s^{2}$, there exists a neighborhood that is isometric to an open subset of the unit sphere $S^{2}$ in $\mathbb{E}^{3}$. Through analytic continuation, we derive an orientation-preserving locally isometric developing map  $f: \mathcal{M}\setminus\{\text{singularities of}~ {\rm d}s^{2}\}\rightarrow S^{2}$, which is a multivalued meromorphic function. This developing map induces a representation of the fundamental group of $\mathcal{M}\setminus\{\text{singularities of}~{\rm d}s^{2}\}$ to the group $PSU(2)$. The image of this representation is referred to as the monodromy group. If the monodromy group is contained within $U(1)$, or if it can be diagonalized, then ${\rm d}s^{2}$ is classified as reducible. \par

The central objective of this paper is to delve into the intricate geometric structure and the existence of reducible spherical conical metrics on compact Riemann surfaces. Our primary result, referred to the \textbf{fundamental theorem of reducible spherical conical metrics}, is outlined below.

\begin{theorem}[\textbf{Fundamental theorem of reducible spherical conical metrics}]\label{Thm-1}
Suppose ${\rm d}s^{2}$ is a reducible spherical conical metric on a compact Riemann surface $\mathcal{M}$, then, there exists an Abelian differential $\omega$ of the third kind on $\mathcal{M}$ with the following properties:
\begin{enumerate}
\item All residues of $\omega$ are nonzero and real.
\item $\omega$ has an exact real part outside its poles.
\end{enumerate}
The conical angles of ${\rm d}s^{2}$ at the zeros of $\omega$ are $2\pi\cdot(ord_{p}(\omega)+1)$, and at the poles, they are either $2\pi \cdot {\rm Res}_{p}(\omega)$ or $-2\pi\cdot {\rm Res}_{p}(\omega)$ depending on the sign of ${\rm Res}_{p}(\omega)$. If ${\rm Res}_{p}(\omega)=\pm1$, then $p$ is a smooth point of ${\rm d} s^{2}$. Furthermore, there exists a smooth surjective function $\Phi:\mathcal{M}\setminus\{\text{poles of} ~\omega\}\rightarrow (0,4)$ that can be continuously extended to $\mathcal{M}$ and satisfies the differential equation $\mathcal{M}\setminus\{\text{poles of} ~\omega\}$
\begin{equation}\label{E-1}
\frac{4{\rm d}\Phi}{\Phi(4-\Phi)}=\omega+\overline{\omega},
\end{equation}
with the initial condition $\Phi(p_{0})\in (0,4)$ for some $p_{0}\in \mathcal{M}\setminus\{\text{poles of} ~\omega \}$. The metric ${\rm d}s^{2}$  is then given by
 $${\rm d}s^{2}=\frac{\Phi(4-\Phi)}{4}\omega\overline{\omega}.$$

Conversely, given an Abelian differential $\omega$ of the third kind on a compact Riemann surface $\mathcal{M}$ with the aforementioned properties, there exists a smooth surjective function $\Phi:\mathcal{M}\setminus\{\text{poles of} ~\omega\}\rightarrow (0,4)$ satisfying equation (\ref{E-1}) and
$${\rm d}s^{2}=\frac{\Phi(4-\Phi)}{4}\omega\overline{\omega}$$
is a reducible spherical conical metric on $\mathcal{M}$. This metric represents the divisor
 $$D=\sum_{i=1}^{I}(\alpha_{i}-1)Q_{i}+\sum_{n=1}^{N} (|{\rm Res}_{p_{n}}(\omega)|-1)P_{n},$$
 where $p_{1},\ldots,p_{N}$ are poles of $\omega$ and $q_{1},\ldots,q_{I}$ are zeros of $\omega$ with orders $\alpha_{1}-1,\ldots,\alpha_{I}-1$ respectively.
\end{theorem}

\begin{remark}
On a compact Riemann surface, an Abelian differential is referred to a meromorphic 1-form (see \cite{G57}).
\end{remark}

\begin{remark}
The equation (\ref{E-1}) exhibits a notable symmetry. By substituting $-\omega$ for $\omega$,  we observe that the metric ${\rm d}s^{2}$ remains invariant. This is due to the transformation of the function $\Phi$ to its complementary value $4-\Phi$, which maintains the integrity of the metric expression.
\end{remark}

\begin{remark}
The first part  of \textbf{Theorem \ref{Thm-1}} is derived from the work of \cite{Xu15} with the exception of equation  (\ref{E-1}) and  specific expression for the metric ${\rm d}s^{2}$.  In \cite{Xu15}, Chen  et al. introduced a Morse function denoted as $\Phi$, which is defined on the compact Riemann surface $\mathcal{M}$ by the formula
$$\Phi(z)=\frac{4|f(z)|^{2}}{1+|f(z)|^{2}},$$
where $f$  is a multivalued meromorphic function defined by equation ${\rm d}\ln f=\omega$. It can be readily verified that $\Phi$ satisfies equation (\ref{E-1}). The authors of \cite{Xu15} referred to $\omega$ as a character 1-form of the metric ${\rm d}s^{2}$.
\end{remark}

Our second result in this paper is the \textbf{structure theorem of reducible spherical conical metrics}. To clarify our theorem, we introduce the concept of a football metric, or simply a `football', which is the elements and the simplest type of resducible spherical conical metrics.\par

\begin{definition}
Let ${\rm d}s^{2}$ be a spherical conical metric with two conical singularities on a 2-sphere. If the distance between these two singularities is $\pi$, we define ${\rm d}s^{2}$ as a football metric. The standard metric on $S^{2}$  also classified as a football metric in this paper. Furthermore, a 2-sphere equipped with a football metric is referred to as a football. We will also denote the standard football metric by $S^{2}_{\{1,1\}}$. In fact, all footballs are reducible spherical conical metrics.
\end{definition}

\begin{figure}[htbp]
\centering
\includegraphics[width=8cm]{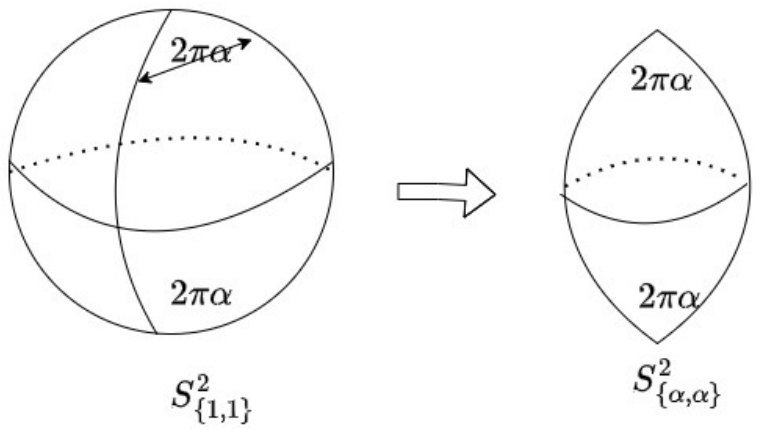}
\caption{Football.}
\end{figure}

For example, on the standard football $S^{2}=S^{2}_{\{1,1\}}$, by taking a bigon with an angle $2\pi\alpha(0<\alpha<1)$ (as shown in Figure 1), we can construct an American football $S^{2}_{\{\alpha,\alpha\}}$.

\begin{theorem}[\textbf{Structure theorem of reducible spherical conical metrics}]\label{Main-th-1}
 Let ${\rm d}s^{2}$ be a reducible spherical conical metric on a compact Riemann surface, denoted by $\mathcal{M}$. Let $\omega$ be a character 1-form associated with ${\rm d}s^{2}$, and let $\Phi$ be a real function on $\mathcal{M}$ as defined by equation (\ref{E-1}). Then, by cutting along a finite number of suitable geodesics that connect the extremal points and saddle points of $\Phi$, $\mathcal{M}$ can be uniquely partitioned into a finite number of pieces. Each piece resulting from this partition is isometric to a piece that would be obtained by cutting a football along a geodesic connecting its two conical singularities.
\end{theorem}

\textbf{Theorem \ref{Main-th-1}}  implies that it is possible to construct such metrics on any compact Riemann surface. Our third result in this paper pertains to the existence of reducible spherical conical metrics on $S^{2}$. To express this result, we introduce the following notations.\par

Suppose $m$ is a nonnegative integer, and $a_{1},a_{2},\ldots,a_{m+2}$ are $m+2$ nonzero real numbers with $a_{1}+\ldots+a_{m+2}=0$. We defined $\overrightarrow{r}=\{a_{1},a_{2},\ldots,a_{m+2}\}$ as a real residue vector, with $a_{1},a_{2},\ldots,a_{m+2}$ being its components.  A real residue vector $\overrightarrow{r}$ is called rational if there exists a nonzero real constant $c$ such that all components of $c\overrightarrow{r}$ are integers. Otherwise, $\overrightarrow{r}$ is called irrational.
For a rational vector $\overrightarrow{r}$, there exists a unique nonzero real number $c$, up to sign, such that all the components of $c\overrightarrow{r}$ are integers and coprime. In this case, $c\overrightarrow{r}$ is referred to as primitive and we define the degree of $\overrightarrow{r}$ by ${\rm deg}(\overrightarrow{r})=\sum\limits_{ca_{i}>0}ca_{i}$. If $\overrightarrow{r}$ is irrational, we define the degree of $\overrightarrow{r}$ by ${\rm deg}(\overrightarrow{r})=+\infty$. Then our existence theorem can now be stated as follows.\par

\begin{theorem}\label{Main-th-4}
Suppose $\beta_{1},\ldots,\beta_{J}$ and $\gamma_{1},\ldots,\gamma_{L}$ are positive real numbers with $\beta_{j}\neq1$ and $\gamma_{l}\neq1$ for all $j$ and $l$, and  $2\leq\alpha_{1},\ldots,\alpha_{I}$ are integers satisfying $I\geq1,J\geq0,L\geq0$ and $I+J+L\geq 4$. There exists a reducible spherical conical metric ${\rm d}s^{2}$ on $S^{2}_{\{\alpha_{1},\ldots,\alpha_{I},\beta_{1},\ldots,\beta_{J},\gamma_{1},\ldots,\gamma_{L}\}}$ such that the function $\Phi$ defined by equation (\ref{E-1}) satisfies
\begin{enumerate}
\item The singularities of angles $2\pi\alpha_{1},\ldots,2\pi\alpha_{I}$ are saddle points of $\Phi$.
\item The singularities of angles $2\pi\beta_{1},\ldots,2\pi\beta_{J}$ are minimum points of $\Phi$, and the singularities of angles $2\pi\gamma_{1},\ldots$, $ 2\pi\gamma_{L}$ are maximum points of $\Phi$,
\end{enumerate}
 if and only if there exist nonnegative integers $p$ and $q$ satisfying the following system of equations
\begin{equation*}
\begin{cases}
p+q=\sum\limits_{i=1}^{I}(\alpha_{i}-1)-(J+L)+2,\\
q-p=\sum\limits_{j=1}^{J}\beta_{j}-\sum\limits_{l=1}^{L}\gamma_{l},
\end{cases}
\end{equation*}
and
$${\rm deg}(\overrightarrow{r})>\max\{\alpha_{i}-1\}_{1\leq i\leq I},$$
where
$$\overrightarrow{r}=\{\beta_{1},\ldots,\beta_{J},-\gamma_{1},\ldots,-\gamma_{L},\underbrace{1,\ldots,1}_{p ~times},\underbrace{-1,\ldots,-1}_{q ~times}\}.$$
\end{theorem}

A spherical conical metric on $S^{2}$ with at most two non-integer orders must be reducible \cite{Xu15}. Hence, from \textbf{Theorem \ref{Main-th-4}}, we obtain the following corollary.\par

\begin{corollary}\label{Cor-N-1}
Suppose $2\leq\alpha_{1},\ldots,\alpha_{I}$ are $I\geq 1$ integers and $\beta,\gamma\neq1$ are positive real numbers. There exists a spherical conical metric on $S^{2}_{\{\alpha_{1},\ldots,\alpha_{I},\beta,\gamma\}}$ such that the singularities of angles $2\pi\alpha_{1},\ldots,2\pi\alpha_{I}$ are all saddle points of $\Phi$ defined by equation (\ref{E-1}) if and only if one of the following conditions is satisfied:
\begin{enumerate}
 \item $\sum\limits_{i=1}^{I}(\alpha_{i}-1)-\beta-\gamma$ and $ \sum\limits_{i=1}^{I}(\alpha_{i}-1)+\beta+\gamma$ are both nonnegative even integers.
 \item $\sum\limits_{i=1}^{I}(\alpha_{i}-1)+\beta-\gamma$ and $ \sum\limits_{i=1}^{I}(\alpha_{i}-1)-\beta+\gamma$ are both nonnegative even integers.
\end{enumerate}
\end{corollary}

Additionally, from the proof of \textbf{Theorem \ref{Main-th-4}} in section 9, we obtain the following theorem.

\begin{theorem}
If there exists a reducible spherical conical metric on K-surface $S^{2}_{\{\alpha_{1},\ldots,\alpha_{N}\}}$, then there is a reducible spherical conical metric on $S^{2}_{\{\alpha_{1},\ldots,\alpha_{N}\}}$ such that all saddle points of $\Phi$ defined by equation (\ref{E-1}) lie on the same geodesic connecting a minimum point and a maximum point of $\Phi$. Furthermore, if $\alpha_{1},\ldots,\alpha_{N}$ are integers, using some standard footballs $S^{2}_{\{1,1\}}$, we can construct a reducible spherical cone metric on $S^{2}_{\{\alpha_{1},\ldots,\alpha_{I}\}}$ such that all singularities  lie on the same geodesic connecting a minimum point and a maximum point of $\Phi$ defined by equation (\ref{E-1}).
\end{theorem}

\begin{remark}
The theorem above does not hold on a spherical surface with genus $g\geq1$. For example, the is a reducible spherical conical metric on a toric $T_{\{2,2\}}$, but $T_{\{2,2\}}=S^{2}_{\{\frac{1}{2},\frac{1}{2}\}}+S^{2}_{\{\frac{1}{2},\frac{1}{2}\}}$ (see the notations in section 9).
\end{remark}

Our last result in this paper is about the existence of reducible spherical conical metrics on a compact orientable surface $\mathcal{M}^{g}$ of genus $g\geq1$.

\begin{theorem}\label{Main-thm-5}
Let $g\geq1$ be an integer. Suppose $\beta_{1},\ldots,\beta_{J}$ and $\gamma_{1},\ldots,\gamma_{L}$ are positive real numbers with $\beta_{j}\neq1$ and $\gamma_{l}\neq1$ for all $j$ and $l$, where $J\geq0$ and $L\geq0$. Suppose $2\leq\alpha_{1},\ldots,\alpha_{I}$ are $I\geq1$ integers.
There exist a compact Riemann surface $\mathcal{M}^{g}$ of genus $g$ and a reducible spherical conical metric ${\rm d}s^{2}$ on $$\mathcal{M}^{g}_{\{\alpha_{1},\ldots,\alpha_{I},\beta_{1},\ldots,\beta_{J},\gamma_{1},\ldots,\gamma_{L}\}}$$
if and only if there exist nonnegative integers $p$ and $q$ such that $p+J\geq1,~q+L\geq1$, and
\begin{equation*}
\begin{cases}
(p+J)+(q+L)=\sum\limits_{i=1}^{I}(\alpha_{i}-1)+2-2g,\\
q-p=\sum\limits_{j=1}^{J}\beta_{j}-\sum\limits_{l=1}^{L}\gamma_{l}.
\end{cases}
\end{equation*}

Furthermore, the function $\Phi$ defined by equation (\ref{E-1}) satisfies the following two conditions.
\begin{enumerate}
 \item  The singularities of angles $2\pi\alpha_{1},\ldots,2\pi\alpha_{I}$ are saddle points of $\Phi$.
\item The singularities of angles $2\pi\beta_{1},\ldots,2\pi\beta_{J}$ are  minimum points of $\Phi$, and the singularities of angles $2\pi\gamma_{1},\ldots,$ $2\pi\gamma_{L}$ are maximum points of $\Phi$.
\end{enumerate}
\end{theorem}

\begin{remark}
Recently, Gendron and Tahar \cite{GT23} proposed an angle condition for the existence of reducible spherical metrics. However, our method differs from theirs.
\end{remark}

 The paper is structured as follows: \textbf{Section 2} presents the proof of the fundamental theorem of reducible spherical conical metrics. \textbf{Section 3} applies this theorem to offer a novel proof for the classification \cite{Tr89} of spherical conical metrics on $S^{2}_{\{\alpha,\alpha\}}$. \textbf{Section 4} begins with an examination of the properties of $\Phi$. We then introduce a complex vector field $Y$ defined by $\omega(Y)=\sqrt{-1}$ and a real vector field $\overrightarrow{V}=\frac{1}{2}(Y+\overline{Y})=Re(Y)$. These vector fields are crucial for the proof of the structure theorem of reducible conical spherical metrics.  \textbf{Section 5} details the proof of the structure theorem of reducible spherical conical metrics. \textbf{Section 6} illustrates the geometric construction of spherical conical metrics on $S^{2}_{\{\alpha,\alpha\}}$, serving as a concrete example of the structure theorem. \textbf{Section 7} presents additional explicit examples of spherical conical metrics, demonstrating the wide applicability of the structure theorem. \textbf{Section 8} delves into the existence of reducible spherical conical metrics on $S^{2}$ with 3 singularities, a subject previously investigated by Umehara-Yamada \cite{UY00} and Eremenko \cite{Er04}. Finally, \textbf{Sections 9 and 10} contain the proofs for the existence of reducible spherical conical metrics on  $S^{2}$ and a compact Riemann surface $\mathcal{M}^{g}$ of genus $g$, respectively.

\section{Proof of \textbf{Theorem \ref{Thm-1}}}

In \cite{Xu15}, Chen et al. established the first part of \textbf{Theorem \ref{Thm-1}}, excluding equation  (\ref{E-1}) and the expression for ${\rm d}s^{2}$.  These components can be easily derived from the existing work, so the emphasis now is on proving the second part of \textbf{Theorem \ref{Thm-1}}.\par

Let the poles of $\omega$ be denoted as $p_{1},\ldots,p_{N}$, and the zeros as $q_{1},\ldots,q_{I}$. Since $\omega+\overline{\omega}$ is exact on $ \mathcal{M}\setminus\{p_{1},\ldots,p_{N}\}$, there exists a smooth function $f$ defined on $ \mathcal{M}\setminus\{p_{1},\ldots,p_{N}\}$ such that
$$\omega+\overline{\omega}={\rm d}f.$$
Considering the relation
$$\frac{4{\rm d}\Phi}{\Phi(4-\Phi)}={\rm d}\ln\frac{\Phi}{4-\Phi},$$
we have
$$ \ln\frac{\Phi}{4-\Phi}=f+A_{0},$$
which implies
$$\Phi=\frac{4e^{f+A_{0}}}{1+e^{f+A_{0}}}\in (0,4),$$
where $A_{0}=\ln\frac{\Phi(p_{0})}{4-\Phi(p_{0})}-f(p_{0})=\ln\frac{\Phi_{0}}{4-\Phi_{0}}-f(p_{0})$.\par

(1) $\Phi$ can be continuously extended to $p_{n},1\leq n\leq N$. \par
To prove this, it suffices to show that $\Phi$ has limits at each point $p_{n}$. Let $(D,z)$ be a local complex coordinate disk around a $p_{n}$ with $z(p_{n})=0$, where no other poles or zeros of $\omega$ exist in $D$. Without loss of generality, we can suppose
\begin{equation}\label{w-1}
\omega=\frac{\lambda_{n}}{z}{\rm d}z,
\end{equation}
where $\lambda_{n}\in\mathbb{R}\setminus\{0\}$. Then $(\omega+\overline{\omega})|_{D\setminus\{0\}}={\rm d}(\lambda_{n} \ln |z|^{2})={\rm d}f$, or
$f=\lambda_{n}\ln|z|^{2}+a^{*}$ on $D\setminus\{0\}$, where $a^{*}$ is a real constant. This leads to
\begin{equation}\label{Exp-1}
\Phi=\frac{4|z|^{2\lambda_{n}}e^{a^{*}+A_{0}}}{1+|z|^{2\lambda_{n}}e^{a^{*}+A_{0}}}~ on ~D\setminus\{0\}.
\end{equation}
Thus if $\lambda_{n}>0$, $\lim\limits_{z\rightarrow 0}\Phi=0$; if $\lambda_{n}<0$, $\lim\limits_{z\rightarrow 0}\Phi=4$.\par

(2) ${\rm d}s^{2}=\frac{\Phi(4-\Phi)}{4}\omega\overline{\omega}$ represents a Riemannian metric with constant Gauss curvature $K\equiv 1$.\par
Given $0< \Phi< 4$ and $\omega\neq0$ on $\mathcal{M}\setminus\{q_{1},\ldots,q_{I},p_{1},\ldots,p_{N}\}$, ${\rm d}s^{2}$ serves as a Riemannian metric on $M\setminus\{q_{1},\ldots,q_{I},p_{1},\ldots,p_{N}\}$. Consider a local complex coordinate domain $(X,z)$ on $\mathcal{M}\setminus\{q_{1},\ldots,q_{I},p_{1},\ldots,p_{N}\}$ with $\omega={\rm d}z$ on $X$. Then ${\rm d}s^{2}|_{X}=\frac{\Phi(4-\Phi)}{4}|{\rm d}z|^{2}$. A straightforward calculation reveals that the Gauss curvature of ${\rm d}s^{2}$ is $K\equiv 1$.\par

(3) The points $q_{i}, i=1,\ldots,I$ are conical singularities of ${\rm d}s^{2}$ with conical angles $2\pi(ord_{q_{i}}(\omega)+1),i=1,\ldots,I$, and $p_{n},n=1,\ldots,N$ are conical singularities of ${\rm d}s^{2}$ with conical angles $2\pi |{\rm Res}_{p_{n}}(\omega)|,n=1,\ldots, N$.\par

Select a local complex coordinate chart $(W,z)$ around $q_{i}$ with $z(q_{i})=0$ and suppose $\omega|_{W}=z^{\alpha_{i}-1}{\rm d}z$. Then $${\rm d}s^{2}|_{W\setminus\{0\}}=\frac{\Phi(4-\Phi)}{4}|z|^{2(\alpha_{i}-1)}|{\rm d}z|^{2}.$$ Thus, $q_{i}, i=1,\ldots,I$ are conical singularities of ${\rm d}s^{2}$ with conical angles $2\pi(ord_{q_{i}}(\omega)+1),i=1,\ldots,I$ respectively.

By equations (\ref{w-1}) and (\ref{Exp-1}), it can be easily shown that $p_{n},n=1,\ldots,N$ are conical singularities of ${\rm d}s^{2}$ with conical angles $2\pi |{\rm Res}_{p_{n}}(\omega)|,n=1,\ldots, N$ respectively. Additionally, ${\rm Res}_{p}(\omega)=1$ or ${\rm Res}_{p}(\omega)=-1$ indicates that $p$ is a smooth point of ${\rm d}s^{2}$.

\section{Classification of conformal spherical conical  metrics on $S^{2}_{\{\alpha,\alpha\}}$}

A spherical conical metric on $S^{2}$ that possesses exactly two conical singularities is necessarily reducible, with both singularities featuring equal conical angles (see \cite{Xu15}). By applying \textbf{Theorem \ref{Thm-1}}, we provide a novel proof for the classification of such metrics on $S^{2}_{\{\alpha,\alpha\}}$ as presented in \cite{Tr89}.\par

\begin{theorem}\label{Thm-2}
Regard $S^{2}$ as $\mathbb{C}\cup\{\infty\}$. Let ${\rm d}s^{2}$ be a spherical conical metric on $\mathbb{C}\cup\{\infty\}$ with conical singularities at $z=0$ and $z=\infty$,
each with conical angles $2\pi\alpha$ and $2\pi\alpha$ respectively, where $0<\alpha\neq1$. Then, up to a change of coordinate $z\mapsto pz$, where $p\in \mathbb{C}\setminus\{0\}$ is a constant, the metric can be expressed as follows.
\begin{enumerate}
\item If $\alpha\notin~\mathbb{Z}^{+}$, the metric is given by ${\rm d}s^{2}=\frac{4\alpha^{2}|z|^{2(\alpha-1)}}{(1+|z|^{2})^{2}}|{\rm d}z|^{2}$.
\item If $\alpha\in \mathbb{Z}^{+}$, the metric is given by ${\rm d}s^{2}=\frac{4\alpha^{2}|z|^{2(\alpha-1)}}{(1+|z^{\alpha}+b|^{2})^{2}}|{\rm d}z|^{2}$, where $b\in\mathbb{R}$ is a constant.
\end{enumerate}
\end{theorem}

Initially, we present a standard forms of certain meromorphic 1-forms on the Riemann sphere $S^{2}\cong\mathbb{C}\cup\{\infty\}$.\par

\begin{lemma}\label{L-1}
Let $\omega$ be an Abelian differential of the third kind on $\mathbb{C}\cup\{\infty\}$, with each of its residues being a nonzero real number. The standard forms of $\omega$ are as follows.
\begin{enumerate}
\item If $(\omega)=-0-\infty$, then $\omega=\frac{\lambda}{z}{\rm d}z$, where $\lambda={\rm Res}_{0}(\omega)\neq 0$.
\item If $(\omega)=(\alpha-1)\cdot0-\infty-\sum\limits_{i=1}^{\alpha}P_{i}, \alpha\geq2$, and ${\rm Res}_{p_{i}}(\omega)=1,i=1,\ldots,\alpha$, then, up to a change of coordinate $z\rightarrow pz$, where $p\in \mathbb{C}\setminus\{0\}$ is a constant, $\omega=\frac{\alpha z^{\alpha-1}}{z^{\alpha}+1}{\rm d}z$.
\item If $(\omega)=(\alpha-1)\cdot0+(\alpha-1)\cdot\infty-\sum\limits_{i=1}^{\alpha}P_{i}-\sum\limits_{j=1}^{\alpha}Q_{j}, \alpha\geq2$,  and ${\rm Res}_{p_{i}}(\omega)=1,{\rm Res}_{q_{i}}(\omega)=-1, i=1,\ldots,\alpha$, then, up to a change of coordinate $z\rightarrow pz$, where $p\in \mathbb{C}\setminus\{0\}$ is a constant,  $\omega=\frac{\alpha (a- 1) z^{\alpha-1}}{(z^{\alpha}+a)(z^{\alpha}+1)}{\rm d}z$, where $a\neq0,1$ is a complex constant.
\end{enumerate}
\end{lemma}

\begin{proof}
(1) It is evident and straightforward.\par

(2) Suppose $p_{i}=a_{i}\in\mathbb{C}\setminus\{0\},i=1,\ldots,\alpha,\mu\neq 0$ and
$$\omega=(\sum_{i=1}^{\alpha}\frac{1}{z-a_{i}}){\rm d}z=\frac{\alpha \mu z^{\alpha-1}}{\prod\limits_{i=1}^{\alpha}(z-a_{i})}{\rm d}z.$$

Setting $t(z)=\prod\limits_{i=1}^{\alpha}(z-a_{i})$, we find that
$$t'(z)=\alpha \mu z^{\alpha-1}.$$
Hence, $\mu=1$ and $t(z)= z^{\alpha}+a$, where $a\in\mathbb{C}\setminus\{0\}$ is a constant,  leading to $\omega=\frac{\alpha  z^{\alpha-1}}{z^{\alpha}+a}{\rm d}z$.

By setting $z=a^{\frac{1}{\alpha}}w$, we obtain
$$\omega=\frac{\alpha  w^{\alpha-1}}{w^{\alpha}+1}{\rm d}w.$$

(3) Suppose $p_{i}=a_{i},q_{i}=b_{i}\in\mathbb{C}\setminus\{0\},i=1,\ldots,\alpha,\mu\neq0$ and
$$\omega=(\sum_{i=1}^{\alpha}\frac{1}{z-a_{i}}-\sum_{i=1}^{\alpha}\frac{1}{z-b_{i}} ){\rm d}z=\frac{\alpha \mu z^{\alpha-1}}{\prod\limits_{i=1}^{\alpha}(z-a_{i})\prod\limits_{i=1}^{\alpha}(z-b_{i})}{\rm d}z.$$

Setting $t(z)=\prod\limits_{i=1}^{\alpha}(z-a_{i}),~~~s(z)=\prod\limits_{i=1}^{\alpha}(z-b_{i})$, we derive the equation
\begin{equation}\label{L-E-1}
t'(z)s(z)-t(z)s'(z)=\alpha \mu z^{\alpha-1}.
\end{equation}

By expressing $t(z)$ and $s(z)$ as polynomials
$$t(z)=\sum_{i=0}^{\alpha}\omega_{i}z^{i},~~~s(z)=\sum_{i=0}^{\alpha}\sigma_{i}z^{i},$$
with $\omega_{\alpha}=\sigma_{\alpha}=1$ and $\omega_{0},\sigma_{0}\neq 0$.
we can rewrite the equation  (\ref{L-E-1}) as
\begin{equation}\label{L-E-2}
\sum_{i=1}^{\alpha}\sum_{j=0}^{\alpha}i(\omega_{i}\sigma_{j}-\sigma_{i}\omega_{j})z^{i+j-1}=\alpha \mu z^{\alpha-1}.
\end{equation}

Through direct calculation, we obtain
$$t(z)=F(z)+\omega_{0},~~~s(z)=F(z)+\sigma_{0},$$
where $F(z)=z^{\alpha}+\omega_{\alpha-1}z^{\alpha-1}+\ldots+\omega_{1}z$.

By equation (\ref{L-E-1}), we obtain
$$F'(z)(\sigma_{0}-\omega_{0})=\alpha \mu z^{\alpha-1}.$$

This implies $\omega_{1}=\ldots=\omega_{\alpha-1}=0,\sigma_{0}-\omega_{0}=\mu$, hence
$$t(z)=z^{\alpha}+\omega_{0}, ~~~s(z)=z^{\alpha}+\sigma_{0}.$$

Thus, the differential $\omega$ can be expressed as
$$\omega=\frac{\alpha (\sigma_{0}-\omega_{0}) z^{\alpha-1}}{(z^{\alpha}+\omega_{0})(z^{\alpha}+\sigma_{0})}{\rm d}z.$$

By setting $z=\omega_{0}^{\frac{1}{\alpha}}w$, we obtain
$$\omega=\frac{\alpha (a-1)w^{\alpha-1}}{(w^{\alpha}+1)(w^{\alpha}+a)}{\rm d}w,$$
where $a=\frac{\sigma_{0}}{\omega_{0}}$ is a complex constant not equal to $0$ or $1$.
\end{proof}

\textbf{Proof of Theorem \ref{Thm-2}}\par
(1) Let $\omega_{}=\frac{\alpha}{z}{\rm d}z$. Through direct calculation, we find that
$${\rm d}s^{2}=\frac{4\alpha^{2}e^{A_{0}}|z|^{2(\alpha-1)}}{(1+e^{A_{0}}|z|^{2\alpha})^{2}}|{\rm d}z|^{2},$$
where $A_{0}\in\mathbb{R}$ is a constant.\par
Setting $z=e^{-\frac{A_{0}}{2\alpha}}w$, we obtain
$${\rm d}s^{2}=\frac{4\alpha^{2}|w|^{2(\alpha-1)}}{(1+|w|^{2\alpha})^{2}}|{\rm d}w|^{2}.$$

(2) Let $\omega=\frac{\alpha z^{\alpha-1}}{z^{\alpha}+1}{\rm d}z$ for $2\leq\alpha\in\mathbb{Z}^{+}$.
Through direct calculation, we derive
$${\rm d}s^{2}=\frac{4\alpha^{2}e^{A_{0}}|z|^{2(\alpha-1)}}{(1+e^{A_{0}}|z^{\alpha}+1|^{2})^{2}}|{\rm d}z|^{2},$$
where $A_{0}\in\mathbb{R}$ is a constant.\par
Setting $z=e^{\frac{-A_{0}}{2\alpha}}w$, then
$${\rm d}s^{2}=\frac{4\alpha^{2}|w|^{2(\alpha-1)}}{(1+|w^{\alpha}+b|^{2})^{2}}|{\rm d} w|^{2},$$
where $b=e^{\frac{A_{0}}{2}}\in\mathbb{R}\setminus\{0\}$.\par

(3) Let $\omega=\frac{\alpha (a-1) z^{\alpha-1}}{(z^{\alpha}+a)(z^{\alpha}+1)}{\rm d}z$, where $a$ is a complex constant with $a\neq0,1$.
Through direct calculation, we obtain
$${\rm d}s^{2}=\frac{4\alpha^{2}e^{A_{0}}|a-1|^{2}|z|^{2(\alpha-1)}}{(|z^{\alpha}+a|^{2}+e^{A_{0}}|z^{\alpha}+1|^{2})^{2}}|{\rm d}z|^{2},$$
where $A_{0}\in\mathbb{R}$ is a constant.\par
Denote $e^{A_{0}}$ by $\lambda^{2}$, then
\begin{equation*}
\begin{aligned}
{\rm d}s^{2}&=\frac{4\alpha^{2}\lambda^{2}|a-1|^{2}|z|^{2(\alpha-1)}}{(|z^{\alpha}+a|^{2}+\lambda^{2}|z^{\alpha}+1|^{2})^{2}}|{\rm d}z|^{2}\\
&=\frac{4\alpha^{2}\lambda^{2}|a-1|^{2}|z|^{2(\alpha-1)}}{(1+\lambda^{2})^{2}[|z^{\alpha}+\frac{a+\lambda^{2}}{1+\lambda^{2}}|^{2}+
\frac{\lambda^{2}|a-1|^{2}}{(1+\lambda^{2})^{2}}]^{2}}|{\rm d}z|^{2}.
\end{aligned}
\end{equation*}

Setting $z=pw$, where $|p^{\alpha}|=\frac{\lambda|a-1|}{1+\lambda^{2}}$ such that $b=\frac{a+\lambda^{2}}{p^{\alpha}(1+\lambda^{2})}\in\mathbb{R}$, then
$${\rm d}s^{2}=\frac{4\alpha^{2}|w|^{2(\alpha-1)}}{(1+|w^{\alpha}+b|^{2})^{2}}|{\rm d}w|^{2}.$$

\section{Some properties}

\begin{proposition}\label{Pro-2-1}
The function $\Phi$, as defined by equation (\ref{E-1}), exhibits local rotational symmetry on $\mathcal{M}$.
\end{proposition}

\begin{proof}~
Suppose $p_{0}\in \mathcal{M}\setminus\{\text{zeros and poles of } \omega\}$, then, there exists a complex coordinate chart $(U,z)$ with $z(p_{0})=0$ in $ \mathcal{M}\setminus\{\text{zeros and poles of }\omega \}$ such that $\omega={\rm d}z$. By setting $z=x+\sqrt{-1}y$ and using equation (\ref{E-1}), we derive
$$\frac{2{\rm d}\Phi}{\Phi(4-\Phi)}={\rm d}x.$$
Thus $\Phi(x)=\frac{4 e^{2x+2C}}{1+e^{2x+2C}}$, where $C$ is a constant. The metric ${\rm d}s^{2}$, as presented in \textbf{Theorem \ref{Thm-1}}, can be written as
\begin{equation*}
\begin{aligned}
{\rm d}s^{2}&=\frac{\Phi(4-\Phi)}{4}\omega\overline{\omega}=\frac{4e^{2x+2C}}{(1+e^{2x+2C})^{2}}({\rm d}x^{2}+{\rm d}y^{2})\\
&=(\frac{2e^{x+C}}{1+e^{2x+2C}}{\rm d}x)^{2}+(\frac{2e^{x+C}}{1+e^{2x+2C}}{\rm d}y)^{2}\\
&={\rm d}r^{2}+\sin^{2}r{\rm d}y^{2},
\end{aligned}
\end{equation*}
where $r=2\arctan e^{x+C}$. Consequently, $\Phi(r)=4\sin^{2}\frac{r}{2}$.\par
It is evident that $(r,y)$ serve as local polar coordinates centered at $p_{0}$ on $\mathcal{M}\setminus\{\text{zeros and poles of } \omega\}$. Since $\Phi$ is a continuous function on $\mathcal{M}$, it naturally exhibits local rotational symmetry on $\mathcal{M}$.
\end{proof}

Define Y by $\omega(Y)=\sqrt{-1}$, then $Y$ is a meromorphic vector field on $\mathcal{M}$ since $\omega$ is meromorphic 1-form. Additionally, we define the real vector field $\overrightarrow{V}$ as the real part of $Y$, given by
$$\overrightarrow{V}=\frac{1}{2}(Y+\overline{Y})=Re(Y).$$
The proof of \textbf{Proposition \ref{Pro-2-1}} yields the subsequent property.

\begin{proposition}\label{Pro-2-2}
Let $Sing(\overrightarrow{V})$ be the set of singular points of $\overrightarrow{V}$. Then,
$$Sing(\overrightarrow{V})=\{\text{poles and zeros of } \omega\}.$$
Additionally, $\overrightarrow{V}$ is a Killing vector field on $\mathcal{M}\setminus Sing(\overrightarrow{V})$ and $\overrightarrow{V}\perp \nabla \Phi$, where $\nabla \Phi$ denotes the real gradient of $\Phi$.
\end{proposition}

Let $\Omega_{p}$ be the set of all integral curves of $\overrightarrow{V}$ that intersect at point $p$. The cardinality of $\Omega_{p}$ is denoted by $|\Omega_{p}|$. For any point $p$ not in the singular set $Sing(\overrightarrow{V})$, there exists a unique integral curve of $\overrightarrow{V}$ passing through point $p$. Therefore, $\Omega_{p}$ is well-defined and consists of exactly two elements: one integral curve entering $p$ and another leaving $p$. However, at a singular point $p$ of $\overrightarrow{V}$, the concept of ``an integral curve of $\overrightarrow{V}$ passing through $p$" is ambiguous, as $\overrightarrow{V}$ may be undefined or vanish at point $p$. To clarify this, an integral curve $C$ of $\overrightarrow{V}$ is considered to be in $\Omega_{p}$ (for $p\in Sing(\overrightarrow{V})$) if and only if there exists a sequence of points in $C$ that converges to $p$. In simpler terms, if $C$ is not in $\Omega_{p}$, then there exists a small disk $B$ centered at $p$ such that $C\cap B=\emptyset$.\par

\begin{proposition}\label{Pro-2-3}
In a small neighborhood of any local extremal point of the function $\Phi$, any integral curve $c_{t}$ of the vector field $\overrightarrow{V}$ forms a topological circle that encompasses the extremal point within its interior.
\end{proposition}

\begin{proof}~
Without loss of generality, we can suppose $\omega=\frac{\alpha}{z}{\rm d}z$ with $\alpha\in \mathbb{R}\setminus\{0\}$ at an extremal point of $\Phi$. Then, we have $Y=\sqrt{-1}\frac{z}{\alpha}\frac{\partial}{\partial z}$. By setting $z=x+\sqrt{-1}y$, we obtain
 $$\overrightarrow{V}=\frac{1}{2\alpha}(x\frac{\partial}{\partial y}-y\frac{\partial}{\partial x}).$$
It is evident that any integral curve of $\overrightarrow{V}$ forms a topological circle containing the extremal point within its interior.
\end{proof}

\begin{proposition}\label{Pro-2-4}
 $Sing(\overrightarrow{V})$ can be partitioned into two disjoint subsets, $Sing(\overrightarrow{V})=S_{1}\cup S_{2}$, where
 \begin{enumerate}
 \item $S_{1}=\{p\in Sing(\overrightarrow{V}) ~|~ \Omega_{p}=\emptyset\}$, and if $p\in S_{1}$, then $p$ is an extremal point of $\Phi$.
 \item $S_{2}=\{p\in Sing(\overrightarrow{V}) ~|~ \Omega_{p}\neq\emptyset\}$, and if $p\in S_{2}$, then $p$ is a saddle point of $\Phi$ and at $p$ the angle of ${\rm d}s^{2}$ is $\pi\cdot|\Omega_{p}|$.
 \end{enumerate}
\end{proposition}

\begin{proof}~
Case (1) has been established based on \textbf{Proposition \ref{Pro-2-3}}. We now proceed to establish Case (2).  Without loss of generality, at a saddle point $p$ of $\Phi$, we can suppose $\omega=z^{n}{\rm d}z,n\in \mathbb{Z}^{+}$, where $z$ is a complex coordinate in a neighborhood of $p$ without any other singularities of $\overrightarrow{V}$, then $Y=\frac{\sqrt{-1}}{z^{n}}\frac{\partial}{\partial z}$. By setting $z=re^{\sqrt{-1}\theta}$,  we can simplify the expression for $Y$ as follows
\begin{equation*}
\begin{aligned}
Y&=\frac{\sqrt{-1}}{r^{n}e^{\sqrt{-1}n\theta}}( \frac{\partial r}{\partial z}\frac{\partial}{\partial r}+\frac{\partial \theta}{\partial z}\frac{\partial}{\partial\theta})=\frac{\sqrt{-1}}{2r^{n}e^{\sqrt{-1}(n+1)\theta}}(\frac{\partial}{\partial r}-\frac{\sqrt{-1}}{r}\frac{\partial}{\partial \theta})\\
&=\frac{1}{2r^{n}}[\sin(n+1)\theta\frac{\partial}{\partial r}+\frac{\cos(n+1)\theta}{r}\frac{\partial}{\partial\theta}+\sqrt{-1}(\cdots)].
\end{aligned}
\end{equation*}
Hence, the vector field $\overrightarrow{V}$ is given by
$$\overrightarrow{V}=\frac{\sin(n+1)\theta}{2r^{n}}\frac{\partial}{\partial r}+\frac{\cos(n+1)\theta}{2r^{n+1}}\frac{\partial}{\partial\theta}.$$

Suppose
\begin{equation*}
\begin{cases}
r=r(t),\\
\theta=\theta(t),
\end{cases}
t\in(-\varepsilon,\varepsilon)\setminus\{0\},\varepsilon>0
\end{equation*}
is an integral curve of $\overrightarrow{V}$  near the saddle point such that $\lim\limits_{t\rightarrow0}r(t)=0$. Then we obtain
\begin{equation}\label{P-2-E-1}
\begin{cases}
r'=\frac{\sin(n+1)\theta}{2r^{n}},\\
\theta'=\frac{\cos(n+1)\theta}{2r^{n+1}}.
\end{cases}
\end{equation}

When $\theta'\neq0$, the solution of equation (\ref{P-2-E-1})  is given by
$$r^{-(n+1)}=\lambda|\cos(n+1)\theta|,~~\lambda>0,$$
i.e.,
$$r=\frac{1}{\sqrt[n+1]{\lambda|\cos(n+1)\theta|}}\geq \frac{1}{\sqrt[n+1]{\lambda}}.$$
However, this curve does not meet $p$.

When $\theta'=0$, the solutions of equation (\ref{P-2-E-1})  are as follows
\begin{equation*}
\begin{cases}
r^{n+1}=-\frac{n+1}{2}t,\\
\theta=\frac{1}{n+1}[\frac{\pi}{2}+(2k+1)\pi],k=0,\ldots,n,
\end{cases}
t<0,
\end{equation*}
and
\begin{equation*}
\begin{cases}
r^{n+1}=\frac{n+1}{2}t,\\
\theta=\frac{1}{n+1}(\frac{\pi}{2}+2k\pi),k=0,\ldots,n,
\end{cases}
t>0.
\end{equation*}

Thus, case (2) is established.
\end{proof}

\begin{proposition}\label{Pro-2-5}
At a saddle point of the function $\Phi$, the angle between two neighboring integral curves of the gradient vector field $\nabla\Phi$ is $\pi$.
\end{proposition}

\begin{proof}~
According to the proof of \textbf{Proposition \ref{Pro-2-4}},  two neighboring integral curves of $\overrightarrow{V}$ at a saddle point of the function $\Phi$ is $\pi$. By \textbf{Proposition \ref{Pro-2-2}}, we know that $\nabla\Phi \perp \overrightarrow{V}$. Thus two neighboring integral curves of the gradient vector field $\nabla\Phi$ is $\pi$.
\end{proof}

According to \textbf{Proposition \ref{Pro-2-3}}, if $p$ is a local minimum point of $\Phi$, the integral curves of $\overrightarrow{V}$ form topologically concentric circles in a small neighborhood of $p$. Moreover, the integral curves of $\nabla \Phi$ are perpendicular to these circles. Let $c(t)~(t\in[0,T])$ be an integral curve of $\overrightarrow{V}$. For each $t$, there exists a unique integral curve $C_{t}$ (i.e., geodesic) of $\nabla \Phi$ originating from $p$ and intersecting the point $c(t)$. It is evident that $C_{t}$ must reach either a saddle point or a local maximum point of $\Phi$ (refer to Figure 2).

 \begin{figure}[htbp]
\centering
\includegraphics[scale=0.4]{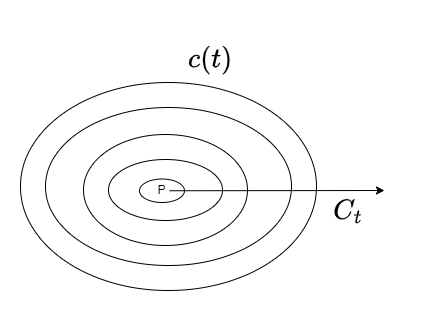}
\caption{\text{Integral curves} }
\end{figure}

\begin{proposition}\label{Pro-2-6}
Assuming that $C_{t}$ reaches the local maximum point $q$ directly without passing through any saddle point of $\Phi$, the arc length of $C_{t}$ is equal to $\pi$.
\end{proposition}
\begin{proof}~
Suppose $r:[0,l]\rightarrow \mathcal{M},s\mapsto r(s)$ is the parameter equation of curve $C_{t}$, where $s$ is the arc length parameter. By the proof of \textbf{Proposition \ref{Pro-2-1}}, we derive that
$$l=\int_{0}^{l}{\rm d}s=\int_{0}^{l}{\rm d}r=\int_{0}^{4}\frac{{\rm d}r}{{\rm d}\Phi}{\rm d}\Phi= \int_{0}^{4}\frac{1}{\sqrt{\Phi(4-\Phi)}}{\rm d}\Phi=\pi.$$
\end{proof}

From \textbf{Proposition  \ref{Pro-2-6}}, we obtain the following corollary.

\begin{corollary}\label{Co-2-1}
For any $t_{1}, t_{2} \in [0, T]$, if both $C_{t_{1}}$ and $C_{t_{2}}$ reach local maxima at points $q_{1}$ and $q_{2}$ respectively without traversing any saddle point of $\Phi$, then the geodesic distance between  $p$ and $q_{1}$ is equal to the geodesic distance between $p$ and $q_{2}$, both being $\pi$.
\end{corollary}

From \textbf{Proposition  \ref{Pro-2-6}} and the proof of \textbf{Proposition \ref{Pro-2-1}}, we derive the following property.

\begin{proposition}
The function $\Phi(r)$ monotonically increases along a geodesic connecting a minimum point to a maximum point of $\Phi$.
\end{proposition}

\begin{proposition}\label{Pro-2-7}
 Let $t_{0}\in (0,T)$ be fixed. Suppose that $C_{t_{0}}$ directly reaches a maximum point $q$ of $\Phi$. Then, there exists $\varepsilon>0$ such that for all $t\in(t_{0}-\varepsilon,t_{0}+\varepsilon)$, $C_{t}$ also reaches the same maximum point $q$ without passing through any saddle point of $\Phi$.
\end{proposition}
\begin{proof}~
Since the number of saddle points of $\Phi$ is finite, there exists a small neighborhood $(t_{0}-\varepsilon, t_{0}+\varepsilon)$ with $\varepsilon>0$ centered at $t_{0}$ where each curve $C_{t}$ for $t\in(t_{0}-\varepsilon, t_{0}+\varepsilon)$ does not intersect any saddle point of $\Phi$. The endpoints of $C_{t}$ vary continuously with $t$, and the maximum points of $\Phi$ are limited in number. Thus, there exists $\varepsilon>0$ such that for all $t\in(t_{0}-\varepsilon, t_{0}+\varepsilon)$, $C_{t}$ reaches the same maximum point.
\end{proof}

\section{Proof of the Theorem \ref{Main-th-1}}

By virtue of \textbf{Proposition \ref{Pro-2-3}}, the set $\bigcup_{t\in(t_{0}-\varepsilon, t_{0}+\varepsilon)}C_{t}$ forms a simply connected domain within $\mathcal{M}$. Let $F$ denote the largest simply connected domain in $\mathcal{M}$ that encompasses $\bigcup_{t\in(t_{0}-\varepsilon, t_{0}+\varepsilon)}C_{t}$ and satisfies the following three conditions:
\begin{enumerate}
\item Every curve traced by the vector field $\nabla\Phi$ within $F$ connects points $p$ and $q$.
\item  No curve traced by $\nabla \Phi$ within $F$ passes through any saddle point of $\Phi$.
\item The boundary of $F$ consists of curves traced by $\nabla\Phi$ that pass through some saddle points of $\Phi$.
\end{enumerate}

The boundary of $F$ can be divided into two curves, denoted as $\gamma_{1}$ and $\gamma_{2}$, connecting points $p$ and $q$. Along $\gamma_{1}$, there exist saddle points connected by geodesic segments (integral curves of $\nabla\Phi$). At each saddle point, according to \textbf{Proposition \ref{Pro-2-3}}, the included angle between two adjacent geodesics is $\pi$, ensuring the smoothness of $\gamma_{1}$ at each saddle point. Consequently, $\gamma_{1}$ and $\gamma_{2}$ are both smooth geodesics.

Evidently, $F$ can be characterized as a bigon, and the metric ${\rm d}s^{2}$ represents a conformal spherical metric on $F$. Thus, the pair $(F,{\rm d}s^{2})$ is isometric to a piece obtained by cutting a football along a geodesic that links the two conical singularities.

At a minimum point $p$ of $\Phi$, since there are a finite number of integral curves of $\nabla\Phi$ originating from $p$ and reaching saddle points, we can iteratively apply the above argument to obtain finite pieces of the largest simply connected domains in $\mathcal{M}$, each of them is isometric to a piece obtained by cutting a football along a geodesic that links the two conical singularities and contains $p$ as a vertex. This process can be repeated at each minimum point of $\Phi$ to yield a finite collection of pieces. We assert that every point in $\mathcal{M}$ lies within the union of these pieces. For any point in $\mathcal{M}$, there exists an integral curve of $\nabla\Phi$ that either starts from it, ends at it, or passes through it. If saddle points exist along this curve, the point must lie on the boundary of a largest domain; otherwise, it must lie within a largest domain. Consequently, every point in $\mathcal{M}$ is encompassed by the union of these pieces, and their uniqueness is evident.

\section{Geometric construction of spherical conical metrics on $S^{2}_{\{\alpha,\alpha\}}$}

\textbf{Theorem \ref{Main-th-1}} implies that reducible spherical conical metrics can be constructed on a compact Riemann surface using some suitable footballs. To illustrate, we demonstrate the construction of reducible spherical conical metrics on $S^{2}_{\{\alpha,\alpha\}}, 2\leq\alpha\in\mathbb{Z}^{+}$, which differ from a football.\par

 Suppose $S^{2}_{\{\alpha,\alpha\}}$ has a (reducible) spherical conical metric which is not a football, then the function $\Phi$ defined by equation (\ref{E-1}) has one or two saddle points. \par

 When $\Phi$ has only one saddle point, we proceed as follows: Take $\alpha$ standard footballs $S^{2}_{\{1,1\}}$ and cut each $S^{2}_{\{1,1\}}$ along the meridian (i.e. geodesic) from an extremal point $m$ of $\Phi$ (defined on $S^{2}_{\{1,1\}}$) to a point $s$. Ensure the distance between $s$ and $m$ satisfies $dis(s,m)=\lambda<\pi$. Then, gluing these arcs together (as shown in Figure 3), we obtain a reducible spherical conical metric on $S^{2}_{\{\alpha,\alpha\}}$. In this metric, the singularity at $s$ acts as a saddle point of $\widetilde{\Phi}$ with an angle of $2\pi\alpha$, while point $m$ serves as an extremal point of $\widetilde{\Phi}$ (defined on $S^{2}_{\{\alpha,\alpha\}}$) with an angle of $2\pi\alpha$.

\begin{figure}[htbp]
\centering
\includegraphics[width=12cm]{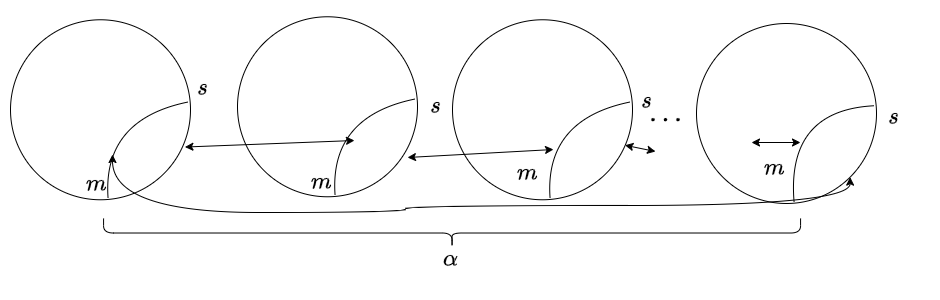}
\caption{$\Phi~has ~only~one~ saddle~point.$}
\end{figure}

 When $\Phi$ has two saddle points, we proceed as follows: Take $\alpha$ standard footballs $S^{2}_{\{1,1\}}$ and cut each $S^{2}_{\{1,1\}}$ along the meridian (i.e. geodesic) from a point $s_{1}$ to a point $s_{2}$ (both are not the extremal point of $\Phi$ defined on $S^{2}_{\{1,1\}}$). Ensure the distance between $s_{1}$ and $s_{2}$ satisfying $dis(s_{1},s_{2})=\lambda<\pi$.  Then, gluing these arcs together (as shown in Figure 4), we obtain a reducible spherical conical metric on $S^{2}_{\{\alpha,\alpha\}}$. In this metric, the singularities at $s_{1}$ and $s_{2}$ act as saddle points of $\widetilde{\Phi}$ (defined on $S^{2}_{\{\alpha,\alpha\}}$) with angles of $2\pi\alpha$ and $2\pi\alpha$.

\begin{figure}[htbp]
\centering
\includegraphics[width=12cm]{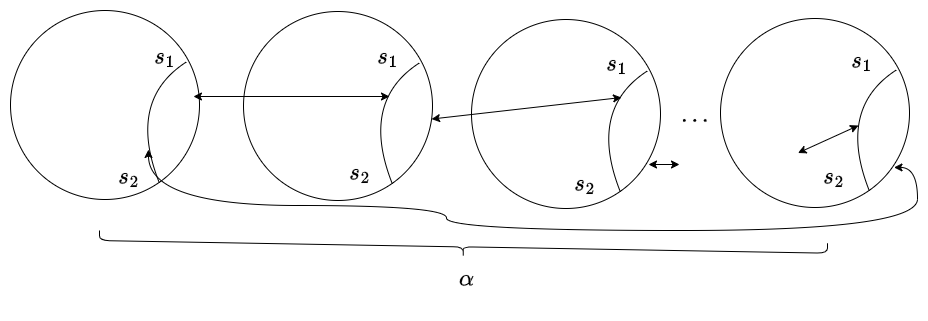}
\caption{$\Phi~has~two~saddle~ points.$}
\end{figure}

Although we have two methods for constructing a reducible spherical conical metric on $S^{2}_{\{\alpha,\alpha\}}$, these metrics are identical if the distance between the two singularities is equal. Through these constructions, we have proven the following theorem.

\begin{theorem}
Two spherical conical metrics on $S^{2}_{\{\alpha,\alpha\}}$ are identical if and only if the distances between their respective conical singularities are equal.
\end{theorem}

\section{Examples of reducible spherical cone metrics on $S^{2}$}

This section focuses on constructing some explicit reducible spherical conical metrics on $S^{2}$. To facilitate understanding, we introduce the following identity $$S^{2}_{\{\alpha_{1},\ldots,\alpha_{I}\}} = S^{2}_{\{\beta_{1},\ldots,\beta_{J}\}} + S^{2}_{\{\gamma,\gamma\}}.$$
This identity illustrates that a sphere $S^{2}_{\{\alpha_{1},\ldots,\alpha_{I}\}}$ with a reducible spherical conical metric can be decomposed into two components. The first component is a sphere $S^{2}_{\{\beta_{1},\ldots,\beta_{J}\}}$ with a reducible spherical conical metric. The second component is a football $S^{2}_{\{\gamma,\gamma\}}$. Alternatively, this identity also demonstrates that by attaching a football $S^{2}_{\{\gamma,\gamma\}}$ to a sphere $S^{2}_{\{\beta_{1},\ldots,\beta_{J}\}}$ with a reducible spherical conical metric, we can construct a reducible spherical conical metric on $S^{2}_{\{\alpha_{1},\ldots,\alpha_{I}\}}$.\par

\begin{figure}[htbp]
\centering
\includegraphics[width=9cm]{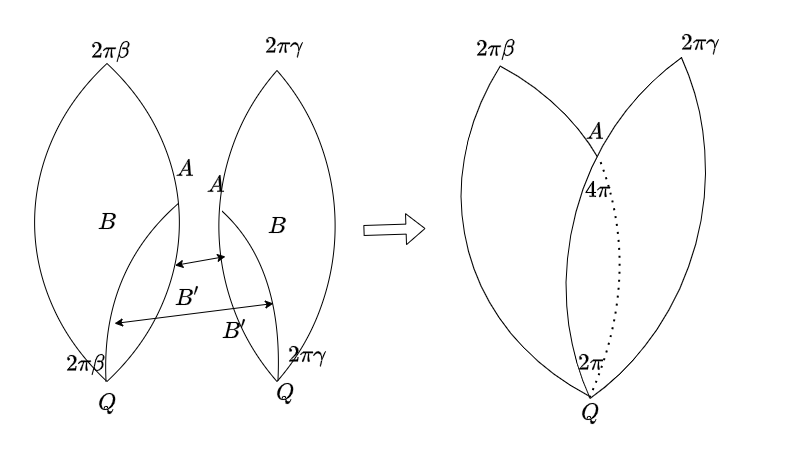}
\caption{$S^{2}_{\{2,\beta,\gamma\}},\beta+\gamma=1.$}
\end{figure}
\begin{example}
Construct a reducible spherical conical metric on $S^{2}_{\{2,\beta,\gamma\}}$ with the condition that $\beta+\gamma=1$.\par

Take two footballs, denoted as $S^{2}_{\{\beta,\beta\}}$ and $S^{2}_{\{\gamma,\gamma\}}$. Cut $S^{2}_{\{\beta,\beta\}}$ along the meridian from point $A$ to point $Q$, where $Q$ is an extremal point of $\widehat{\Phi}$ (defined on $S^{2}_{\{\beta,\beta\}}$), but $A$ is not. The arc $\widehat{AQ}$ splits into two identical arcs, $B$ and $B'$. Similarly, cut $S^{2}_{\{\gamma,\gamma\}}$ along the meridian to obtain arcs $B$ and $B'$. Join arc $B$ from $S^{2}_{\{\beta,\beta\}}$ with arc $B$ from $S^{2}_{\{\gamma,\gamma\}}$, and do the same for arcs $B'$ (as shown in Figure 5). Since $\beta+\gamma=1$, we construct a reducible spherical conical metric on $S^{2}_{\{2,\beta,\gamma\}}$, where singularity at $A$ with an angle of $4\pi$ acts as the saddle point of $\Phi$ (defined on $S^{2}_{\{2,\beta,\gamma\}}$). Thus
$$S^{2}_{\{2,\beta,\gamma\}}= S^{2}_{\{\beta,\beta\}}+S^{2}_{\{\gamma,\gamma\}}.$$

Obviously, this method can be generalized to construct a reducible spherical conical metric on $S^{2}_{\{2,\beta,\gamma,\beta+\gamma\}}$ for any values of $\beta$ and $\gamma$ such that $\beta+\gamma\neq1$.

\end{example}

\begin{example}
Construct a reducible spherical conical metric on $S^{2}_{\{3,\beta,\gamma\}}$ with the condition that $\beta+\gamma=2$.\par
\begin{figure}[htbp]
\centering
\includegraphics[width=10cm]{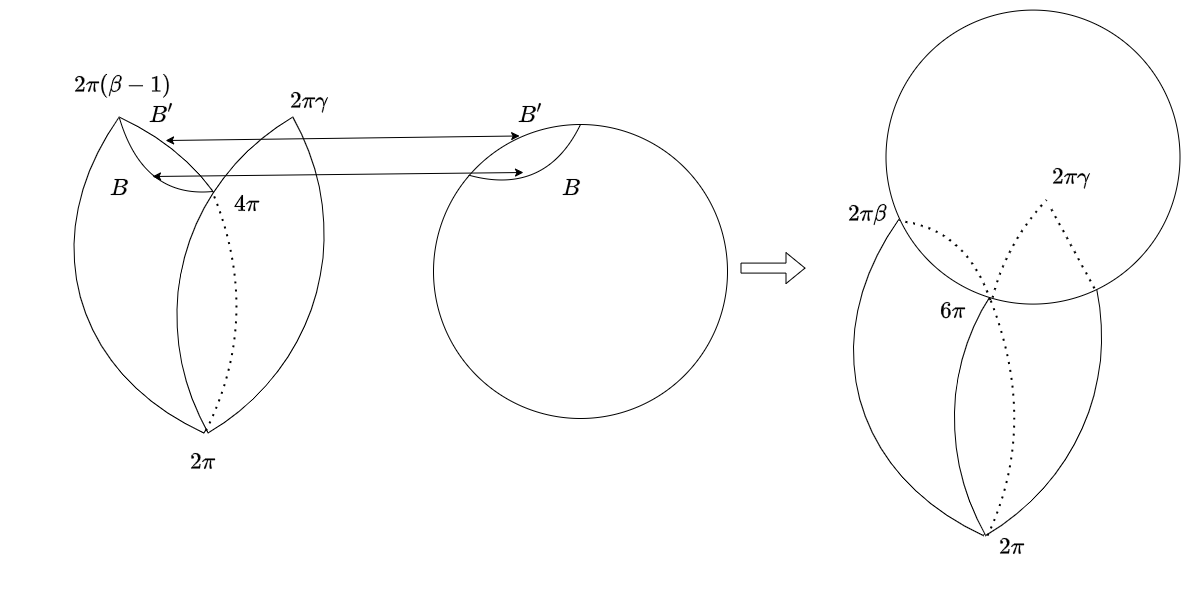}
\caption{$S^{2}_{\{3,\beta,\gamma\}},\beta+\gamma=2.$}
\end{figure}
Since $\beta+\gamma=2$, we can choose values for $\beta$ and $\gamma$ such that $1<\beta<2$ and $0<\gamma<1$. According to \textbf{Example 7.1}, there exists a reducible spherical conical metric on $S^{2}_{\{2,\beta-1,\gamma\}}$, since $(\beta-1)+\gamma=1$. Cut along the geodesic on $S^{2}_{\{2,\beta-1,\gamma\}}$ that connects the singularities with corresponding angles of $4\pi$ and $2\pi(\beta-1)$, respectively. We obtain two arcs $B$ and $B'$, each with a length less than $\pi$. Take a standard football $S^{2}_{\{1,1\}}$ and cut it along a meridian such that the resulting arcs from one extremal point to a point have lengths equal to those of arcs $B$ and $B'$ on $S^{2}_{\{2,\beta-1,\gamma\}}$. Join arc $B$ from $S^{2}_{\{2,\beta-1,\gamma\}}$ with arc $B$ from $S^{2}_{\{1,1\}}$, and do the same for arcs $B'$ (as shown in Figure 6). Then, we construct a reducible spherical conical metric on $S^{2}_{\{3,\beta,\gamma\}}$. The resulting metric can be expressed as
$$S^{2}_{\{3,\beta,\gamma\}}=S^{2}_{\{2,\beta-1,\gamma\}}+S^{2}_{\{1,1\}}.$$
\end{example}

\begin{example}
Construct a reducible spherical conical metric on $S^{2}_{\{3,2,2\}}$ such that the function $\Phi$ defined by equation (\ref{E-1}) has 1,2,3 saddle points, respectively.\par
(1)~A reducible spherical conical metric on $S^{2}_{\{3,2,2\}}$ such that $\Phi$ has 1 saddle point.\par

\begin{figure}[htbp]
\centering
\includegraphics[width=10cm]{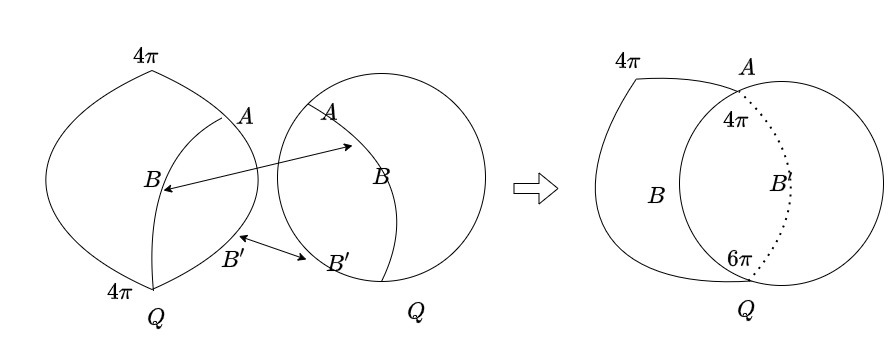}
\caption{$\Phi~ has~ one~ saddle~ point.$}
\end{figure}
\textbf{Method one}~~Consider two footballs, denoted as $S^{2}_{\{2,2\}}$ and $S^{2}_{\{1,1\}}$. Cut $S^{2}_{\{2,2\}}$ along the meridian from a singularity point $Q$ to point $A$ such that the distance between $A$ and $Q$ is least than $\pi$. The segment $\widehat{AQ}$ divides into two identical arcs, labeled as $B$ and $B'$. Similarly, cut $S^{2}_{\{1,1\}}$ along the meridian to obtain arcs $B$ and $B'$. Next, glue arc $B$ from $S^{2}_{\{1,1\}}$ to $S^{2}_{\{2,2\}}$ and do the same for arcs $B'$ (as shown in Figure 7). Then, we construct a reducible spherical conical metric on $S^{2}_{\{3,2,2\}}$. In this metric, one singularity with an angle of $4\pi$ serves as the unique saddle point of $\Phi$. The resulting metric can be expressed as
$$S^{2}_{\{3,2,2\}}=S^{2}_{\{2,2\}}+S^{2}_{\{1,1\}}.$$

\begin{figure}[htbp]
\centering
\includegraphics[width=10cm]{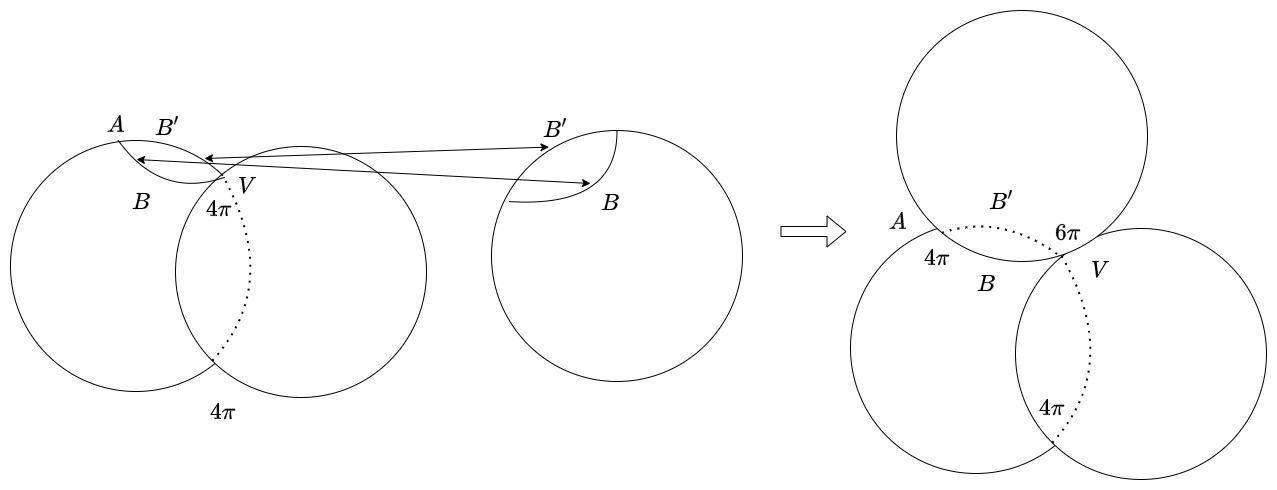}
\caption{$\Phi~ has~ one~ saddle~ point.$}
\end{figure}
\textbf{Method two}~~Consider a standard football $S^{2}_{\{1,1\}}$ and a reducible spherical conical metric on $S^{2}_{\{2,2\}}$ such that one singularity is an extremal point  of the function $\widehat{\Phi}$ (defined on $S^{2}_{\{2,2\}}$) and another singularity is a saddle point (denoted by $V$) of $\widehat{\Phi}$. Denote another extremal point of $\widehat{\Phi}$ by $A$. Cut $S^{2}_{\{2,2\}}$ along a geodesic from $A$ to $V$. We obtain two identical arcs, $B$ and $B'$. Similarly, cut the standard football $S^{2}_{\{1,1\}}$ along the meridian yields arcs $B$ and $B'$. Next, glue arc $B$ form $S^{2}_{\{2,2\}}$ to $S^{2}_{\{1,1\}}$ and do the same for arcs $B'$ (as shown in Figure 8). Then, we construct a reducible spherical conical metric on $S^{2}_{\{3,2,2\}}$, where the singularity of angle $6\pi$ serves as the unique saddle point of $\Phi$. The resulting metric can be expressed as
$$S^{2}_{\{3,2,2\}}=S^{2}_{\{2,2\}}+S^{2}_{\{1,1\}}.$$

(2)~A reducible spherical conical metric on $S^{2}_{\{3,2,2\}}$ such that $\Phi$ has 2 saddle points.\par

\begin{figure}[htbp]
\centering
\includegraphics[width=12cm]{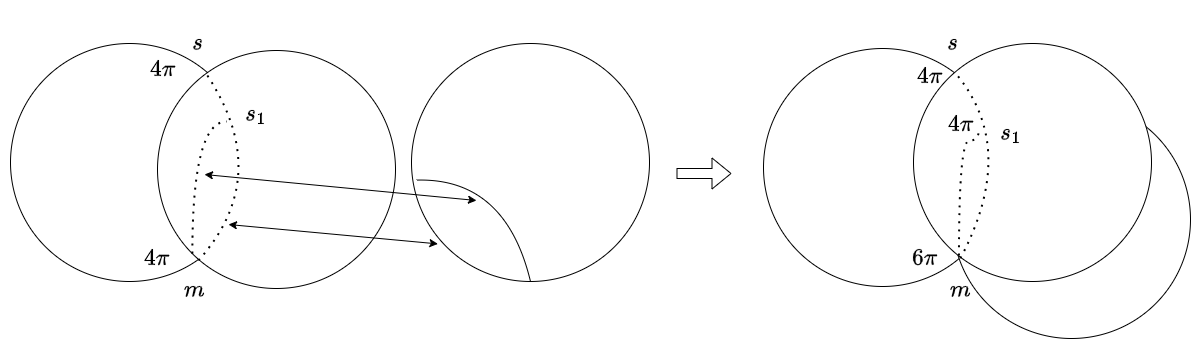}
\caption{$\Phi~ has~ two~ saddle~ points.$}
\end{figure}
\textbf{Method one}~~Consider a standard football $S^{2}_{\{1,1\}}$ and a reducible spherical conical metric on $S^{2}_{\{2,2\}}$ such that two singularities are extremal point (denoted by $m$) and saddle point (denoted by $s$) of the function $\widehat{\Phi}$ (defined on $S^{2}_{\{2,2\}}$), respectively.  Choose a new point (denoted by $s_{1}$) on the geodesic connected the point $m$ and its antipodal point such that $s$ and $s_{1}$ lie on this geodesic. Cut $S^{2}_{\{2,2\}}$ along this geodesic from $m$ to $s_{1}$. We obtain two identical arcs, $B$ and $B'$. Similarly, cut the standard football $S^{2}_{\{1,1\}}$ along the meridian yields arcs $B$ and $B'$. Next, glue $B$ in $S^{2}_{\{2,2\}}$ with $B$ in $S^{2}_{\{1,1\}}$ and do the same for arcs $B'$ (as shown in Figure 9). Then, we construct a reducible spherical conical metric on $S^{2}_{\{3,2,2\}}$, where the two singularities of angles $4\pi$ and $4\pi$ are saddle points of $\Phi$. The resulting metric can be expressed as
$$S^{2}_{\{3,2,2\}}=S^{2}_{\{2,2\}}+S^{2}_{\{1,1\}}.$$

\begin{figure}[htbp]
\centering
\includegraphics[width=12cm]{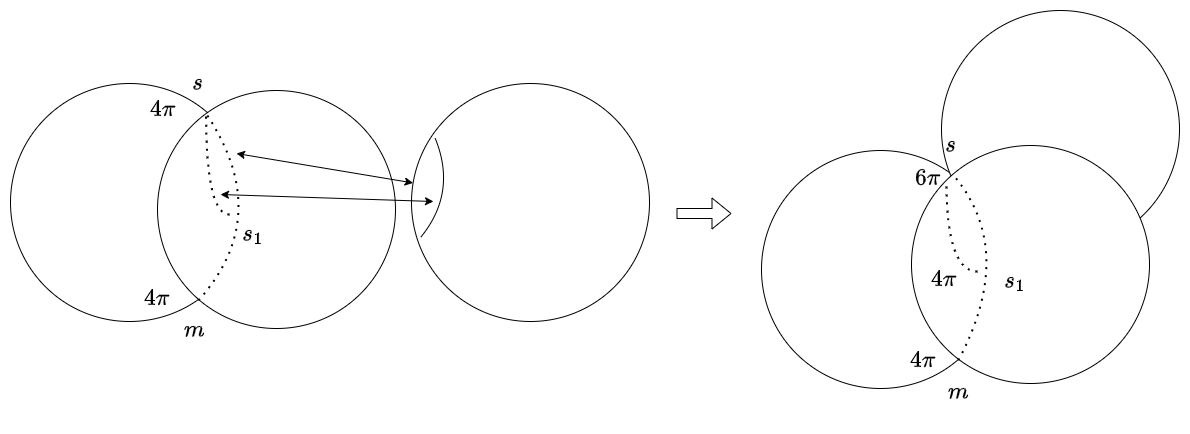}
\caption{$\Phi~ has~ two~ saddle~ points.$}
\end{figure}
\textbf{Method two}~~Consider a standard football $S^{2}_{\{1,1\}}$ and a reducible spherical conical metric on $S^{2}_{\{2,2\}}$ such that two singularities are extremal point (denoted by $m$) and saddle point (denoted by $s$) of the function $\widehat{\Phi}$ (defined on $S^{2}_{\{2,2\}}$).  Choose a point (denoted by $s_{1}$) on the geodesic connected the point $m$ and its antipodal point such that $s$ and $s_{1}$ lie on this geodesic. Cut $S^{2}_{\{2,2\}}$ along this geodesic from $s$ to $s_{1}$. We obtain two identical arcs, $B$ and $B'$. Similarly, cut the standard football $S^{2}_{\{1,1\}}$ along the meridian yields arcs $B$ and $B'$. Next, glue $B$ in $S^{2}_{\{2,2\}}$ with $B$ in $S^{2}_{\{1,1\}}$ and do the same for arcs $B'$ (as shown in Figure 10). Then, we construct a reducible spherical conical metric on $S^{2}_{\{3,2,2\}}$, where the two singularities of angles $6\pi$ and $4\pi$ are two saddle points of $\Phi$. The resulting metric can be expressed as
$$S^{2}_{\{2,2,3\}}=S^{2}_{\{2,2\}}+S^{2}_{\{1,1\}}.$$

(3)~A reducible spherical conical metric on $S^{2}_{\{3,2,2\}}$ such that $\Phi$ has 3 saddle points.\par

\begin{figure}[htbp]
\centering
\includegraphics[width=12cm]{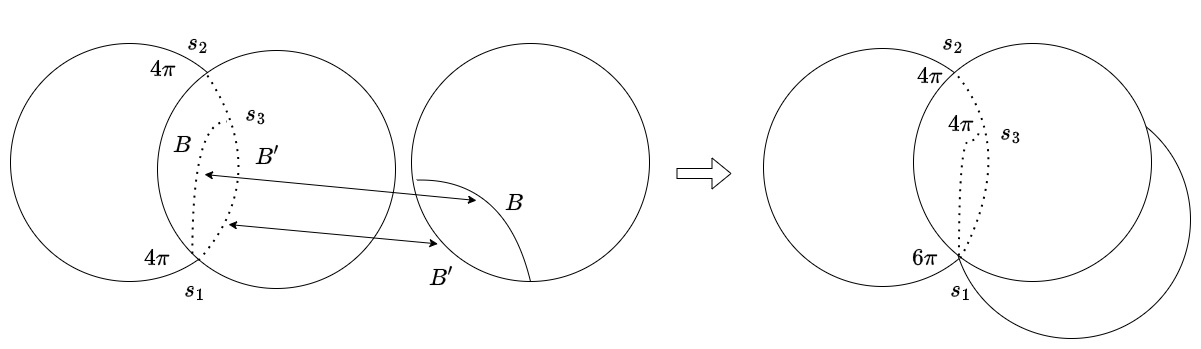}
\caption{$\Phi~ has~ three~ saddle~ points.$}
\end{figure}
~~Consider a standard football $S^{2}_{\{1,1\}}$ and a reducible spherical conical metric on $S^{2}_{\{2,2\}}$ such that two singularities are both saddle points (denoted by $s_{1}$ and $s_{2}$ respectively) of the function $\widehat{\Phi}$ (defined on $S^{2}_{\{2,2\}}$). Choose a new point (denoted by $s_{3}$) on the geodesic which connects points $s_{1}$ and $s_{2}$. Cut $S^{2}_{\{2,2\}}$ along this geodesic from $s_{1}$ to $s_{3}$. We obtain two identical arcs, $B$ and $B'$. Similarly, cut the standard football $S^{2}_{\{1,1\}}$ along the meridian yields arcs $B$ and $B'$. Next, glue $B$ in $S^{2}_{\{2,2\}}$ with $B$ in $S^{2}_{\{1,1\}}$ and do the same for arcs $B'$ (as shown in Figure 11). Then, we construct a reducible spherical conical metric on $S^{2}_{\{3,2,2\}}$, where three singularities are all saddle points of $\Phi$. The resulting metric can be expressed as
$$S^{2}_{\{3,2,2\}}=S^{2}_{\{2,2\}}+S^{2}_{\{1,1\}}.$$
\end{example}

\begin{remark}
In the construction of \textbf{Example 7.3 (3)}, the 3 saddles lie on the same geodesic connecting a minimum point and a maximum point of $\Phi$. Nevertheless, the specific locations of these three saddles can vary.
\end{remark}

\begin{example}
Figure 12 shows some constructions of reducible spherical conical metrics on $S^{2}$, where the function $\Phi$, defined by equation (\ref{E-1}) has 2 saddle points, each with an angle of $4\pi$. In cases $A$, $B'$, and $C$, these saddle points are located on a common geodesic that directly links a minimum point to a maximum point of $\Phi$ on the sphere. These constructions provide more concrete examples for understanding reducible spherical conical metrics.\par

\begin{figure}[htbp]
\centering
\includegraphics[width=13cm]{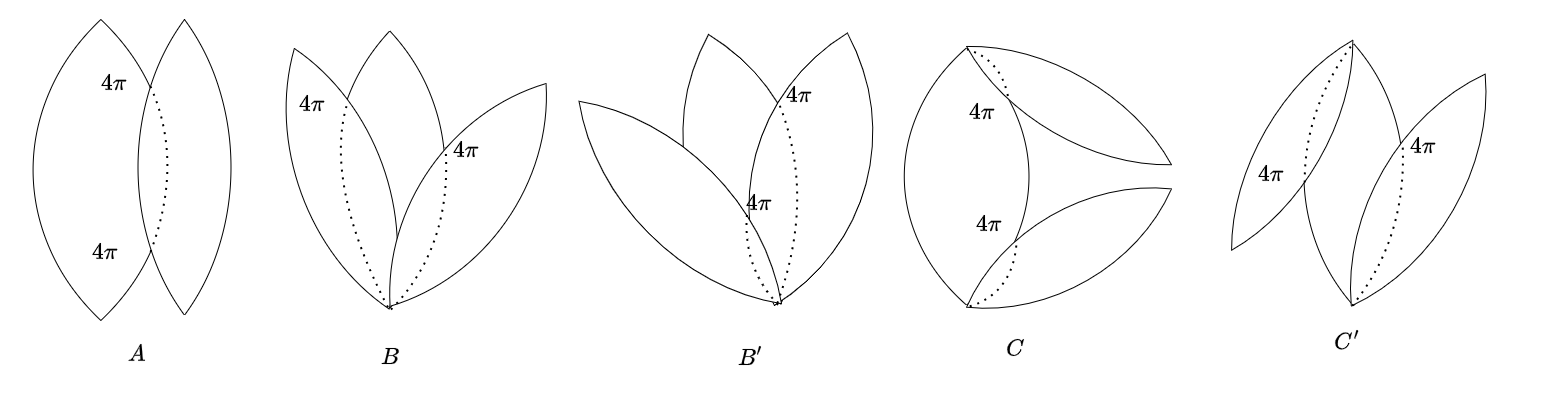}
\caption{$\Phi~has~ two~ saddle~ points.$}
\end{figure}

\begin{enumerate}
\item A reducible spherical conical metric on $S^{2}_{\{2,2,\beta,\beta\}},\beta\neq1$ can be constructed by case $A$, i.e.
$$S^{2}_{\{2,2,\beta,\beta\}}=S^{2}_{\{1,1\}}+S^{2}_{\{\beta,\beta\}}.$$

\item A reducible spherical conical metric on $S^{2}_{\{2,2,\beta,\beta+2\}}(\beta\neq1)$ can be  constructed by cases $B$ or $B'$, i.e.
$$S^{2}_{\{2,2,\beta,\beta+2\}}=S^{2}_{\{1,1\}}+S^{2}_{\{\beta,\beta\}}+S^{2}_{\{1,1\}}.$$

\item A reducible spherical conical metric on $S^{2}_{\{2,2,\beta_{1},\beta_{2},\beta_{1}+\beta_{2}+1\}}(\beta_{1},\beta_{2}\neq1)$ can be  constructed by the cases $B$ or $B'$, i.e.
$$S^{2}_{\{2,2,\beta_{1},\beta_{2},\beta_{1}+\beta_{2}+1\}}=S^{2}_{\{\beta_{1},\beta_{1}\}}+S^{2}_{\{\beta_{2},\beta_{2}\}}+S^{2}_{\{1,1\}}.$$

\item A reducible spherical conical metric on $S^{2}_{\{2,2,\beta_{1},\beta_{2},\gamma_{1}\}}(\beta_{1},\beta_{2},\gamma_{1}\neq1,\beta_{1}<1,\beta_{1}+\beta_{2}=1+\gamma_{1})$ can be  constructed by cases $C$ or $C'$, i.e.
$$S^{2}_{\{2,2,\beta_{1},\beta_{2},\gamma_{1}\}}=S^{2}_{\{\beta_{2}-\gamma_{1},\beta_{2}-\gamma_{1}\}}+S^{2}_{\{\beta_{1},\beta_{1}\}}+S^{2}_{\{\gamma_{1},\gamma_{1}\}}.$$

\item A reducible spherical conical metric on $S^{2}_{\{2,2,\beta_{1},\beta_{2},\gamma_{1},\gamma_{2}\}}(\beta_{1},\beta_{2},\gamma_{1},\gamma_{2}\neq1,\beta_{1}\leq\beta_{2},\gamma_{1}\leq\gamma_{2}, \beta_{1}+\beta_{2}=\gamma_{1}+\gamma_{2})$ can be construct as follows. \par
\begin{enumerate}
\item If $\beta_{1}=\gamma_{1}=\beta_{2}=\gamma_{2}=\beta$,
this metric can be constructed by the case $A$, i.e.
$$S^{2}_{\{2,2,\beta,\beta,\beta,\beta\}}=S^{2}_{\{\beta,\beta\}}+S^{2}_{\{\beta,\beta\}}.$$

\item If $\beta_{1}=\gamma_{1}<\beta_{2}=\gamma_{2}$,
this metric can be constructed by cases $C$ or $C'$, i.e.
$$S^{2}_{\{2,2,\beta_{1},\beta_{2},\beta_{1},\beta_{2}\}}=S^{2}_{\{\beta_{2}-\beta_{1},\beta_{2}-\beta_{1}\}}+S^{2}_{\{\beta_{1},\beta_{1}\}}+S^{2}_{\{\beta_{1},\beta_{1}\}}.$$
or by case $A$, i.e.
$$S^{2}_{\{2,2,\beta_{1},\beta_{2},\beta_{1},\beta_{2}\}}=S^{2}_{\{\beta_{1},\beta_{1}\}}+S^{2}_{\{\beta_{2},\beta_{2}\}}.$$

\item If $\beta_{1}<\gamma_{1}$, then $\beta_{2}-\gamma_{2}=\gamma_{1}-\beta_{1}>0$, this metric can be constructed by cases $C$ or $C'$, i.e.
$$S^{2}_{\{2,2,\beta_{1},\beta_{2},\gamma_{1},\gamma_{2}\}}=S^{2}_{\{\beta_{2}-\gamma_{2},\beta_{2}-\gamma_{2}\}}+S^{2}_{\{\gamma_{2},\gamma_{2}\}}+S^{2}_{\{\beta_{1},\beta_{1}\}}.$$
\end{enumerate}
\end{enumerate}
\end{example}

\textbf{Example 7.4} implies that some reducible spherical conical metrics on $S^{2}$ can be constructed by using some suitable footballs such that some singularities lie on the same geodesic connecting two extremal points of $\Phi$ with length $\pi$. We will utilize this observation to demonstrate the existence of reducible spherical conical metrics.

\section{Existence of reducible spherical conical metrics on  $S^{2}_{\{\alpha,\beta,\gamma\}}$}

In this section, we will apply \textbf{Theorems \ref{Thm-1}} and \textbf{\ref{Main-th-1}} to establish the necessary and sufficient conditions for the existence of reducible spherical conical metrics on $S^{2}_{\{\alpha,\beta,\gamma\}}$. We will present detailed proofs of these conditions in the following subsections. Notably, this subject has been previously explored by Umehara and Yamada \cite{UY00}, as well as Eremenko \cite{Er04}. \par

Let ${\rm d}s^{2}$ be a reducible spherical conical metric on $S^{2}_{\{\alpha,\beta,\gamma\}}$. Let $\omega$ be a character 1-form of ${\rm d}s^{2}$. Since $\omega$ is an Abelian differential on the Riemann sphere $S^{2}$, according to the Riemann-Roch theorem and \textbf{Theorem \ref{Thm-1}} , we can deduce the following statements.
\begin{enumerate}
\item  A single order among $\alpha,\beta$ and $\gamma$ is an integer.
\item  All of the orders $\alpha,\beta$ and $\gamma$ are integers.
\end{enumerate}

When only one of order among $\alpha,\beta$ and $\gamma$ is an integer, $\omega$ has a unique zero, which is the singularity corresponding to the integer order, while the other singularities are simple poles of $\omega$. Therefore, we only need to consider the signs of the residues of $\omega$ at these two singularities. By examining various cases, we derive the following theorem.\par

\begin{theorem}\label{Main-th-3-0}
Let $\alpha$ be an integer with $\alpha \geq 2$, and let $\beta$ and $\gamma$ be non-integer  positive real numbers. There exists a (reducible) spherical conical metric on $S^{2}_{\{\alpha,\beta,\gamma\}}$ if and only if one of the following conditions  is satisfied.
\begin{enumerate}
\item Both $\alpha-\beta-\gamma-1$ and $\alpha+\beta+\gamma-1$ are nonnegative even integers, which is equivalent to $\alpha-\beta-\gamma$ and $\alpha+\beta+\gamma$ being positive odd integers.
\item Both $\alpha+\beta-\gamma-1$ and $\alpha-\beta+\gamma-1$ are nonnegative even integers, which is equivalent to $\alpha+\beta-\gamma$ and $\alpha-\beta+\gamma$ being positive odd integers.
\end{enumerate}
\end{theorem}

When all orders $\alpha,\beta$ and $\gamma$ are integers, $\omega$ may have one, two or three zeros among the singularities. By examining various cases, we can derive the following three theorems.

\begin{theorem}\label{Main-th-3-1}
Let $\alpha,\beta$ and $\gamma$ be integers, each greater than or equal to 2.  There exists a (reducible) spherical conical metric on $S^{2}_{\{\alpha,\beta,\gamma\}}$ such that the conical singularity of angle $2\pi\alpha$ is the unique saddle point of $\Phi$, as defined by equation (\ref{E-1}), if and only if the following conditions are met:
\begin{enumerate}
\item Both $\alpha-\beta+\gamma-1$ and $\alpha+\beta-\gamma-1$ are nonnegative even integers.
\item The inequality $\alpha+\beta+\gamma-1>2(\alpha-1)$ is satisfied, which is equivalent to the condition $\beta+\gamma>\alpha-1$.
\end{enumerate}
\end{theorem}

\begin{theorem}\label{Main-th-3-2}
Let $\alpha,\beta$ and $\gamma$ be integers, each greater than or equal to 2. There exists a (reducible) spherical conical metric on $S^{2}_{\{\alpha,\beta,\gamma\}}$ such that the conical singularities of angles $2\pi\alpha,2\pi\beta$ are all saddle points of $\Phi$, as defined by equation (\ref{E-1}), if and only if the following conditions are met:
\begin{enumerate}
\item Both $\alpha+\beta-\gamma-1$ and $\alpha+\beta+\gamma-1$ are nonnegative even integers.
\item The inequality $\alpha+\beta+\gamma-1> 2\cdot \max\{\alpha-1,\beta-1\}$ is satisfied.
\end{enumerate}
\end{theorem}

\begin{theorem}\label{Main-th-3-3}
Let $\alpha,\beta$ and $\gamma$ be integers, each greater than or equal to 2. There exists a (reducible) spherical conical metric on $S^{2}_{\{\alpha,\beta,\gamma\}}$ such that all conical singularities are saddle points of $\Phi$, as defined by equation (\ref{E-1}), if and only if the following conditions are met:
 \begin{enumerate}
\item The number $\alpha+\beta+\gamma-1$ is an even integer.
\item The inequality $\alpha+\beta+\gamma-1>2\cdot \max\{\alpha-1,\beta-1,\gamma-1\}$ is satisfied.
\end{enumerate}
\end{theorem}

\begin{remark}
These results can be integrated into \textbf{Theorem \ref{Main-th-4}} by considering the specific condition where  $I+J+L=3$ as stated in \textbf{Theorem \ref{Main-th-4}}.
\end{remark}

In this section, we will utilize the following identity to illustrate the construction of reducible spherical conical metrics:
$$S^{2}_{\{\alpha_{1},\ldots,a_{N_{1}}\}}=S^{2}_{\{\beta_{1},\ldots,\beta_{N_{2}}\}}+S^{2}_{\{\gamma_{1},\gamma_{1}\}}$$
This identity shows that by attaching a football $S^{2}_{\{\gamma_{1},\gamma_{1}\}}$ to $S^{2}_{\{\beta_{1},\ldots,\beta_{N_{2}}\}}$, which possesses a reducible spherical conical metric, we can construct a reducible spherical conical metric on $S^{2}_{\{\alpha_{1},\ldots,a_{N_{1}}\}}$.\par

We will demonstrate the necessity of these theorems using \textbf{Theorem} \ref{Thm-1}, and establish their sufficiency with \textbf{Theorem} \ref{Main-th-1}.

\subsection{Proof of \textbf{Theorem} \ref{Main-th-3-0}}
~~\par

\textbf{Necessity}~ \par

Let ${\rm d}s^{2}$ be a reducible spherical conical metric on $S^{2}_{\{\alpha,\beta,\gamma\}}$ satisfied the condition. Let $\omega$ be a character 1-form of ${\rm d}s^{2}$, and  denote by $p_{1},p_{2}$ and $p_{3}$ the conical singularities of ${\rm d}s^{2}$ with corresponding angles $2\pi\alpha,2\pi\beta$ and $2\pi\gamma$, respectively. It follows that $p_{1}$ is the unique zero of $\omega$ with order $\alpha-1$ and $p_{2},p_{3}$ are simple poles of $\omega$. The residues at $p_{2}$ and $p_{3}$ lead us to consider the following two distinct cases. \par
\textbf{(A)} The product of the residues $Res_{p_{2}}(\omega)\cdot Res_{p_{3}}(\omega)$ is greater than zero.\par
\textbf{(B)} The product of the residues $Res_{p_{2}}(\omega)\cdot Res_{p_{3}}(\omega)$ is less than zero.\par
When case \textbf{(A)} arises, without loss of generality, we  may assume that $Res_{p_{2}}(\omega)=\beta$ , $Res_{p_{3}}(\omega)=\gamma$ (otherwise, consider $-\omega$ instead).  Suppose that $\omega$ has $p$ poles with residues $+1$ and $q$ poles with residues $-1$. By applying the Residue theorem and the Riemann-Roch theorem, we obtain
\begin{equation*}
\begin{cases}
\beta+\gamma+p-q=0,\\
(\alpha-1)-p-q-2=-2.
\end{cases}
\end{equation*}

Solving these equations yields
\begin{equation*}
\begin{cases}
p=\frac{\alpha-\beta-\gamma-1}{2},\\
q=\frac{\alpha+\beta+\gamma-1}{2}.
\end{cases}
\end{equation*}

Since $p$ and $q$ are nonnegative integers, $\alpha-\beta-\gamma-1$ and $ \alpha+\beta+\gamma-1$ must also be nonnegative even integers.\par

When case \textbf{(B)} arises, without loss of generality, we may assume that $Res_{p_{2}}(\omega)=\beta$, $Res_{p_{3}}(\omega)=-\gamma$. Suppose that $\omega$ has $p$ poles with residues $+1$ and $q$ poles with residues $-1$. By applying the Residue theorem and the Riemann-Roch theorem, we obtain
\begin{equation*}
\begin{cases}
\beta-\gamma+p-q=0,\\
(\alpha-1)-p-q-2=-2.
\end{cases}
\end{equation*}

Solving these equations yields
\begin{equation*}
\begin{cases}
p=\frac{\alpha-\beta+\gamma-1}{2},\\
q=\frac{\alpha+\beta-\gamma-1}{2}.
\end{cases}
\end{equation*}

Since $p$ and $q$ are nonnegative integers, $\alpha+\beta-\gamma-1$ and $\alpha-\beta+\gamma-1$  must also be nonnegative even integers.\par

\textbf{Sufficiency}~\par
We shall now apply mathematical induction to the integer $\alpha$ to verify the sufficiency of the conditions. \par

 When $\alpha=2$, if $\alpha-\beta-\gamma-1$ and $\alpha+\beta+\gamma-1$ are nonnegative even integers, and since $0<\beta,\gamma$, we deduce that $\beta+\gamma=1$. According to \textbf{Example 7.1}, we derive the following identity
$$S^{2}_{\{2,\beta,\gamma\}}=S^{2}_{\{\beta,\beta\}}+S^{2}_{\{\gamma,\gamma\}}.$$
Thus, there exists a reducible spherical conical metric on $S^{2}_{\{2,\beta,\gamma\}}$ satisfied the condition.\par

If $\alpha+\beta-\gamma-1$ and $\alpha-\beta+\gamma-1$ are nonnegative even integers, without loss of generality, we may suppose $\beta\leq\gamma$, then we deduce that $\gamma=\beta+1$. According to \textbf{Example 7.1}, we derive the following identity
$$S^{2}_{\{2,\beta,\gamma\}}=S^{2}_{\{2,\beta,\beta+1\}}=S^{2}_{\{\beta,\beta\}}+S^{2}_{\{1,1\}}.$$
Thus, there exists a reducible spherical conical metric on $S^{2}_{\{2,\beta,\gamma\}}$ satisfied the condition.\par

Suppose that when $\alpha=k\geq 2$, if either $\alpha-\beta-\gamma-1$ and $\alpha+\beta+\gamma-1$ or $\alpha+\beta-\gamma-1$ and $\alpha-\beta+\gamma-1$ are nonnegative even integers, there exists a reducible spherical conical metric on $S^{2}_{\{\alpha,\beta,\gamma\}}=S^{2}_{\{k,\beta,\gamma\}}$ satisfied the condition. Now, when $\alpha=k+1$ and either $\alpha-\widetilde{\beta}-\widetilde{\gamma}-1$ and $\alpha+\widetilde{\beta}+\widetilde{\gamma}-1$ or $\alpha+\widetilde{\beta}-\widetilde{\gamma}-1$ and $\alpha-\widetilde{\beta}+\widetilde{\gamma}-1$ are nonnegative even integers, we aim to demonstrate the existence of a reducible conical spherical metric on $S^{2}_{\{k+1,\widetilde{\beta},\widetilde{\gamma}\}}$ satisfied the condition.\par

\textbf{Case A:~} Both $\alpha-\widetilde{\beta}-\widetilde{\gamma}-1=k-\widetilde{\beta}-\widetilde{\gamma}$ and $ \alpha+\widetilde{\beta}+\widetilde{\gamma}-1=k+\widetilde{\beta}+\widetilde{\gamma}$ are nonnegative even integers.\par

If $0<\widetilde{\beta}<1$, according to our assumption, there exists a reducible spherical conical metric on $S^{2}_{\{k, 1 - \widetilde{\beta}, \widetilde{\gamma}\}}$, since $k-1+(1-\widetilde{\beta})-\widetilde{\gamma}=k-\widetilde{\beta}-\widetilde{\gamma}$ and $k-1-(1-\widetilde{\beta})+\widetilde{\gamma}=k+\widetilde{\beta}+\widetilde{\gamma}- 2$ are nonnegative even integers. Since
$$S^{2}_{\{k+1, \widetilde{\beta}, \widetilde{\gamma}\}}=S^{2}_{\{k, 1-\widetilde{\beta}, \widetilde{\gamma}\}}+S^{2}_{\{\widetilde{\beta}, \widetilde{\beta}\}},$$
we obtain a reducible spherical conical metric on $S^{2}_{\{k + 1, \widetilde{\beta}, \widetilde{\gamma}\}}$ satisfied the condition.\par

If $1<\widetilde{\beta}$, according to our assumption, there exists a reducible spherical conical metric on $S^{2}_{\{k, \widetilde{\beta}-1, \widetilde{\gamma}\}}$, since $k-1- (\widetilde{\beta}-1)-\widetilde{\gamma}=k-\widetilde{\beta}-\widetilde{\gamma}$ and $k-1+(\widetilde{\beta}-1)+\widetilde{\gamma}=k+\widetilde{\beta}+\widetilde{\gamma}-2$ are nonnegative even integers. Since
$$S^{2}_{\{k+1, \widetilde{\beta}, \widetilde{\gamma}\}}=S^{2}_{\{k, \widetilde{\beta}-1, \widetilde{\gamma}\}}+S^{2}_{\{1, 1\}},$$
we obtain a reducible spherical conical metric on $S^{2}_{\{k+1, \widetilde{\beta}, \widetilde{\gamma}\}}$ satisfied the condition.\par

 \textbf{Case B:}~Both $\alpha+\widetilde{\beta}-\widetilde{\gamma}-1=k+\widetilde{\beta}-\widetilde{\gamma}$ and $ \alpha-\widetilde{\beta}+\widetilde{\gamma}-1=k-\widetilde{\beta}+\widetilde{\gamma}$ are nonnegative even integers. Without of loss of generality, we may suppose $\widetilde{\beta}\leq \widetilde{\gamma}$.\par

 \textbf{Case B-1:} When $k+\widetilde{\beta}-\widetilde{\gamma}=0$, i.e.,
$\widetilde{\gamma}=\widetilde{\beta}+k>2$, according to our assumption, there exists a reducible spherical conical metric on $S^{2}_{\{k,\widetilde{\beta},\widetilde{\gamma}-1\}}=S^{2}_{\{k,\widetilde{\beta},\widetilde{\beta}+k-1\}}$, since $k-1+\widetilde{\beta}-(\widetilde{\gamma}-1)=k-1+\widetilde{\beta}-(\widetilde{\beta}+k-1)=0$ and $k-1-\widetilde{\beta}+(\widetilde{\gamma}-1)=k-1-\widetilde{\beta}+(\widetilde{\beta}+k-1)=2(k-1)$ are nonnegative even integers. Since
$$S^{2}_{\{k+1,\widetilde{\beta},\widetilde{\beta}+k\}}
=S^{2}_{\{k,\widetilde{\beta},\widetilde{\beta}+k-1\}}+S^{2}_{\{1,1\}},$$
we obtain a reducible spherical conical metric on $S^{2}_{\{k+1,\widetilde{\beta},\widetilde{\gamma}\}}$ satisfied the condition.\par

\textbf{Case B-2:}~When $k+\widetilde{\beta}-\widetilde{\gamma}\geq2$,
if $0<\widetilde{\beta}<1$, according to our assumption, there exists a reducible spherical conical metric on $S^{2}_{\{k,1-\widetilde{\beta},\widetilde{\gamma}\}}$, since $k-1-(1-\widetilde{\beta})-\widetilde{\gamma}=k+\widetilde{\beta}-\widetilde{\gamma}-2$ and $k-1+(1-\widetilde{\beta})+\widetilde{\gamma}=k-\widetilde{\beta}+\widetilde{\gamma}$ are nonnegative even integers. Since
$$S^{2}_{\{k+1,\widetilde{\beta},\widetilde{\gamma}\}}=S^{2}_{\{k,1-\widetilde{\beta},\widetilde{\gamma}\}}+S^{2}_{\{\widetilde{\beta},\widetilde{\beta}\}},$$
we obtain a reducible spherical conical metric on $S^{2}_{\{k+1,\widetilde{\beta},\widetilde{\gamma}\}}$ satisfied the condition. \par

If $1<\widetilde{\beta}$, according to our assumption, there exists a  reducible spherical conical metric on $S^{2}_{\{k,\widetilde{\beta}-1,\widetilde{\gamma}\}}$, since $k-1-(\widetilde{\beta}-1)+\widetilde{\gamma}=k-\widetilde{\beta}+\widetilde{\gamma}$ and $k-1+(\widetilde{\beta}-1)-\widetilde{\gamma}=k+\widetilde{\beta}-\widetilde{\gamma}-2$ are nonnegative even integers. Since
$$S^{2}_{\{k+1,\widetilde{\beta},\widetilde{\gamma}\}}=S^{2}_{\{k,\widetilde{\beta}-1,\widetilde{\gamma}\}}+S^{2}_{\{1,1\}},$$
there exists a reducible spherical conical metric on $S^{2}_{\{k+1,\widetilde{\beta},\widetilde{\gamma}\}}$ satisfied the condition.\par

\subsection{Proof of Theorem \ref{Main-th-3-1}}
~~\par
\textbf{Necessity}~\par
Let ${\rm d}s^{2}$ be a reducible spherical conical metric on $S^{2}_{\{\alpha,\beta,\gamma\}}$ satisfied the given condition. Let $\omega$ be a character 1-form of ${\rm d}s^{2}$, and denote by $p_{1},p_{2}$ and $p_{3}$ the conical singularities of ${\rm d}s^{2}$ with corresponding angles $2\pi\alpha,2\pi\beta$ and $2\pi\gamma$, respectively. Since $\Phi$ has a single saddle point at $p_{1}$, it follows that $p_{1}$ is the unique zero of $\omega$ with order $\alpha-1$ and $p_{2},p_{3}$ are simple poles of $\omega$. The residues at $p_{2}$ and $p_{3}$ lead us to consider the following two distinct cases. \par
\textbf{(A)} $Res_{p_{2}}(\omega)\cdot Res_{p_{3}}(\omega)>0$.\par
\textbf{(B)} $Res_{p_{2}}(\omega)\cdot Res_{p_{3}}(\omega)<0$.\par
When case \textbf{(A)} occurs, without loss of generality, we may assume that  $Res_{p_{2}}(\omega)=\beta, Res_{p_{3}}(\omega)=\gamma$. Assume that $\omega$ has $p$ poles with residue $+1$ and $q$ poles with residue $-1$. Then, by the Residue theorem and the Riemann-Roch theorem, we obtain
\begin{equation*}
\begin{cases}
\beta+\gamma+p-q=0,\\
(\alpha-1)-p-q-2=-2.
\end{cases}
\end{equation*}

Solving these equations yields
\begin{equation*}
\begin{cases}
2p=\alpha-\beta-\gamma-1,\\
2q=\alpha+\beta+\gamma-1.
\end{cases}
\end{equation*}

Since $p$ and $q$ are nonnegative integers, both $\alpha-\beta-\gamma-1$ and $\alpha+\beta+\gamma-1$ are nonnegative even integers. Particularly, we obtain $\beta+\gamma\leq \alpha-1$.\par

Now set $f(z)=\exp(\int^{z} \omega)$. Then $f$ is a non-trivial rational function on $S^{2}$ with ${\rm deg}(f)=\frac{\alpha+\beta+\gamma-1}{2}$ and ${\rm d}f=f\cdot\omega$.
Since $f$ looks like $z\mapsto z^{\alpha}$ near the zero of $\omega$, ${\rm deg}(f)=\frac{\alpha+\beta+\gamma-1}{2}\geq \alpha$. Thus, $\beta+\gamma\geq\alpha+1$. It is a contradiction to $\beta+\gamma\leq \alpha-1$! Thus case \textbf{(A)} can not occur.\par
When case \textbf{(B)} occurs, without loss of generality, assume that  $Res_{p_{2}}(\omega)=\beta, Res_{p_{3}}(\omega)=-\gamma$. Suppose $\omega$ has $p$ poles with residue $+1$ and $q$ poles with residue $-1$. Then, by the Residue theorem and the Riemann-Roch theorem, we obtain
\begin{equation*}
\begin{cases}
\beta-\gamma+p-q=0,\\
(\alpha-1)-p-q-2=-2.
\end{cases}
\end{equation*}

Solving these equations yields
\begin{equation*}
\begin{cases}
2p=\alpha-\beta+\gamma-1,\\
2q=\alpha+\beta-\gamma-1.
\end{cases}
\end{equation*}
Since $p$ and $q$ are nonnegative integers, both $\alpha+\beta-\gamma-1$ and $\alpha-\beta+\gamma-1$ are nonnegative even integers.\par

Set $f(z)=\exp(\int^{z} \omega)$. Then $f$ is a non-trivial rational function on $S^{2}$ with ${\rm deg}(f)=\frac{\alpha+\beta+\gamma-1}{2}$ and ${\rm d} f=f\cdot\omega$.
Since $f$ looks like $z\mapsto z^{\alpha}$ near the zero of $\omega$, ${\rm deg}(f)=\frac{\alpha+\beta+\gamma-1}{2}\geq \alpha$. Hence, $\alpha+\beta+\gamma-1>2(\alpha-1)$.

\textbf{Sufficiency}~ \par
We shall now apply mathematical induction to the integer $\alpha$ to verify the sufficiency of the condition. \par

When $\alpha=2$, if both $\alpha-\beta+\gamma-1=1-\beta+\gamma$ and $\alpha+\beta-\gamma-1=1+\beta-\gamma$ are nonnegative even integers, then $|\beta-\gamma|=1$.  Without loss of generality, we may suppose $\beta<\gamma$. Then $\gamma=\beta+1$. According to \textbf{Example 7.1}, we derive the following identity $$S^{2}_{\{2,\beta,\gamma\}}=S^{2}_{\{2,\beta,\beta+1\}}=S^{2}_{\{\beta,\beta\}}+S^{2}_{\{1,1\}}.$$
Thus there exits a reducible spherical conical metric on $S^{2}_{\{2,\beta,\beta+1\}}$ satisfied the condition.\par

Note that when $\beta=\gamma=2$, if both $\alpha-\beta+\gamma-1$ and $\alpha+\beta-\gamma-1$ are nonnegative even integers and $\alpha+\beta+\gamma-1>2(\alpha-1)$, then
$\alpha=3$. According to \textbf{Example 7.3}, there exits a reducible spherical conical metric on $S^{2}_{\{3,2,2\}}$ satisfied the condition.\par

Suppose that when $\alpha=k\geq 2$, if both $\alpha-\beta+\gamma-1$ and $\alpha+\beta-\gamma-1$ are nonnegative even integers and $\alpha+\beta+\gamma-1>2(\alpha-1)$, there exists a reducible spherical conical metric on $S^{2}_{\{k,\beta,\gamma\}}$ such that the singularity of angle $2\pi\alpha$ is the unique saddle point of $\Phi$, the singularity of angle $2\pi\beta$ is a minimum point of $\Phi$ and the singularity of angle $2\pi\gamma$ is a maximum point of $\Phi$.\par

When $\alpha=k+1$, suppose that $\alpha-\widetilde{\beta}+\widetilde{\gamma}-1$ and $\alpha+\widetilde{\beta}-\widetilde{\gamma}-1$ are nonnegative even integers and $\alpha+\widetilde{\beta}+\widetilde{\gamma}-1>2(\alpha-1)$. \par

Without loss of generality, we may suppose $\widetilde{\beta}\leq\widetilde{\gamma}$, and $3\leq \widetilde{\gamma}$ (Otherwise, we obtain $\widetilde{\beta}=\widetilde{\gamma}=2$ and $\alpha=3$, and we have noted that there exits a reducible spherical conical metric on $S^{2}_{\{3,2,2\}}$ satisfied the condition).
Since $\alpha+\widetilde{\beta}-\widetilde{\gamma}-1=k+\widetilde{\beta}-\widetilde{\gamma}$ and $\alpha-\widetilde{\beta}+\widetilde{\gamma}-1=k-\widetilde{\beta}+\widetilde{\gamma}$ are nonnegative even integers,
we obtain both
$k-\widetilde{\beta}+(\widetilde{\gamma}-1)-1=\alpha-\widetilde{\beta}+\widetilde{\gamma}-3$ and
$k+\widetilde{\beta}-(\widetilde{\gamma}-1)-1=\alpha+\widetilde{\beta}-\widetilde{\gamma}-1$ are nonnegative even integers.
Since $\alpha+\widetilde{\beta}+\widetilde{\gamma}-1>2(\alpha-1)$,
we deduce that $k+\widetilde{\beta}+(\widetilde{\gamma}-1)-1=k+\widetilde{\beta}+\widetilde{\gamma}-2>2(k-1)$.
Thus, according to the assumption, there exists a reducible spherical conical metric on
$S^{2}_{\{k,\widetilde{\beta},\widetilde{\gamma}-1\}}$ such that the singularity of angle $2\pi k$ is the unique saddle point of $\widetilde{\Phi}$, the singularity of angle $2\pi\widetilde{\beta}$ is a minimum point of $\widetilde{\Phi}$ and the singularity of angle $2\pi(\widetilde{\gamma}-1)$ is a maximum point of $\widetilde{\Phi}$. Since
$$S^{2}_{\{k+1,\widetilde{\beta},\widetilde{\gamma}\}}=S^{2}_{\{k,\widetilde{\beta},\widetilde{\gamma}-1\}}+S^{2}_{\{1,1\}},$$
there exists a reducible spherical conical metric on $S^{2}_{\{k+1,\widetilde{\beta},\widetilde{\gamma}\}}$ such that the singularity of angle $2\pi\alpha$ is the unique saddle point of $\widetilde{\Phi}$, the singularity of angle $2\pi\widetilde{\beta}$ is a minimum point of $\widetilde{\Phi}$ and the singularity of angle $2\pi\widetilde{\gamma}$ is a maximum point of $\widetilde{\Phi}$.

\subsection{Proof of Theorem \ref{Main-th-3-2}}
~~\par

\textbf{Necessity}~\par

Let ${\rm d}s^{2}$ be a reducible spherical conical metric on $S^{2}_{\{\alpha,\beta,\gamma\}}$ satisfied the condition. Let $\omega$ be a character 1-form of ${\rm d}s^{2}$, and denote by $p_{1},p_{2}$ and $p_{3}$ the conical singularities of ${\rm d}s^{2}$ with corresponding angles $2\pi\alpha,2\pi\beta$ and $2\pi\gamma$, respectively. Since $p_{1},p_{2}$ are saddle points of $\Phi$, it follows that $p_{1}$ and $p_{2}$ are zeros of $\omega$ with orders $\alpha-1$ and $\beta-1$, respectively. Moreover, $p_{3}$ is a simple pole of $\omega$. Without loss of generality, we may suppose $Res_{p_{3}}(\omega)=\gamma$ (otherwise, consider $-\omega$ instead). Suppose that $\omega$ has $p$ poles with residues $+1$ and $q$ poles with residues $-1$. By applying the Residue theorem and the Riemann-Roch theorem, we obtain
\begin{equation*}
\begin{cases}
q-p=\gamma,\\
p+q=\alpha+\beta-1.
\end{cases}
\end{equation*}

Solving these equations yields
\begin{equation*}
\begin{cases}
2p=\alpha+\beta-\gamma-1,\\
2q=\alpha+\beta+\gamma-1.
\end{cases}
\end{equation*}
Since $p$ and $q$ are nonnegative integers, both $\alpha+\beta-\gamma-1$ and $ \alpha+\beta+\gamma-1$ are nonnegative even integers.\par

Set $f(z)=\exp(\int^{z}\omega)$. Then $f$ is a non-trivial rational function on $S^{2}$ with ${\rm deg}(f)=\frac{\alpha+\beta+\gamma-1}{2}$ and ${\rm d}f=f\cdot\omega$.
Since $f$ looks like $z\mapsto z^{\alpha}$ near the zero of $\omega$ with order $\alpha-1$, we obtain ${\rm deg}(f)=\frac{\alpha+\beta+\gamma-1}{2}> \alpha-1$. That is $\alpha+\beta+\gamma-1>2(\alpha-1)$. Similarly, we obtain $\alpha+\beta+\gamma-1>2(\beta-1)$. Hence, we deduce that
$$\alpha+\beta+\gamma-1>2\cdot \max\{\alpha-1,\beta-1\}.$$

\textbf{Sufficiency}~ \par
We shall now apply mathematical induction to the integer $\alpha+\beta-2$ to verify the sufficiency of the condition. \par

Since $\alpha,\beta,\gamma\geq2$, then $\alpha+\beta-2\geq 2$. Since $\alpha+\beta+\gamma-1$ and $\alpha+\beta-\gamma-1 $ are nonnegative even integers, if $\alpha+\beta-2=2$, then $\alpha=\beta=2$, $\gamma=3$. By \textbf{Example 7.3}, there exists a reducible spherical conical metric on $S^{2}_{\{2,2,3\}}$
such that the singularities of angles $4\pi,4\pi$ are saddle points of $\Phi$, the singularity of angle $6\pi$ is a minimum point of $\Phi$ and the two saddle points lie on the same geodesic connected a minimum point and a maximum point of $\Phi$.\par

Suppose that when $\alpha+\beta-2=k\geq2$, if both $2p\doteq\alpha+\beta-\gamma-1$ and $2q\doteq\alpha+\beta+\gamma-1$ are nonnegative even integers and $q>\max\{\alpha-1,\beta-1\}$, there exists a reducible spherical conical metric on $S^{2}_{\{\alpha,\beta,\gamma\}}$ such that the  singularities of angles $2\pi\alpha,2\pi\beta$ are saddle points of $\Phi$, the singularity of angle $2\pi\gamma$ is a minimum point of $\Phi$ and the two saddle points are lie on the same geodesic connected the minimum point (i.e. the singularity of order $\gamma$) and a maximum point of $\Phi$.\par

When $\widetilde{\alpha}+\widetilde{\beta}-2=k+1$, suppose that both $2\widetilde{p}=\widetilde{\alpha}+\widetilde{\beta}-\widetilde{\gamma}-1$ and $2\widetilde{q}=\widetilde{\alpha}+\widetilde{\beta}+\widetilde{\gamma}-1$ are nonnegative even integers and $\widetilde{q}> \max\{\widetilde{\alpha}-1,\widetilde{\beta}-1\}$.\par

Without loss of generality, we may suppose $\widetilde{\alpha}\geq\widetilde{\beta}$. Then $\widetilde{\alpha}\geq 3$. By the assumption, one can easily check that there exists a reducible spherical conical metric on $S^{2}_{\{\widetilde{\alpha}-1,\widetilde{\beta},\widetilde{\gamma}-1\}}$ such that the two singularities corresponding the orders $\widetilde{\alpha}-1$ and $\widetilde{\beta}$ are two saddle points of $\widehat{\Phi}$ which lie on the same geodesic connected a minimum point (i.e. the third singularity) and a maximum point of $\widehat{\Phi}$. Since
$$S^{2}_{\{\widetilde{\alpha},\widetilde{\beta},\widetilde{\gamma}\}}=S^{2}_{\{\widetilde{\alpha}-1,\widetilde{\beta},\widetilde{\gamma}-1\}}+S^{2}_{\{1,1\}},$$
there exists a reducible spherical conical metric on $S^{2}_{\{\widetilde{\alpha},\widetilde{\beta},\widetilde{\gamma}\}}$ satisfied the condition.\par

\subsection{Proof of Theorem \ref{Main-th-3-3}}
~~\par

\textbf{Necessity}~\par

Let ${\rm d}s^{2}$ be a reducible spherical conical metric on $S^{2}_{\{\alpha,\beta,\gamma\}}$ satisfied the condition. Let $\omega$ be a character 1-form of ${\rm d}s^{2}$, and denote by $p_{1},p_{2}$ and $p_{3}$ the conical singularities of ${\rm d}s^{2}$ with corresponding angles $2\pi\alpha,2\pi\beta$ and $2\pi\gamma$, respectively. Since $p_{1},p_{2}$ and $p_{3}$ are saddle points of $\Phi$, it follows that $p_{1},p_{2}$ and $p_{3}$ are zeros of $\omega$ with orders $\alpha-1, \beta-1$ and $\gamma-1$, respectively. Suppose that $\omega$ has $p$ poles with residues $+1$ and $q$ poles with residues $-1$. By applying the Residue theorem and the Riemann-Roch theorem, we obtain
\begin{equation*}
\begin{cases}
q-p=0,\\
p+q=\alpha+\beta+\gamma-1.
\end{cases}
\end{equation*}

Solving these equations yields
$$2p=2q=\alpha+\beta+\gamma-1.$$

Since $p$ and $q$ are integers, we deduce that $\alpha+\beta+\gamma-1$ is an even integer.\par

Set $f(z)=\exp(\int^{z}\omega)$. Then $f$ is a non-trivial rational function on $S^{2}$ with ${\rm deg}(f)=\frac{\alpha+\beta+\gamma-1}{2}$ and ${\rm d}f=f\cdot\omega$.
Since $f$ looks like $z\mapsto z^{\alpha}$ near the zero of $\omega$ with order $\alpha-1$, we obtain ${\rm deg}(f)=\frac{\alpha+\beta+\gamma-1}{2}> \alpha-1$. That is $\alpha+\beta+\gamma-1>2(\alpha-1)$. Similarly, we obtain $\alpha+\beta+\gamma-1>2(\beta-1)$ and $\alpha+\beta+\gamma-1>2(\gamma-1)$. Hence, we deduce that
$$\alpha+\beta+\gamma-1>2\cdot \max\{\alpha-1,\beta-1,\gamma-1\}.$$

\textbf{Sufficiency}~\par
We shall now apply mathematical induction to the integer $\alpha+\beta+\gamma-1$ to verify the sufficiency of the conditions. \par

Firstly, since $\alpha,\beta,\gamma\geq2$ and $\alpha+\beta+\gamma-1>2\cdot \max\{\alpha-1,\beta-1,\gamma-1\}$, we deduce that $\alpha+\beta+\gamma-1\geq6$. When $\alpha+\beta+\gamma-1=6$, we obtain $(\alpha,\beta,\gamma)=(2,2,3),(2,3,2)$ or $(3,2,2)$. As demonstrated in \textbf{Example 7.3}, there exits a reducible spherical conical metric exists on $S^{2}_{\{2,2,3\}}$ satisfied the condition. Particularly, there exists such a metric such that the three singularities lie on the same geodesic connected a minimum point and a maximum point of $\Phi$. \par

Secondly, we suppose that when $\alpha+\beta+\gamma-1=2l,l\geq3$ and $l>\max\{\alpha-1,\beta-1,\gamma-1\}$, there exists a  reducible spherical conical metric on $S^{2}_{\{\alpha,\beta,\gamma\}}$  satisfied the condition and all singularities lie on the same geodesic connected a minimum point and a maximum point of $\Phi$. \par

When $\widetilde{\alpha}+\widetilde{\beta}+\widetilde{\gamma}-1=2(l+1)$ and $l+1>\max\{\widetilde{\alpha}-1,\widetilde{\beta}-1,\widetilde{\gamma}-1\}$,
without loss of generality, we may suppose $\widetilde{\alpha}\geq\widetilde{\beta} \geq\widetilde{\gamma}$. Then $\widetilde{\alpha}\geq\widetilde{\beta}\geq 3$. Otherwise, if $\widetilde{\beta}=\widetilde{\gamma}=2$, we deduce that $\widetilde{\alpha}-1=2(l-1)>l$ for $l\geq3$. It is a contradiction! According the assumption, one can easily prove that there exists a reducible spherical cone metric
on $S^{2}_{\{\widetilde{\alpha}-1,\widetilde{\beta}-1,\widetilde{\gamma}\}}$ such that all singularities are saddle points and all singularities lie on the same geodesic connected a minimum point and a maximum point of $\widehat{\Phi}$.\par
Since
$$S^{2}_{\{\widetilde{\alpha},\widetilde{\beta},\widetilde{\gamma}\}}=
S^{2}_{\{\widetilde{\alpha}-1,\widetilde{\beta}-1,\widetilde{\gamma}\}}+S^{2}_{\{1,1\}},$$
there exists a reducible spherical conical metric on  $S^{2}_{\{\widetilde{\alpha},\widetilde{\beta},\widetilde{\gamma}\}}$ satisfied the condition.

\section{Proof of Theorem \ref{Main-th-4}}
In this section, we will give the proof of \textbf{Theorem} \ref{Main-th-4}. We observe that establishing the necessity is straightforward, employing \textbf{Theorem} \ref{Thm-1} as our basis. In contrast, demonstrating sufficiency presents a more formidable challenge. To address this, we shall construct a specialized instance of reducible spherical conical metrics, leveraging the principles established in \textbf{Theorem} \ref{Main-th-1} to substantiate the sufficiency of the conditions posited by our theorem.\par

\subsection{Proof of the Necessity of Theorem \ref{Main-th-4}}~\par

Let ${\rm d}s^{2}$ be a reducible spherical conical metric on $S^{2}_{\{\alpha_{1},\ldots,\alpha_{I},\beta_{1},\ldots,\beta_{J},\gamma_{1},\ldots,\gamma_{L}\}}$ satisfied the conditions in \textbf{Theorem} \ref{Main-th-4}. Let $\omega$ be a character 1-form of ${\rm d}s^{2}$, and  $p_{1},\ldots,p_{I+J+L}$ be conical singularities of ${\rm d}s^{2}$ with corresponding angles $2\pi\alpha_{1},\ldots,2\pi\alpha_{I}, 2\pi\beta_{1},\ldots,2\pi\beta_{J} $ and $2\pi\gamma_{1},\ldots,2\pi\gamma_{L}$, respectively. Since $p_{1},\ldots,p_{I}$ are saddle points of $\Phi$, it follows that $p_{1},\ldots,p_{I}$ are zeros of $\omega$ with orders $\alpha_{1}-1,\ldots,\alpha_{I}-1$, respectively. Since $p_{I+1},\ldots,p_{I+J}$ are minimal points of $\Phi$, $p_{I+1},\ldots,p_{I+J}$ are simple poles of $\omega$ with residues $\beta_{1},\ldots,\beta_{J}$, respectively. Similarly, $p_{I+J+1},\ldots,p_{I+J+L}$ are simple poles of $\omega$ with residues $-\gamma_{1},\ldots,-\gamma_{J}$, respectively. Suppose that $\omega$ has $p$ poles with residues $+1$ and $q$ poles with residues $-1$. Then, the associated residue vector of $\omega$ is
$$\overrightarrow{r}=\{\beta_{1},\ldots,\beta_{J},-\gamma_{1},\ldots,-\gamma_{L},\underbrace{1,\ldots,1}_{p~times},\underbrace{-1,\ldots,-1}_{q~times}\}.$$

 By applying the Residue theorem and the Riemann-Roch theorem, we obtain
\begin{equation*}
\begin{cases}
q-p=\sum\limits_{j=1}^{J}\beta_{j}-\sum\limits_{l=1}^{L}\gamma_{l},\\
p+q=\sum\limits_{i=1}^{I}(\alpha_{i}-1)-(J+L)+2.
\end{cases}
\end{equation*}

We use a proof by contradiction to prove ${\rm deg}(\overrightarrow{r})>\max\{\alpha_{i}-1\}_{1\leq i\leq I}$.\par

Suppose
$${\rm deg} (\overrightarrow{r})\leq \max\{\alpha_{i}-1\}_{1\leq i\leq I},$$
then, ${\rm deg} (\overrightarrow{r})$ is finite, and for each $i$, ${\rm deg} (\overrightarrow{r})\leq \alpha_{i}-1$.

Since ${\rm deg} (\overrightarrow{r})$ is finite, for some real number $c\neq0$, all residues of $c\omega$ are integers. Setting  $f(z)=\exp(c\int^{z}\omega)$ for each $z\in \overline{\mathbb{C}}$, then $f$ is a rational function on $S^{2}$ and ${\rm d}f=cf\cdot\omega$. Moreover, ${\rm deg} (f)={\rm deg} (\overrightarrow{r})$. But $f$ looks like $z\mapsto z^{\alpha_{i}}$ near $p_{i}$ since $p_{i}$ is a zero of $\omega$ with order $\alpha_{i}-1$.  Hence, ${\rm deg} (\overrightarrow{r})\geq\alpha_{i}$. It is a contradiction!\par

\subsection{Proof of the sufficiency of Theorem \ref{Main-th-4}}~\par

\textbf{Convention}: In this section, we shall denote the set of the following condition by symbol $\mathcal{D}$.\par

There exist two nonnegative integers $p$ and $q$ satisfied
\begin{equation*}
\begin{cases}
p+q=\sum\limits_{i=1}^{I}(\alpha_{i}-1)-(J+L)+2,\\
q-p=\sum\limits_{j=1}^{J}\beta_{j}-\sum\limits_{l=1}^{L}\gamma_{l},
\end{cases}
\end{equation*}
and
$${\rm deg}(\overrightarrow{r})>\max\{\alpha_{i}-1\}_{1\leq i\leq I},$$
where
$$\overrightarrow{r}=\{\beta_{1},\ldots,\beta_{J},-\gamma_{1},\ldots,-\gamma_{L},\overbrace{1,\ldots,1}^{p\geq0},\overbrace{-1,\ldots,-1}^{q\geq0}\}.$$

On $S^{2}_{\{\alpha_{1},\ldots,\alpha_{I},\beta_{1},\ldots,\beta_{J},\gamma_{1},\ldots,\gamma_{L}\}}$, we will construct a specialized reducible spherical conical metric to prove the sufficiency of \textbf{Theorem \ref{Main-th-4}}. This construction is carefully divided into two distinct scenarios:\par

$$\text{\textbf{Case A:} $I\geq4$ and $J=L=0$.}$$
$$\text{\textbf{Case B:} $I+J+L\geq4,I\geq1$ and $J+L\geq1$.}$$

\subsubsection{\textbf{Case A:} $I\geq4$ and $J=L=0$}~\par

For this scenario, the sufficiency of \textbf{Theorem \ref{Main-th-4}} is derived from the following claims.\par

\textbf{Claim 1} Let $\alpha$ be an integer with $\alpha\geq2$. If $I+\alpha$ is an even integer and $I+2>\alpha$, there exists a reducible spherical conical metric on $S^{2}_{\{\alpha,\underbrace{2,2,\ldots,2}_{I-1}\}}$ such that all singularities are saddle points of $\Phi$ and all singularities lie on the same geodesic connected a minimum point and a maximum point of $\Phi$.

\begin{proof}
We will apply mathematical induction to the integer $\alpha$ to demonstrate this claim.\par

When $\alpha=2$, it follows that $I$ is an even integer denoted by $I=2k,k\geq2$. Since
$$S^{2}_{\{\underbrace{2,2,\ldots,2}_{2k}\}}=S^{2}_{\{\underbrace{2,2,\ldots,2}_{2k-2}\}}+S^{2}_{\{1,1\}}=\ldots=\underbrace{S^{2}_{\{1,1\}}+\ldots+S^{2}_{\{1,1\}}}_{k+1},$$
we obtain a reducible spherical conical metric on $S^{2}_{\{\underbrace{2,2,\ldots,2}_{2k}\}}$ such that all singularities are saddle points of $\Phi$ and all singularities lie on the same geodesic connected a minimum point and a maximum point of $\Phi$.\par
Suppose that when $\alpha=k\geq2$, if $I+k$ is an even integer and $I+2>k$, there exists a reducible spherical conical metric on $S^{2}_{\{k,\underbrace{2,2,\ldots,2}_{I-1}\}}$ such that all singularities are saddle points of $\Phi$ and all singularities lie on the same geodesic connected a minimum point and a maximum point of $\Phi$.\par

When $\alpha=k+1$, $\widehat{I}+k+1$ is an even integer and $\widehat{I}+2>k+1$, it follows that $(\widehat{I}-1)+k$ is an even integer and $(\widehat{I}-1)+2>k$. By the assumption, it follows that there exists a reducible spherical conical metric on $S^{2}_{\{k,\underbrace{2,2,\ldots,2}_{\widehat{I}-2}\}}$ such that all singularities are saddle points of $\widehat{\Phi}$ and all singularities lie on the same geodesic connected a minimum point and a maximum point of $\widehat{\Phi}$.\par

Since
$$S^{2}_{\{k+1,\underbrace{2,2,\ldots,2}_{\widehat{I}-1}\}}=S^{2}_{\{k,\underbrace{2,2,\ldots,2}_{\widehat{I}-2}\}}+S^{2}_{\{1,1\}}$$
we obtain a  reducible spherical conical metric on $S^{2}_{\{k+1,\underbrace{2,2,\ldots,2}_{\widehat{I}-1}\}}$ such that all singularities are saddle points of $\widetilde{\Phi}$ and all singularities lie on the same geodesic connected a minimum point and a maximum point of $\widetilde{\Phi}$.
\end{proof}

\textbf{Claim 2} Let $I_{1}$ and $I_{2}$ be non-negative integers with $I_{1}+I_{2}\geq4$. If $I_{2}$ is an even integer, there exists a reducible spherical conical metric on $S^{2}_{\{\underbrace{3,\ldots,3}_{I_{1}},\underbrace{2,\ldots,2}_{I_{2}}\}}$ such that all singularities are saddle points of $\Phi$ and all singularities lie on the same geodesic connected a minimum point and a maximum point of $\Phi$.

\begin{proof}
 When $I_{1}=0$, since $I_{2}$ is an even integer, by \textbf{Claim 1}, there exists a reducible spherical conical metric on $S^{2}_{\{\underbrace{2,\ldots,2}_{I_{2}}\}}$ such that all singularities are saddle points of $\Phi$ and all singularities lie on the same geodesic connected a minimum point and a maximum point of $\Phi$.\par

 When $I_{1}=1$, since $I_{2}$ is an even integer, it follows that $I_{2}+4$ is an even integer. By \textbf{Claim 1}, there exists a reducible spherical conical metric on $S^{2}_{\{3,\underbrace{2,\ldots,2}_{I_{2}}\}}$ such that all singularities are saddle points of $\Phi$ and all singularities lie on the same geodesic connected a minimum point and a maximum point of $\Phi$.\par

Suppose that when $I_{1}=k\geq1$ and for any non-negative even integers $I_{2}$ with $I_{1}+I_{2}\geq4$, there exists a reducible spherical conical metric on $S^{2}_{\{\underbrace{3,\ldots,3}_{k},\underbrace{2,\ldots,2}_{I_{2}}\}}$ such that all singularities are saddle points of $\Phi$ and all singularities lie on the same geodesic connected a minimum point and a maximum point of $\Phi$.\par

When $I_{1}=k+1$, and $\widehat{I}_{2}$ is a non-negative even integers with $I_{1}+\widehat{I}_{2}\geq4$, it follows that $\widehat{I}_{2}+2$ is an even integer.
 By the assumption, there exists a reducible spherical conical metric on $S^{2}_{\{\underbrace{3,\ldots,3}_{k-1},\underbrace{2,\ldots,2}_{\widehat{I}_{2}+2}\}}$ such that all singularities are saddle points of $\widehat{\Phi}$ and all singularities lie on the same geodesic connected a minimum point and a maximum point of $\widehat{\Phi}$.\par

Since
$$S^{2}_{\{\underbrace{3,\ldots,3}_{k+1},\underbrace{2,\ldots,2}_{\widehat{I}_{2}}\}}=S^{2}_{\{\underbrace{3,\ldots,3}_{k-1},\underbrace{2,\ldots,2}_{\widehat{I}_{2}+2}\}}+S^{2}_{\{1,1\}},$$
we obtain a reducible spherical conical metric on $S^{2}_{\{\underbrace{3,\ldots,3}_{k+1},\underbrace{2,\ldots,2}_{\widehat{I}_{2}}\}}$ such that all singularities are saddle points of $\widetilde{\Phi}$ and all singularities lie on the same geodesic connected a minimum point and a maximum point of $\widetilde{\Phi}$.\par

\end{proof}

\textbf{Claim 3} Let $\alpha_{1},\ldots,\alpha_{I}$ be $I\geq4$ integers with $\alpha_{i}\geq2,\forall i$ and $\sum\limits_{i=1}^{I}(\alpha_{i}-1)\geq I+2$ . Under the condition $\mathcal{D}$, there exists a reducible spherical conical metric on $S^{2}_{\{\alpha_{1},\ldots,\alpha_{I}\}}$ such that all singularities are saddle points of $\Phi$ and all singularities lie on the same geodesic connected a minimum point and a maximum point of $\Phi$.\par

\begin{proof}
By condition $\mathcal{D}$, it follows that $\sum\limits_{i=1}^{I}(\alpha_{i}-1)$ is an even integer. \par

When $\sum\limits_{i=1}^{I}(\alpha_{i}-1)=I+2$, suppose $\alpha_{1}\geq\ldots \geq\alpha_{I}\geq 2$, then it follows that $\alpha_{1}=\alpha_{2}=3,\alpha_{3}=\ldots=\alpha_{I}=2$ or $\alpha_{1}=4,\alpha_{2}=\alpha_{3}=\ldots=\alpha_{I}=2$ and $I$ is an even integer. According to the \textbf{Claim 1} (or \textbf{Claim 2}), we know that
there exists a reducible spherical conical metric on $S^{2}_{\{4,\underbrace{2,\ldots,2}_{2k-1}\}}$ ( or $S^{2}_{\{3,3,\underbrace{2,\ldots,2}_{2k-2}\}}$) such that all singularities are saddle points of $\Phi$ and all singularities lie on the same geodesic connected a minimum point and a maximum point of $\Phi$.\par

Suppose that when $\sum\limits_{i=1}^{I}(\alpha_{i}-1)=2l,l\geq\frac{I+2}{2}$ and $l+1>\max\{\alpha_{1}-1,\ldots,\alpha_{I}-1\}$, there exists a reducible spherical conical metric on $S^{2}_{\{\alpha_{1},\alpha_{2},\ldots,\alpha_{I}\}}$  such that all singularities are saddle points of $\Phi$ and all singularities lie on the same geodesic connected a minimum point and a maximum point of $\Phi$.\par

Assume that $\sum\limits_{i=1}^{I}(\alpha_{i}-1)=2(l+1)$ and $l+2>\max\{\alpha_{1}-1,\ldots,\alpha_{I}-1\}$. Without loss of generality, we may suppose $\alpha_{1}\geq\ldots \geq\alpha_{I}\geq 2$. Hence, $\alpha_{1},\alpha_{2}\geq3$. \par

Since $\alpha_{1}\geq\alpha_{2}\geq\alpha_{3}\geq\ldots\geq\alpha_{I}\geq2$ and $2(l+1)=\sum\limits_{i=1}^{I}(\alpha_{i}-1)>3(\alpha_{3}-1)$, it follows that $l+1>\alpha_{3}-1$. Since $l+2>\max\{\alpha_{1}-1,\ldots,\alpha_{I}-1\}$, it follows that $ l+1>\alpha_{1}-2\geq\alpha_{2}-2$. Consequently,
 $$l+1>\max\{\alpha_{1}-2,\alpha_{2}-2,\alpha_{3}-1,\ldots,\alpha_{I}-1\}.$$

According to the assumption, there exists a reducible spherical conical metric
on $S^{2}_{\{\alpha_{1}-1,\alpha_{2}-1,\alpha_{3},\ldots,\alpha_{I}\}}$ such that all singularities are saddle points of $\widehat{\Phi}$ and all singularities lie on the same geodesic connected a minimum point and a maximum point of $\widehat{\Phi}$.\par

Since
$$S^{2}_{\{\alpha_{1},\alpha_{2},\ldots,\alpha_{I}\}}=
S^{2}_{\{\alpha_{1}-1,\alpha_{2}-1,\alpha_{3},\ldots,\alpha_{I}\}}+S^{2}_{\{1,1\}},$$
we obtain a reducible spherical conical metric on  $S^{2}_{\{\alpha_{1},\alpha_{2},\ldots,\alpha_{I}\}}$ such that all singularities are saddle points of $\widetilde{\Phi}$ and all singularities lie on the same geodesic connected a minimum point and a maximum point of $\widetilde{\Phi}$.\par

\end{proof}

\subsubsection{\textbf{Two Lemmas}}~\par

In this subsection, we shall demonstrate two critical lemmas that serve as foundational elements in our demonstration of the rest case.

\begin{lemma}\label{Lemma-A-1}
Let $m,p$ and $q$ be positive integers with $p+q=m+2$. Let $\overrightarrow{r}=\{a_{1},\ldots,a_{p},-b_{1},\ldots,-b_{q}\}$ be a real residue vector with $0<a_{1}\leq a_{2}\leq\ldots\leq a_{p}$, $0<b_{1}\leq b_{2}\leq\ldots\leq b_{q}$, and ${\rm deg}(\overrightarrow{r})>m$. The following assertions hold.
\begin{enumerate}
 \item If $p<q$, there exist indices $i_{0} \in \{1,2,\ldots,p\}$ and $j_{0}\in\{1, 2\}$ such that for all $i\geq i_{0}$, $a_{i}>b_{j_{0}}$, and the modified residue vector $$\overrightarrow{r}_{j_{0}i}=\{a_{1},\ldots, (a_{i}-b_{j_{0}}),\ldots,a_{p},-b_{1}, \widehat{-b_{j_{0}}}, \ldots, -b_{q}\}$$
      satisfies ${\rm deg}(\overrightarrow{r}_{j_{0}i})>m-1$, where $\widehat{-b_{j_{0}}}$ denotes the exclusion of the term $-b_{j_{0}}$ from the vector.
 \item If $p=q$ and $a_{1}\leq b_{1}$, there exists an index $i_{0}\in \{2,\ldots,p\}$ such that for all $i\geq i_{0}$, $a_{i}>b_{1}$, and the modified residue vector
 $$\overrightarrow{r}_{1i}=\{a_{1}, \ldots, (a_{i}-b_{1}), \ldots, a_{p},-b_{2},\ldots,-b_{q}\}$$
 satisfies ${\rm deg}(\overrightarrow{r}_{1i})>m-1$.
\end{enumerate}
\end{lemma}

\begin{proof}
We focus on demonstrating the proof for the case where the degree of the residue vector $\overrightarrow{r}$ is finite, denoted as ${\rm deg} (\overrightarrow{r})<+\infty$. The analogous case for ${\rm deg} (\overrightarrow{r})=+\infty$ can be addressed with a similar approach.\par

When ${\rm deg} (\overrightarrow{r})<+\infty$, without loss of generality, we may suppose
$$\overrightarrow{r}=\{a_{1},\ldots,a_{p},-b_{1},\ldots,-b_{q}\}$$
is primitive. \par

(1)  Since $\overrightarrow{r}$ is a residue vector, it follows that $\sum\limits_{i=1}^{p}a_{i}=\sum\limits_{j=1}^{q}b_{j}$. Since $0<a_{1}\leq a_{2}\leq\ldots\leq a_{p}$, $0<b_{1}\leq b_{2}\leq\ldots\leq b_{q}$ and  $p<q$, it is obvious that there exist indices $i_{0} \in \{1,2,\ldots,p\}$ and $j_{0}\in\{1, 2\}$ such that $\forall i\geq i_{0},a_{i}>b_{j_{0}}$. \par

If $a_{1}=b_{1}$, define a residue vector by
$$\overrightarrow{r}_{1i}=\{a_{1},\ldots,(a_{i}-b_{1}),\ldots, a_{p},-b_{2},\ldots,-b_{q}\}.$$
It is easy to prove that $\overrightarrow{r}_{1i}$ is primitive, and ${\rm deg}(\overrightarrow{r}_{1i})=\sum\limits_{j=2}^{q}b_{j}={\rm deg}(\overrightarrow{r})-b_{1}$. If $b_{1}=1$, according to ${\rm deg}(\overrightarrow{r})>m$, ${\rm deg}(\overrightarrow{r}_{1i})>m-1$. If $b_{1}\geq2$,
${\rm deg}(\overrightarrow{r}_{1i})=\sum\limits_{j=2}^{q}b_{j}\geq 2(q-1)=2q-2> p+q-2=m>m-1$.\par

If $a_{1}<b_{1}$, then $a_{1}< b_{2}\leq\ldots \leq b_{q}$. Denote the greatest common divisor of $a_{1},\ldots,(a_{i}-b_{1}),\ldots, a_{p},b_{2},\ldots,b_{q}$ by $c$. It follows that $\frac{b_{j}}{c}\geq2, \forall j\geq2$. Obviously, the residue vector defined by
$$\overrightarrow{r}_{1i}=\{a_{1},\ldots,(a_{i}-b_{1}),\ldots, a_{p},-b_{2},\ldots,-b_{q}\}$$
satisfies
$${\rm deg} (\overrightarrow{r}_{1i})=\sum_{j=2}^{q}\frac{b_{j}}{c}\geq 2(q-1)=2q-2>p+q-2=m>m-1.$$

If $a_{1}>b_{1}$, we consider the relations among  $a_{1},b_{1},b_{2}$. \par

Similar as above, if $b_{1}=b_{2}$, the residue vector defined by
$\overrightarrow{r}_{1i}=\{a_{1},\ldots, \\ (a_{i}-b_{1}),\ldots, a_{p},-b_{2},\ldots,-b_{q}\}$ satisfies ${\rm deg} (\overrightarrow{r}_{1i})>m-1$; if $b_{1}<b_{2}\leq a_{1}$, the residue vector defined by
$\overrightarrow{r}_{2i}=\{a_{1},\ldots,(a_{i}-b_{2}),\ldots, a_{p},-b_{1},-b_{3},\ldots,-b_{q}\}$ satisfies
${\rm deg}(\overrightarrow{r}_{1i})>m-1$; if $b_{1}<a_{1}<b_{2}$, the residue vector defined by $\overrightarrow{r}_{1i}=\{a_{1},\ldots,(a_{i}-b_{1}),\ldots a_{p},-b_{2},\ldots,-b_{q}\}$ satisfies ${\rm deg}(\overrightarrow{r}_{1i})>m-1$.\par

(2) Since $p+q=m+2$ and $p=q$, $p=q=\frac{m+2}{2}\geq2$.  Since $\sum\limits_{i=1}^{p}a_{i}=\sum\limits_{j=1}^{q}b_{j}$, there exists $1\leq i_{0}\leq p$ such that $\forall i\geq i_{0},a_{i}>b_{1}$. Otherwise, $a_{1}=\ldots=a_{p}=b_{1}=\ldots=b_{q}$, then ${\rm deg} (\overrightarrow{r})=\frac{m+2}{2}\leq m$. It is a contradiction!\par

Obviously, the residue vector defined by
$$\overrightarrow{r}_{1i}=\{a_{1},\ldots,(a_{i}-b_{1}),\ldots, a_{p},-b_{2},\ldots,-b_{q}\}, i\geq i_{0} $$
satisfies ${\rm deg}(\overrightarrow{r}_{1i})>m-1$.

\end{proof}

\begin{lemma}\label{Lemma-A-2}
 Let $\alpha_{1},\ldots,\alpha_{I}$ be $I\geq2$ integers with $\alpha_{1}\geq \alpha_{2}\geq \ldots\geq \alpha_{I}\geq2$. Let $p,q$ and $m$ be positive integers with $p\leq q, p+q=m+2$ and $(\alpha_{1}-1)+\ldots+(\alpha_{I}-1)=m$.  Suppose  $\overrightarrow{r}=\{a_{1},\ldots,a_{p},-b_{1},\ldots,-b_{q}\}$ is a real residue vector with $0<a_{1}\leq a_{2}\leq \ldots\leq a_{p}$ and $0<b_{1}\leq b_{2}\leq \ldots\leq b_{q}$. Assume that
 ${\rm deg}(\overrightarrow{r})>\max\{\alpha_{1}-1,\ldots,\alpha_{I}-1\}$ and ${\rm deg}(\overrightarrow{r})>\frac{m+2}{2}$. The following assertions hold.
\begin{enumerate}
 \item If $ p<q$, there exist indices $i_{0} \in \{1,2,\ldots,p\}$ and $j_{0}\in\{1, 2\}$ such that for all $i\geq i_{0}$, $a_{i}>b_{j_{0}}$, and the modified residue vector
      $\overrightarrow{r}_{j_{0}i}=\{a_{1},\ldots,(a_{i}-b_{j_{0}}),\ldots, a_{p},-b_{1},-\widehat{b}_{j_{0}},\ldots,-b_{q}\}$
      satisfies
     ${\rm deg}(\overrightarrow{r}_{j_{0}i})>\max\{\alpha_{1}-2,\alpha_{2}-1,\ldots,\alpha_{I}-1\}$, where $\widehat{-b_{j_{0}}}$ denotes the exclusion of the term $-b_{j_{0}}$ from the vector.

\item If $p=q$ and $a_{1}\leq b_{1}$, there exists a index $i_{0} \in \{2,\ldots,p\}$ such that for all $i\geq i_{0}$, $a_{i}>b_{1}$ and the modified residue $\overrightarrow{r}_{1i}=\{a_{1},\ldots,(a_{i}-b_{1}),\ldots, a_{p},-b_{2},\ldots,-b_{q}\}$ satisfies
    ${\rm deg}(\overrightarrow{r}_{1i})>\max\{\alpha_{1}-2,\alpha_{2}-1, \ldots,\\ \alpha_{I}-1\}$.
\end{enumerate}
\end{lemma}

\begin{proof}
When ${\rm deg}(\overrightarrow{r})>m$, according to \textbf{Lemma \ref{Lemma-A-1}}, the assertion holds true. Now, let us consider the case where ${\rm deg} (\overrightarrow{r})\leq m$. \par
Without loss of generality, we may suppose
$$\overrightarrow{r}=\{a_{1},\ldots,a_{p},-b_{1},\ldots,-b_{q}\}$$
is primitive. \par

Since $\alpha_{1}\geq \alpha_{2}\geq \ldots\geq \alpha_{I}\geq2$, it follows that $\max\{\alpha_{1}-1,\ldots,\alpha_{I}-1\}=\alpha_{1}-1$ and ${\rm deg}(\overrightarrow{r})>\alpha_{1}-1$.

Since $p\leq q$ and $p+q=m+2$, we obtain $q\geq\frac{m+2}{2}$. Additionally, we have $b_{1}=b_{2}=1$. Otherwise, if $b_{2}\geq2$, then
${\rm deg}(\overrightarrow{r})=\sum\limits_{j=1}^{q}b_{j}>\sum\limits_{j=2}^{q}b_{j}\geq2(q-1)\geq p+q-2=m$, which contradicts our assumption.\par

If $p<q$, since $\sum\limits_{i=1}^{p}a_{i}=\sum\limits_{j=1}^{q}b_{j}$, it is obvious that there exists an index $i_{0}\in\{1,\ldots,p\}$ such that $a_{i_{0}}>1$. \par

If $p=q$, since $p+q=m+2$, $p=q=\frac{m+2}{2}$. Since ${\rm deg}(\overrightarrow{r})>\frac{m+2}{2}$, there exists an index $i_{0}\in\{2,\ldots,p\}$ such that $a_{i_{0}}>1$. Otherwise, $a_{1}=\ldots=a_{p}=1$. Moreover, ${\rm deg}(\overrightarrow{r})=\frac{m+2}{2}$. It is a contradiction! \par

Define a  residue vector defined by
$$\overrightarrow{r}_{1i}=\{a_{1},\ldots, a_{i-1},(a_{i}-1),a_{i+1},\ldots, a_{p},-b_{2},\ldots,-b_{q}\},~i\geq i_{0}.$$
Since $b_{2}=1$, it follows that $ \overrightarrow{r}_{1i}$ is primitive.\par

When $\alpha_{1}>\alpha_{2}\geq\alpha_{3}\geq\ldots\geq\alpha_{I}$, it follows that
 ${\rm deg}(\overrightarrow{r}_{1i})=\sum\limits_{j=2}^{q}b_{j}={\rm deg} (\overrightarrow{r})-1>\alpha_{1}-2=\max\{\alpha_{1}-2,\alpha_{2}-1,\ldots,\alpha_{I}-1\}$.\par

When $\alpha_{1}=\ldots=\alpha_{I_{1}}>\alpha_{I_{1}+1}\geq\ldots\geq\alpha_{I}$ for some $I_{1}\geq2$,  since $(\alpha_{1}-1)+\ldots+(\alpha_{I}-1)=m\geq I_{1}\cdot(\alpha_{1}-1)$, we obtain
$\alpha_{1}-1\leq\frac{m}{I_{1}}\leq\frac{m}{2}$.
Then
${\rm deg} (\overrightarrow{r})>\frac{m+2}{2}=\frac{m}{2}+1\geq\alpha_{1}$, and
$${\rm deg}(\overrightarrow{r}_{1i})= {\rm deg}(\overrightarrow{r})-1>\alpha_{1}-1=\alpha_{2}-1=\max\{\alpha_{1}-2,\alpha_{2}-1,\ldots,\alpha_{I}-1\}.$$

\end{proof}

\subsubsection{\textbf{Case B: $I+J+L\geq4,I\geq1$ and $J+L\geq1$}}~\par

Initially, under the condition $\mathcal{D}$, we demonstrate that the sufficiency of \textbf{Theorem \ref{Main-th-4}} holds when $J=L=\frac{\sum\limits_{i=1}^{I}(\alpha_{i}-1)+2}{2}> \max\{\alpha_{1}-1,\ldots,\alpha_{I}-1\}$ and $\beta_{1}=\ldots=\beta_{J}=\gamma_{1}=\ldots=\gamma_{L}$. In this case, the degree of the residue vector $\overrightarrow{r}$ is identified as the minimum.\par

\textbf{Claim 1}: Under the condition $\mathcal{D}$, when $J=L=\frac{\sum\limits_{i=1}^{I}(\alpha_{i}-1)+2}{2}> \max\{\alpha_{1}-1,\ldots,\alpha_{I}-1\}$ and $\beta_{1}=\ldots=\beta_{J}=\gamma_{1}=\ldots=\gamma_{L}=\beta$, there exists a reducible spherical conical metric on $S^{2}_{\{\alpha_{1},\ldots,\alpha_{I},\underbrace{\beta,\ldots,\beta}_{J+L}\}}$ such that the singularities of angles $2\pi\alpha_{1},\ldots,2\pi\alpha_{I}$ are saddle points of $\Phi$, the singularities of angles $\underbrace{2\pi\beta,\ldots,2\pi\beta}_{J}$ are minimum points of $\Phi$, the singularities of angles $\underbrace{2\pi\beta,\ldots,2\pi\beta}_{L}$ are maximal points of $\Phi$, and all saddle points lie on the same geodesic connected a minimum point and a maximum point of $\Phi$.\par

\begin{proof}
Since $J=L=\frac{\sum\limits_{i=1}^{I}(\alpha_{i}-1)+2}{2}$ and condition $\mathcal{D}$, it follows that $p=q=0$ and $I\geq2$.

According to the following facts, we know that the claim holds true.\par

\textbf{Fact 1}: When $I=2k$ with $k\geq2$ and $\alpha_{1}=\ldots=\alpha_{I}=2$, since (refer to Figure 3)
$$S^{2}_{\{\underbrace{2,\ldots,2}_{2k},\underbrace{\beta,\ldots,\beta}_{2k+2}\}}=S^{2}_{\{\underbrace{2,\ldots,2}_{2k-2},\underbrace{\beta,\ldots,\beta}_{2k}\}}+S^{2}_{\{\beta,\beta\}}=\ldots
=\underbrace{S^{2}_{\{\beta,\beta\}}+\ldots+S^{2}_{\{\beta,\beta\}}}_{k+1},$$
we obtain a reducible spherical conical metric on $S^{2}_{\{\underbrace{2,\ldots,2}_{2k},\underbrace{\beta,\ldots,\beta}_{2k}\}}$ which satisfies the claim. \par

\textbf{Fact 2}: Under the condition $\mathcal{D}$, when $I=2k+1$ with $k\geq1$ and $\alpha_{1}=3,\alpha_{2}=\ldots=\alpha_{I}=2$, since
$$S^{2}_{\{3,\underbrace{2,\ldots,2}_{2k},\underbrace{\beta,\ldots,\beta}_{2k+4}\}}
=S^{2}_{\{\underbrace{2,\ldots,2}_{2k},\underbrace{\beta,\ldots,\beta}_{2k+2}\}}+S^{2}_{\{\beta,\beta\}},$$
we obtain a reducible spherical conical metric on $S^{2}_{\{\underbrace{2,\ldots,2}_{2k},\underbrace{\beta,\ldots,\beta}_{2k}\}}$ which satisfies the claim. \par

\textbf{Fact 3}: Under the condition $\mathcal{D}$, when $I=2k$ with $k\geq2$ and $\alpha_{1}=4,\alpha_{2}=\ldots=\alpha_{I}=2$, since
$$S^{2}_{\{4,\underbrace{2,\ldots,2}_{2k-1},\underbrace{\beta,\ldots,\beta}_{2k+4}\}}
=S^{2}_{\{3,\underbrace{2,\ldots,2}_{2k-2},\underbrace{\beta,\ldots,\beta}_{2k+2}\}}+S^{2}_{\{\beta,\beta\}},$$
we obtain a reducible spherical conical metric on $S^{2}_{\{4,\underbrace{2,\ldots,2}_{2k-2},\underbrace{\beta,\ldots,\beta}_{2k}\}}$ which satisfies the claim.\par

Furthermore, facts 2 and 3 are included in the following fact.\par

\textbf{Fact 4}: Under the condition $\mathcal{D}$, when $\alpha_{1}\geq3,\alpha_{2}=\ldots=\alpha_{I}=2$, since
$$S^{2}_{\{\alpha_{1},\underbrace{2,\ldots,2}_{I-1},\beta,\ldots,\beta\}}
=S^{2}_{\{\alpha_{1}-1,\underbrace{2,\ldots,2}_{I-2},\beta,\ldots,\beta\}}+S^{2}_{\{\beta,\beta\}},$$
we obtain a reducible spherical conical metric on $S^{2}_{\{\alpha_{1},\underbrace{2,\ldots,2}_{I-1},\beta,\ldots,\beta\}}$ which satisfies the claim.\par

\textbf{Fact 5}:  Under the condition $\mathcal{D}$, when $\alpha_{1}\geq3$ and $\alpha_{2}\geq3$, since
$$S^{2}_{\{\alpha_{1}, \alpha_{2},\alpha_{3},\ldots,\alpha_{I},\underbrace{\beta,\ldots,\beta}_{2k+2}\}}=
S^{2}_{\{\alpha_{1}-1, \alpha_{2}-1,\alpha_{3},\ldots,\alpha_{I},\underbrace{\beta,\ldots,\beta}_{2k}\}}+S^{2}_{\{\beta,\beta\}},$$
we obtain a reducible spherical conical metric on $S^{2}_{\{\alpha_{1}, \alpha_{2},\alpha_{3},\ldots,\alpha_{I},\underbrace{\beta,\ldots,\beta}_{2k+2}\}}$ which satisfies the claim.\par

\end{proof}

\textbf{Claim 2}: Under the condition $\mathcal{D}$, when $I=1$, there is a reducible spherical conical metric on $S^{2}_{\{\alpha_{1},\beta_{1},\ldots,\beta_{J},\gamma_{1},\ldots,\gamma_{L}\}}$ such that the singularity of angle $2\pi\alpha_{1}$ is the unique saddle points of $\Phi$, the singularities of angles $2\pi\beta_{1},\ldots,2\pi\beta_{J}$ are minimum points of $\Phi$, and the singularities of angles $2\pi\gamma_{1},\ldots,2\pi\gamma_{L}$ are maximum points of $\Phi$.\par

\begin{proof}

We shall now apply mathematical induction to the integer $\alpha_{1}$ to verify the claim. \par

When $\alpha_{1}=2$, since $J+L\geq3$, we deduce that $J+L=3$  and the following two cases.
\begin{equation*}
\textbf{a-1:}
\begin{cases}
J=2,\\
L=1,\\
p=0,\\
q=0,\\
\gamma_{1}=\beta_{1}+\beta_{2};
\end{cases}
\textbf{a-2:}
\begin{cases}
J=1,\\
L=2,\\
p=0,\\
q=0,\\
\beta_{1}=\gamma_{1}+\gamma_{2}.
\end{cases}
\end{equation*}

When \textbf{Case a-1} arises, since  (refer to \textbf{Example 7.1})
$$S^{2}_{\{2,\beta_{1},\beta_{2},\gamma_{1}\}}=S^{2}_{\{2,\beta_{1},\beta_{2},\beta_{1}+\beta_{2}\}}=S^{2}_{\{\beta_{1},\beta_{1}\}}+S^{2}_{\{\beta_{2},\beta_{2}\}},$$
the claim holds true.\par

When \textbf{Case a-2} arises, we have (refer to \textbf{Example 7.1})
$$S^{2}_{\{2,\beta_{1},\gamma_{1},\gamma_{2}\}}=S^{2}_{\{2,\gamma_{1}+\gamma_{2},\gamma_{1},\gamma_{2}\}}=S^{2}_{\{\gamma_{1},\gamma_{1}\}}+S^{2}_{\{\gamma_{2},\gamma_{2}\}},$$
the claim holds true.\par

Suppose that when $\alpha_{1}=k\geq2$, and there are two nonnegative integers $p$ and $q$ satisfied the condition $\mathcal{D}$,
there exists a reducible spherical conical metric on $S^{2}_{\{k,\beta_{1},\ldots,\beta_{J},\gamma_{1},\ldots,\gamma_{L}\}}$ such that
the singularity of angle $2\pi k$ is the unique saddle points of $\Phi$, the singularities of angles $2\pi\beta_{1},\ldots,2\pi\beta_{J}$ are minimum points of $\Phi$, and the singularities of angles $2\pi\gamma_{1},\ldots,2\pi\gamma_{L}$ are maximum points of $\Phi$. \par

When $\alpha_{1}=k+1$, suppose that $\widetilde{\beta}_{1},\ldots,\widetilde{\beta}_{\widetilde{J}},\widetilde{\gamma}_{1},\ldots,\widetilde{\gamma}_{\widetilde{L}}$ are $\widetilde{J}+\widetilde{L}\geq3$ positive real numbers with $\widetilde{\beta}_{j}\neq1,\forall j$ and $\widetilde{\gamma}_{l}\neq1,\forall l$, and there are two nonnegative integers $\widetilde{p}$ and $\widetilde{q}$ satisfied
\begin{equation*}
\begin{cases}
\widetilde{p}+\widetilde{q}=k+2-(\widetilde{J}+\widetilde{L}),\\
\widetilde{q}-\widetilde{p}=\sum\limits_{j=1}^{\widetilde{J}}\widetilde{\beta}_{j}-\sum\limits_{l=1}^{\widetilde{L}}\widetilde{\gamma}_{l},\\
{\rm deg}(\overrightarrow{\widetilde{r}})>k,
\end{cases}
\end{equation*}
where
$$\overrightarrow{\widetilde{r}}=\{\widetilde{\beta}_{1},\ldots,\widetilde{\beta}_{\widetilde{J}},-\widetilde{\gamma}_{1},\ldots,-\widetilde{\gamma}_{\widetilde{L}},
\underbrace{1,\ldots,1}_{\widetilde{p}},\underbrace{-1,\ldots,-1}_{\widetilde{q}}\}.$$

Without loss of generality, we may suppose that $\widetilde{J}+\widetilde{p}\leq \widetilde{L}+\widetilde{q}$, $\widetilde{\gamma}_{1}\leq\widetilde{\gamma}_{2}\leq\ldots\leq\widetilde{\gamma}_{\widetilde{q}}$ and $\widetilde{\beta}_{1}\leq\widetilde{\beta}_{1}\leq\ldots\leq \widetilde{\beta}_{J}$. Since $I=1$, it follows that the case of \textbf{Claim 1} does not arise.\par

 Next, we will prove there exists a reducible spherical conical metric on $S^{2}_{\{k+1,\widetilde{\beta}_{1},\ldots,\widetilde{\beta}_{\widetilde{J}},\widetilde{\gamma}_{1},\ldots,\widetilde{\gamma}_{\widetilde{L}}\}}$ such that
the singularity of angle $2\pi (k+1)$ is the unique saddle points of $\widetilde{\Phi}$, the singularities of angles $2\pi\widetilde{\beta}_{1},\ldots,2\pi\widetilde{\beta}_{\widetilde{J}}$ are minimum points of $\widetilde{\Phi}$, and the singularities of angles $2\pi\widetilde{\gamma}_{1},\ldots,2\pi\widetilde{\gamma}_{L}$ are maximum points of $\widetilde{\Phi}$ when $1\leq\widetilde{q}$, $\widetilde{\gamma}_{1}<1<\widetilde{\gamma}_{2}\leq\ldots\leq\widetilde{\gamma}_{\widetilde{q}}$ and $1<\widetilde{\beta}_{1}\leq \ldots\leq \widetilde{\beta}_{J}$. For other cases, the proofs are similar.\par

Since the degree of residue vector
$$\overrightarrow{\widetilde{r}}=\{\widetilde{\beta}_{1},\ldots,\widetilde{\beta}_{\widetilde{J}},-\widetilde{\gamma}_{1},\ldots,-\widetilde{\gamma}_{\widetilde{L}},
\underbrace{1,\ldots,1}_{\widetilde{p}},\underbrace{-1,\ldots,-1}_{\widetilde{q}}\}$$
satisfies ${\rm deg}(\overrightarrow{\widetilde{r}})>k$,  according to \textbf{Lemma \ref{Lemma-A-1}}, the modified residue vector defined by
$$\overrightarrow{\widehat{r}}=\{\widetilde{\beta}_{1}-1,\ldots,\widetilde{\beta}_{\widetilde{J}},-\widetilde{\gamma}_{1},
-\widetilde{\gamma}_{2},\ldots,-\widetilde{\gamma}_{\widetilde{L}}, \underbrace{1,\ldots,1}_{\widetilde{p}},\underbrace{-1,\ldots,-1}_{\widetilde{q}-1} \}$$
satisfies ${\rm deg}(\overrightarrow{\widehat{r}})>k-1$.\par
Setting $\widehat{p}=\widetilde{p}$ and $\widehat{q}=\widetilde{q}-1$, by the assumption, it follows that
\begin{equation*}
\begin{cases}
\widehat{p}+\widehat{q}=k+1-(\widetilde{J}+\widetilde{L}),\\
\widehat{q}-\widehat{p}=\sum\limits_{j=1}^{\widetilde{J}}\widetilde{\beta}_{j}-\sum\limits_{l=1}^{\widetilde{L}}\widetilde{\gamma}_{l}-1.
\end{cases}
\end{equation*}
Hence, there exists a reducible spherical conical metric on $S^{2}_{\{k,\widetilde{\beta}_{1}-1,\ldots,\widetilde{\beta}_{\widetilde{J}},\widetilde{\gamma}_{1},\ldots,\widetilde{\gamma}_{\widetilde{L}}\}}$ such that the singularity of angle $2\pi k$ is the unique saddle point of $\widehat{\Phi}$, the singularities of angles $2\pi(\widetilde{\beta}_{1}-1),\ldots,2\pi\widetilde{\beta}_{\widetilde{J}}$ are minimum points of $\widehat{\Phi}$, and $2\pi\widetilde{\gamma}_{1},\ldots,2\pi\widetilde{\gamma}_{\widetilde{L}}$ are maximum points of $\widehat{\Phi}$. \par

Since (similar as \textbf{Example 7.2})
$$ S^{2}_{\{k+1,\widetilde{\beta}_{1},\ldots,\widetilde{\beta}_{\widetilde{J}},\widetilde{\gamma}_{1},\ldots,\widetilde{\gamma}_{\widetilde{L}}\}}= S^{2}_{\{k,\widetilde{\beta}_{1}-1,\ldots,\widetilde{\beta}_{\widetilde{J}},\widetilde{\gamma}_{1},\ldots,\widetilde{\gamma}_{\widetilde{L}}\}}+S^{2}_{\{1,1\}},$$
we obtain a reducible spherical conical metric on $S^{2}_{\{k+1,\widetilde{\beta}_{1},\ldots,\widetilde{\beta}_{\widetilde{J}},\widetilde{\gamma}_{1},\ldots,\widetilde{\gamma}_{\widetilde{L}}\}}$ such that the singularity of angle $2\pi (k+1)$ is the unique saddle point of $\widetilde{\Phi}$, the singularities of angles $2\pi\widetilde{\beta}_{1},\ldots,2\pi\widetilde{\beta}_{\widetilde{J}}$ are minimum points of $\widetilde{\Phi}$, and $2\pi\widetilde{\gamma}_{1},\ldots,2\pi\widetilde{\gamma}_{\widetilde{L}}$ are maximum points of $\widetilde{\Phi}$. \par

\end{proof}

\textbf{Claim 3}: Under the condition $\mathcal{D}$, when $I\geq2$, there is a reducible spherical conical metric on $S^{2}_{\{\alpha_{1},\ldots,\alpha_{I},\beta_{1},\ldots,\beta_{J},\gamma_{1},\ldots,\gamma_{L}\}}$ such that the singularities of angles $2\pi\alpha_{1},\ldots,2\pi\alpha_{I}$ are saddle points of $\Phi$, the singularities of angles $2\pi\beta_{1},\ldots,2\pi\beta_{J}$ are minimal points of $\Phi$, the singularities of angles $2\pi\gamma_{1},\ldots,2\pi\gamma_{L}$ are maximal points of $\Phi$, and all saddle points lie on the same geodesic connected a minimum to a maximum point of $\Phi$.\par

\begin{proof}
First, we have the following facts.\par

\textbf{Fact 1}: According to \textbf{Claim 2}, there exists a reducible spherical conical metric on $S^{2}_{\{2,\beta_{1},\ldots,\beta_{J},\gamma_{1},\ldots,\gamma_{L}\}}$ such that the singularity of angle $4\pi$ is the unique saddle points of $\Phi$, the singularities of angles $2\pi\beta_{1},\ldots,2\pi\beta_{J}$ are minimum points of $\Phi$, and the singularities of angles $2\pi\gamma_{1},\ldots,2\pi\gamma_{L}$ are maximum points of $\Phi$.\par

\textbf{Fact 2}: Using mathematical induction to the integer $I$, we can prove that there exists a reducible spherical conical metric on $S^{2}_{\{\underbrace{2,\ldots,2}_{I\geq1},\beta_{1},\ldots,\beta_{J},\gamma_{1},\ldots,\gamma_{L}\}}$ such that the singularities of angles $\underbrace{4\pi,\ldots,4\pi}_{I}$ are saddle points of $\Phi$, the singularities of angles $2\pi\beta_{1},\ldots,2\pi\beta_{J}$ are minimum points of $\Phi$, the singularities of angles $2\pi\gamma_{1},\ldots,2\pi\gamma_{L}$ are maximum points of $\Phi$, and all saddle points lie on the same geodesic connected a minimum to a maximum point of $\Phi$.

Based on these facts, we will apply mathematical induction to  the sum $\sum\limits_{i=1}^{I}(\alpha_{i}-1)$ (fixing $I$) to demonstrate our claim.\par

Under the condition $\mathcal{D}$, suppose that when $\sum\limits_{i=1}^{I}(\alpha_{i}-1)=k\geq I$,
there exists a reducible spherical conical metric on $S^{2}_{\{\alpha_{1},\ldots,\alpha_{I},\beta_{1},\ldots,\beta_{J},\gamma_{1},\ldots,\gamma_{L}\}}$ such that
the singularities of angles $2\pi\alpha_{1},\ldots, 2\pi\alpha_{I}$ are saddle points of $\Phi$, the singularities of angles $2\pi\beta_{1},\ldots,2\pi\beta_{J}$ are minimal points of $\Phi$,  the singularities of angles $2\pi\gamma_{1},\ldots,2\pi\gamma_{L}$ are maximum points of $\Phi$, and all saddle points lie on the same geodesic connected a minimum point and a maximum point of $\Phi$. \par

When $\sum\limits_{i=1}^{I}(\widehat{\alpha}_{i}-1)=k+1 \geq I+1$, suppose that $\widehat{\beta}_{1},\ldots,\widehat{\beta}_{\widehat{J}},\widehat{\gamma}_{1},\ldots,\widehat{\gamma}_{L}$ are $\widehat{J}+\widehat{L}\geq1$ positive real numbers with $\widehat{\beta}_{j}\neq1,\forall j$ and $\widehat{\gamma}_{l}\neq1,\forall l$, and suppose that there exist two nonnegative integers $\widehat{p}$ and $\widehat{q}$ satisfied
\begin{equation*}
\begin{cases}
\widehat{p}+\widehat{q}=\sum\limits_{i=1}^{I}(\widehat{\alpha}_{i}-1)-(\widehat{J}+\widehat{L})+2,\\
\widehat{q}-\widehat{p}=\sum\limits_{j=1}^{\widehat{J}}\widehat{\beta}_{j}-\sum\limits_{l=1}^{\widehat{L}}\widehat{\gamma}_{l},\\
{\rm deg}(\overrightarrow{\widehat{r}})>\max\{\widehat{\alpha}_{i}-1\}_{1\leq i\leq I},
\end{cases}
\end{equation*}
where
$$\overrightarrow{\widehat{r}}=\{\widehat{\beta}_{1},\ldots,\widehat{\beta}_{\widehat{J}},-\widehat{\gamma}_{1},\ldots,-\widehat{\gamma}_{\widehat{L}},
\underbrace{1,\ldots,1}_{\widehat{p}},\underbrace{-1,\ldots,-1}_{\widehat{q}}\}.$$

Without loss of generality, we may suppose $\widehat{\alpha}_{1}\geq\widehat{\alpha}_{2}\geq\ldots\widehat{\alpha}_{I}\geq2, \alpha_{1}\geq 3, \widehat{J}+\widehat{p}\leq \widehat{L}+\widehat{q},\widehat{\gamma}_{1}\leq\widehat{\gamma}_{2}\leq\ldots\leq\widehat{\gamma}_{\widehat{q}}$ and $\widehat{\beta}_{1}\leq\widehat{\beta}_{1}\leq\ldots\leq \widehat{\beta}_{J}$. For other cases, the proofs are similar. \par

By \textbf{Lemma \ref{Lemma-A-2}},  the modified residue vector defined by
$$\overrightarrow{\widehat{r}}=\{ \widehat{\beta}_{1}-1,\ldots,\widehat{\beta}_{\widehat{J}},
-\widehat{\gamma}_{1},\ldots,-\widehat{\gamma}_{\widehat{L}},\underbrace{1,\ldots,1}_{\widehat{p}},\underbrace{-1,\ldots,-1}_{\widehat{q}-1}\}$$
satisfies ${\rm deg}(\overrightarrow{\widehat{r}})>\max\{\widehat{\alpha}_{1}-2,\ldots, \widehat{\alpha}_{I}-1\}$.\par

According to the assumption, there exists a reducible spherical conical metric on
 $$S^{2}_{\{\widehat{\alpha}_{1}-1,\widehat{\alpha}_{2}, \widehat{\alpha}_{3} ,\ldots,\widehat{\alpha}_{I+1},\widehat{\beta}_{1}-1,\ldots,\widehat{\beta}_{\widehat{J}},\widehat{\gamma}_{1},\ldots,\widehat{\gamma}_{\widehat{L}}\}}$$
such that the singularities of angles $2\pi(\widehat{\alpha}_{1}-1), 2\pi\widehat{\alpha}_{2}, 2\pi\widehat{\alpha}_{3}, \ldots, 2\pi\widehat{\alpha}_{I}$ are saddle points of $\widetilde{\Phi}$, the singularities of angles $2\pi(\widehat{\beta}_{1}-1),\ldots,2\pi\widehat{\beta}_{\widehat{J}}$ are minimum points of $\widetilde{\Phi}$, $2\pi\widehat{\gamma}_{1},\ldots,2\pi\widehat{\gamma}_{\widehat{L}}$ are maximum points of $\widetilde{\Phi}$, and all saddle points lie on the same geodesic connected a minimum point and a maximum point of $\widetilde{\Phi}$.\par

\begin{figure}[htbp]
\centering
\includegraphics[width=11cm]{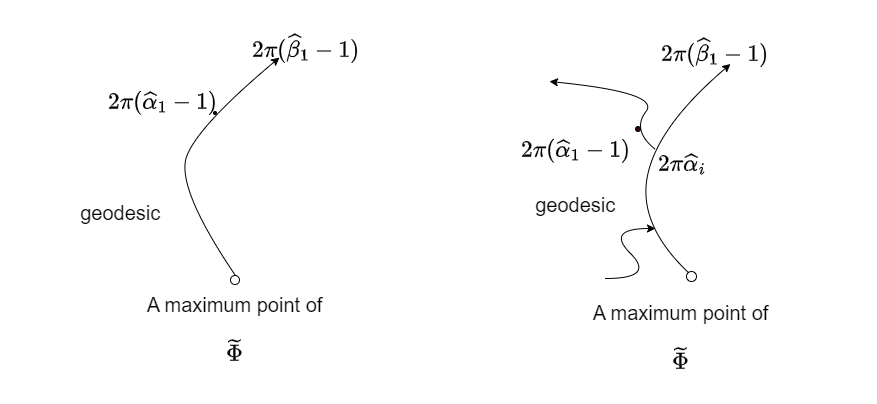}
\caption{The geodesic of all saddle points lies on.}
\end{figure}

 When there is a geodesic from the singularity of angle $2\pi(\widehat{\beta}_{1}-1)$ to a maximum point of $\widetilde{\Phi}$ directly such that the singularity of angle $2\pi(\widehat{\alpha}_{1}-1)$ also lies on this geodesic (see Figure 13), then on $S^{2}_{\{\widehat{\alpha}_{1}-1,\widehat{\alpha}_{2} ,\ldots,\widehat{\alpha}_{I},\widehat{\beta}_{1}-1,\ldots,\widehat{\beta}_{J},\widehat{\gamma}_{1},\ldots,\widehat{\gamma}_{L}\}}$, by gluing a standard football $S^{2}_{\{1,1\}}$, we can construct a reducible spherical conical metric on $S^{2}_{\{\widehat{\alpha}_{1},\ldots,\widehat{\alpha}_{I},\widehat{\beta}_{1},\ldots,\widehat{\beta}_{\widehat{J}},\widehat{\gamma}_{1},\ldots,\widehat{\gamma}_{\widehat{L}}\}}$ such that the singularities of angles $2\pi\widehat{\alpha}_{1},\ldots, 2\pi\widehat{\alpha}_{I}$ are saddle points of $\widehat{\Phi}$, the singularities of angles $2\pi\widehat{\beta}_{1},\ldots,2\pi\widehat{\beta}_{\widehat{J}}$ are minimum points of $\widehat{\Phi}$, the singularities of angles $2\pi\widehat{\gamma}_{1},\ldots,2\pi\widehat{\gamma}_{\widehat{L}}$ are maximum points of $\widehat{\Phi}$ and all saddle points lie on the same geodesic connected a minimum point and a maximum point of $\widehat{\Phi}$.\par

 When there is no geodesic from the singularity of angle $2\pi(\widehat{\beta}_{1}-1)$ to a maximum point of $\widetilde{\Phi}$ directly such that the singularity of angle $2\pi(\widehat{\alpha}_{1}-1)$ also lies on this geodesic, since all saddle points lie on the same geodesic connected a minimum point and a maximum point of $\widetilde{\Phi}$, we can adjust the position of the singularity of $2\pi(\widehat{\alpha}_{1}-1)$ relative to some saddle points (see Figure 13). This adjustment ensures the existence of a geodesic from the singularity of angle $2\pi(\widehat{\alpha}_{1}-1)$ a maximum point of $\widetilde{\Phi}$ directly such that the singularity of angle $2\pi(\widehat{\alpha}_{1}-1)$ also lie on this geodesic. Then, by gluing a standard football $S^{2}_{\{1,1\}}$ to $S^{2}_{\{\widehat{\alpha}_{1}-1,\widehat{\alpha}_{2},\ldots,\widehat{\alpha}_{I},\widehat{\beta}_{1}-1,\ldots,\widehat{\beta}_{J},\widehat{\gamma}_{1},\ldots,\widehat{\gamma}_{L}\}}$, we can construct a reducible spherical conical metric on $S^{2}_{\{\widehat{\alpha}_{1},\ldots,\widehat{\alpha}_{I},\widehat{\beta}_{1},\ldots,\widehat{\beta}_{\widehat{J}},\widehat{\gamma}_{1},\ldots,\widehat{\gamma}_{\widehat{L}}\}}$ such that the singularities of angles $2\pi\widehat{\alpha}_{1},\ldots, 2\pi\widehat{\alpha}_{I}$ are saddle points of $\widehat{\Phi}$, the singularities of angles $2\pi\widehat{\beta}_{1},\ldots,2\pi\widehat{\beta}_{\widehat{J}}$ are minimum points of $\widehat{\Phi}$, the singularities of angles $2\pi\widehat{\gamma}_{1},\ldots,2\pi\widehat{\gamma}_{\widehat{L}}$ are maximum points of $\widehat{\Phi}$ and all saddle points lie on the same geodesic connected a minimum point and a maximum point of $\widehat{\Phi}$.\par

\end{proof}

\section{Proof of \textbf{Theorem \ref{Main-thm-5}}}

The necessity of the argument is similar as the proof of \textbf{Theorem \ref{Main-th-4}}, and can be easily obtain, so we leave it to the interesting readers. We will now establish its sufficiency.\par

In this section, we will also utilize the following identity to illustrate the construction of reducible spherical conical metrics:
$$\mathcal{M}_{\{\alpha_{1},\ldots,a_{N_{1}}\}}=\widetilde{\mathcal{M}}_{\{\beta_{1},\ldots,\beta_{N_{2}}\}}+S^{2}_{\{\gamma_{1},\gamma_{1}\}}$$
This identity shows that by attaching a football $S^{2}_{\{\gamma_{1},\gamma_{1}\}}$ to $\widetilde{\mathcal{M}}_{\{\beta_{1},\ldots,\beta_{N_{2}}\}}$, which possesses a reducible spherical conical metric, we can construct a reducible spherical conical metric on $\mathcal{M}_{\{\alpha_{1},\ldots,a_{N_{1}}\}}$.\par

\textbf{Claim 1 :} Let $g\geq1$ be an integer. Let $\beta\neq1$ be a positive real number and  $2\leq\alpha_{1},\ldots,\alpha_{I}$ be $I\geq1$ integers.
If $\sum\limits_{i=1}^{I}(\alpha_{i}-1)=2g$,
then there exist a  compact Riemann surface $\mathcal{M}^{g}$  with genus $g$ and a reducible spherical conical metric ${\rm d}s^{2}$ on $\mathcal{M}^{g}_{\{\alpha_{1},\ldots,\alpha_{I},\beta,\beta\}}$
such that the function $\Phi$ defined by equation (\ref{E-1}) satisfies that
\begin{enumerate}
\item the singularities of angles $2\pi\alpha_{1},\ldots,2\pi\alpha_{I}$ are saddle points of $\Phi$,  and they lie on the same geodesic connected a minimum point and a maximum point of $\Phi$;
\item  one singularity of angle $2\pi\beta$ is a minimum points of $\Phi$,
and another singularity of angle $2\pi\beta$ is a maximum point of $\Phi$.
\end{enumerate}

\begin{proof}
\begin{figure}[htbp]
\centering
\includegraphics[width=11cm]{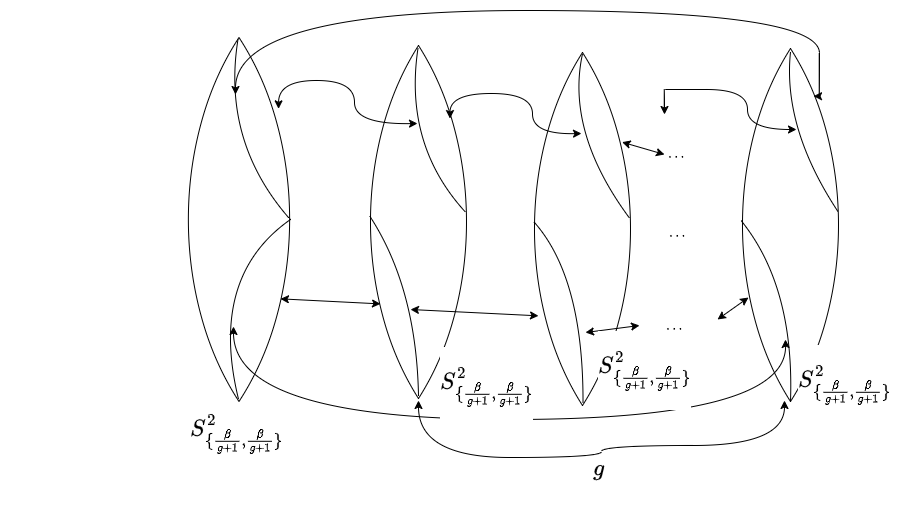}
\caption{$\mathcal{M}^{g}_{\{2g+1,\beta,\beta\}}.$}
\end{figure}
When $I=1$, it follows that $\alpha_{1}=2g+1$.  Since (as show in Figure 14)
$$\mathcal{M}_{\{2g+1,\beta,\beta\}}=S^{2}_{\{\frac{\beta}{g+1},\frac{\beta}{g+1}\}}
+\underbrace{S^{2}_{\{\frac{\beta}{g+1},\frac{\beta}{g+1}\}}+\ldots+S^{2}_{\{\frac{\beta}{g+1},\frac{\beta}{g+1}\}}}_{g},$$
the claim holds true.\par

When $I\geq2$, it follows that $\sum\limits_{i=1}^{I}(\alpha_{i}-1)\geq2$. Without loss of generality, we may suppose $\alpha_{1}\geq\ldots\geq\alpha_{I}$.\par

If $g=2$, by $\sum\limits_{i=1}^{I}(\alpha_{i}-1)=2g$, we obtain $I=2$ and $\alpha_{1}=\alpha_{2}=3$ or $\alpha_{1}=4,\alpha_{2}=2$, $I=3$ and $\alpha_{1}=3,\alpha_{2}=\alpha_{3}=2$, or $I=4$ and $\alpha_{1}=\alpha_{2}=\alpha_{3}=\alpha_{4}=2$. \par
\begin{figure}[htbp]
\centering
\includegraphics[width=13cm]{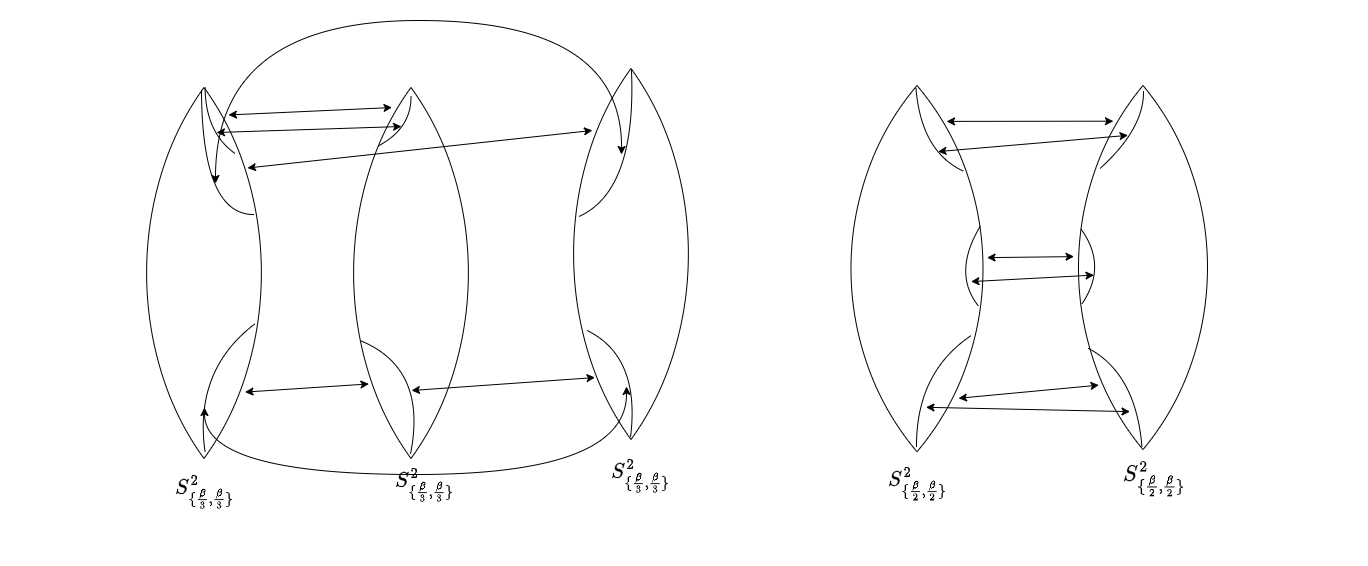}
\caption{$\mathcal{M}^{2}_{\{3,2,2,\beta,\beta\}}~(left)~and~\mathcal{M}^{2}_{\{2,2,2,2,\beta,\beta\}}~(right).$}
\end{figure}
When $I=2$ and $\alpha_{1}=\alpha_{2}=3$ or $\alpha_{1}=4,\alpha_{2}=2$, since (as show in Figure 15 )
$$\mathcal{M}^{2}_{\{3,3,\beta,\beta\}}=S^{2}_{\{\frac{\beta}{3},\frac{\beta}{3}\}}+S^{2}_{\{\frac{\beta}{3},\frac{\beta}{3}\}}+S^{2}_{\{\frac{\beta}{3},\frac{\beta}{3}\}},$$
$$\mathcal{M}^{2}_{\{4,2,\beta,\beta\}}=T^{1}_{\{3,\frac{2\beta}{3},\frac{2\beta}{3}\}}+S^{2}_{\{\frac{\beta}{3},\frac{\beta}{3}\}}=S^{2}_{\{\frac{\beta}{3},\frac{\beta}{3}\}}+S^{2}_{\{\frac{\beta}{3},\frac{\beta}{3}\}}+S^{2}_{\{\frac{\beta}{3},\frac{\beta}{3}\}},$$
the claim holds true.
\begin{figure}[htbp]
\centering
\includegraphics[width=13cm]{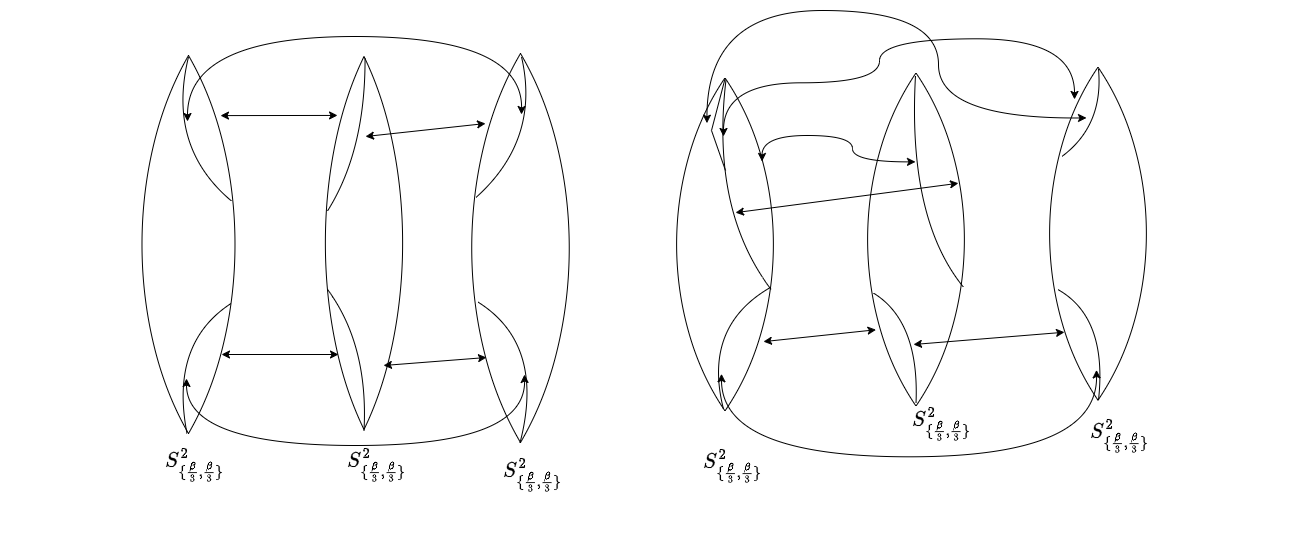}
\caption{$\mathcal{M}^{2}_{\{3,3,\beta,\beta\}}~(left)~and~\mathcal{M}^{2}_{\{4,2,\beta,\beta\}}~(right).$}
\end{figure}

For other two cases, the proofs are similar (as show in Figure 16 ).\par

Suppose that when $g=k\geq2$ and $\sum\limits_{i=1}^{I}(\alpha_{i}-1)=2g$, there exist a compact Riemann surface $\mathcal{M}^{k}$  with genus $k$ and a reducible spherical conical metric ${\rm d}s^{2}$ on $\mathcal{M}^{k}_{\{\alpha_{1},\ldots,\alpha_{I},\beta,\beta\}}$
such that the function $\Phi$ defined by equation (\ref{E-1}) satisfies that
\begin{enumerate}
\item the singularities of angles $2\pi\alpha_{1},\ldots,2\pi\alpha_{I}$ are saddle points of $\Phi$ and they lie on the same geodesic connected a minimum point and a maximum point of $\Phi$;
\item  one singularity of angle $2\pi\beta$ is a minimum point of $\Phi$,
and another singularity of angle $2\pi\beta$ is a maximum point of $\Phi$.
\end{enumerate}

When $g=k+1$ and $ \sum\limits_{i=1}^{\widetilde{I}}(\widetilde{\alpha}_{i}-1)=2(k+1)$, it follows that
 $\sum\limits_{i=1}^{\widetilde{I}}(\widetilde{\alpha}_{i}-1)-2=2k$. Suppose that $\widetilde{\alpha}_{1}\geq\widetilde{\alpha}_{2}\geq\ldots\geq\widetilde{\alpha}_{\widetilde{I}}$, if $\widetilde{\alpha}_{1}=2$, then $\widetilde{I}=2(k+1)$ and $\widetilde{\alpha}_{2}=\ldots=\widetilde{\alpha}_{\widetilde{I}}=2$. Obviously (as show in Figure 17),
$$ \mathcal{M}^{k+1}_{\{\underbrace{2,\ldots,2}_{2(k+1)},\widetilde{\beta},\widetilde{\beta}\}}=S^{2}_{\{\frac{\widetilde{\beta}}{2}, \frac{\widetilde{\beta}}{2}\}}+S^{2}_{\{\frac{\widetilde{\beta}}{2},\frac{\widetilde{\beta}}{2}\}},$$
the claim holds true.
\begin{figure}[htbp]
\centering
\includegraphics[width=13cm]{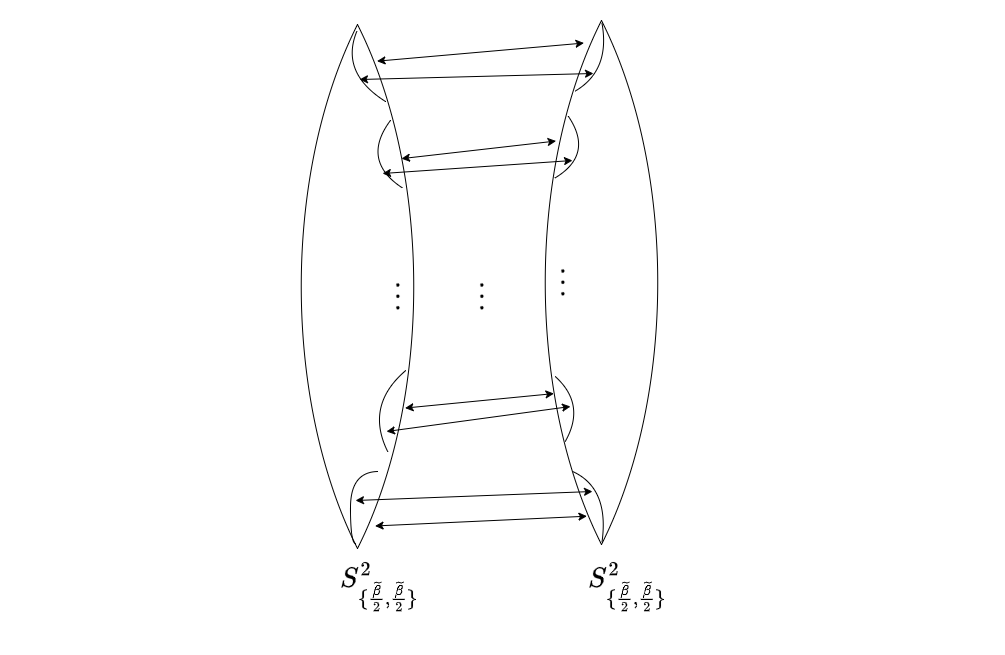}
\caption{$\mathcal{M}^{k+1}_{\{2,\ldots,2,\widetilde{\beta},\widetilde{\beta}\}}.$}
\end{figure}

If $\widetilde{\alpha}_{1}=3, \widetilde{\alpha}_{2}=2$, then $\widetilde{I}=2k-1$ and $\widetilde{\alpha}_{3}=\ldots=\widetilde{\alpha}_{\widetilde{I}}=2$.  Since
$$ \widetilde{\mathcal{M}}^{k+1}_{\{3,\underbrace{2,\ldots,2}_{2k},\widetilde{\beta},\widetilde{\beta}\}}=\mathcal{M}^{k}_{\{\underbrace{2,\ldots,2}_{2k},\frac{\widetilde{\beta}}{2}, \frac{\widetilde{\beta}}{2}\}}+S^{2}_{\{\frac{\widetilde{\beta}}{2},\frac{\widetilde{\beta}}{2}\}},$$
the claim holds true.\par

If $\widetilde{\alpha}_{1},\widetilde{\alpha}_{2}\geq 3$, by the assumption, there exist a compact Riemann surface $\widehat{\mathcal{M}}^{k}$ with genus $k$ and a  reducible spherical conical metric $\widehat{{\rm d}s^{2}}$ on $$\widehat{\mathcal{M}}^{k}_{\{\widetilde{\alpha}_{1}-1,\widetilde{\alpha}_{2}-1,\widetilde{\alpha}_{3},\ldots,\widetilde{\alpha}_{\widetilde{I}},\frac{\widetilde{\beta}}{2},\frac{\widetilde{\beta}}{2}\}}$$
such that the function $\widehat{\Phi}$ defined by equation (\ref{E-1}) satisfies that

\begin{enumerate}
\item the singularities of angles $2\pi(\widetilde{\alpha}_{1}-1),2\pi(\widetilde{\alpha}_{2}-1),2\pi\widetilde{\alpha}_{3},\ldots,2\pi\widetilde{\alpha}_{\widetilde{I}}$ are saddle points of $\widehat{\Phi}$, and they lie on the same geodesic connected a minimum point and a maximum point of $\widehat{\Phi}$;
\item a singularity of angle $2\pi\widetilde{\beta}$ is a minimum point of $\widehat{\Phi}$,
and another singularity of angles $2\pi\widetilde{\beta}$ is a maximum point of $\widehat{\Phi}$.
\end{enumerate}

Since
\begin{equation*}
\widetilde{\mathcal{M}}^{k+1}_{\{\widetilde{\alpha}_{1},\widetilde{\alpha}_{2},\ldots,\widetilde{\alpha}_{\widetilde{I}},\widetilde{\beta},\widetilde{\beta}\}}=
 \widehat{\mathcal{M}}^{k}_{\{\widetilde{\alpha}_{1}-1,\widetilde{\alpha}_{2}-1,\widetilde{\alpha}_{3},\ldots,\widetilde{\alpha}_{\widetilde{I}},\frac{\widetilde{\beta}}{2},\frac{\widetilde{\beta}}{2}\}}
 +S^{2}_{\{\frac{\widetilde{\beta}}{2},\frac{\widetilde{\beta}}{2}\}},
\end{equation*}
the claim holds true.
\end{proof}

\textbf{Claim 2 :}  Let $g\geq1$ be an integer. Let $\beta\neq1$ be a positive real numbers and  $2\leq\alpha_{1},\ldots,\alpha_{I}$ be $I\geq1$ integers.
If there exists an integer $J\geq1$ such that
$2J=\sum\limits_{i=1}^{I}(\alpha_{i}-1)-2g+2$,
then there exist a  compact Riemann surface $\mathcal{M}^{g}$  with genus $g$ and a reducible spherical conical metric ${\rm d}s^{2}$ on $\mathcal{M}^{g}_{\{\alpha_{1},\ldots,\alpha_{I},\underbrace{\beta,\ldots,\beta}_{2J}\}}$
such that the function $\Phi$ defined by equation (\ref{E-1}) satisfies that

\begin{enumerate}
\item the singularities of angles $2\pi\alpha_{1},\ldots,2\pi\alpha_{I}$ are saddle points of $\Phi$ and  they lie on the same geodesic which connects a maximum point and a minimum point of $\Phi$;
\item $J$ singularities of angles $2\pi\beta$ are minimum points of $\Phi$,
and other $J$ singularities of angles $2\pi\beta$ are maximum points of $\Phi$.
\end{enumerate}

\begin{proof}
When $J=1$, by \textbf{Claim 1},  the claim holds true.\par

Suppose that when $J=k\geq1$ and $ 2k=\sum\limits_{i=1}^{I}(\alpha_{i}-1)-2g+2$, there exist a compact Riemann surface $\mathcal{M}^{g}$  with genus $g$ and a reducible spherical conical metric ${\rm d}s^{2}$ on $$\mathcal{M}^{g}_{\{\alpha_{1},\ldots,\alpha_{I},\underbrace{\beta,\ldots,\beta}_{2k}\}}$$
such that the function $\Phi$ defined by equation (\ref{E-1}) satisfies that

\begin{enumerate}
\item the singularities of angles $2\pi\alpha_{1},\ldots,2\pi\alpha_{I}$ are saddle points of $\Phi$ and  they are lie on the same geodesic which connects a maximum point and a minimum point of $\Phi$;
\item $k$ singularities of angles $2\pi\beta$ are minimum points of $\Phi$, and
other $k$ singularities of angles $2\pi\beta$ are maximum points of $\Phi$.
\end{enumerate}

When $J=k+1$ and $2(k+1)=\sum\limits_{i=1}^{\widetilde{I}}(\widetilde{\alpha}_{i}-1)-2g+2$, it follows that $2k=[\sum\limits_{i=1}^{\widetilde{I}}(\widetilde{\alpha}_{i}-1)-2]-2g+2$. By our assumption, there exist a compact Riemann surface $\widehat{M}^{g}$  with genus $g$ and a  reducible spherical conical metric $\widehat{{\rm d}s^{2}}$ on $$\widehat{\mathcal{M}}^{g}_{\{\widetilde{\alpha}_{1}-1,\widetilde{\alpha}_{2}-1,\widetilde{\alpha}_{3},\ldots,\widetilde{\alpha}_{\widetilde{I}},\underbrace{\widetilde{\beta},\ldots,\widetilde{\beta}}_{2k}\}}$$
such that the function $\widehat{\Phi}$ defined by equation (\ref{E-1}) satisfies that

\begin{enumerate}
\item  the singularities of angles $2\pi(\widetilde{\alpha}_{1}-1),2\pi(\widetilde{\alpha}_{2}-1),2\pi\widetilde{\alpha}_{1},\ldots,2\pi\widetilde{\alpha}_{\widetilde{I}}$ are saddle points of $\widehat{\Phi}$ and they are in the same geodesic which connects a maximum point and a minimum point of $\widehat{\Phi}$;
\item $k$ singularities of angles $2\pi\widetilde{\beta}$ are minimum points of $\widehat{\Phi}$,
and other $k$ singularities of angles $2\pi\widetilde{\beta}$ are maximum points of $\widehat{\Phi}$.
\end{enumerate}

Since
$$\widetilde{\mathcal{M}}^{g}_{\{\widetilde{\alpha}_{1},\ldots,\widetilde{\alpha}_{\widetilde{I}},\underbrace{\widetilde{\beta},\ldots,\widetilde{\beta}}_{2(k+1)}\}}=
 \widehat{\mathcal{M}}^{g}_{\{\widetilde{\alpha}_{1}-1,\widetilde{\alpha}_{2}-1,\widetilde{\alpha}_{3},\ldots,\widetilde{\alpha}_{\widetilde{I}},\underbrace{\widetilde{\beta},\ldots,\widetilde{\beta}}_{2k}\}}
 +S^{2}_{\{\widetilde{\beta},\widetilde{\beta}\}},$$
the claim holds true.\par

\end{proof}

\begin{remark}
By the proof, we know that if $\beta=1$, \textbf{Claims 1 and 2} hold also true. This means that all singularities are saddle points of $\Phi$.
\end{remark}

\textbf{Claim 3 :} Let $g\geq1$ be an integer. Let $\beta_{1},\ldots,\beta_{J},\gamma_{1},\ldots,\gamma_{L}\neq1$ be $J+L$ positive real numbers and  $2\leq\alpha_{1},\ldots,\alpha_{I}$ be $I\geq1$ integers$~(J,L\geq0)$.
If there exists exist two nonnegative integers $p$ and $q$ such that $p+J\geq1,~q+L\geq1$ and
\begin{equation*}
\begin{cases}
(p+J)+(q+L)=\sum\limits_{i=1}^{I}(\alpha_{i}-1)+2-2g,\\
q-p=\sum\limits_{j=1}^{J}\beta_{j}-\sum\limits_{l=1}^{L}\gamma_{l},
\end{cases}
\end{equation*}
then, there exist a  compact Riemann surface $\mathcal{M}^{g}$  with genus $g$ and a reducible spherical conical metric ${\rm d}s^{2}$ on $$\mathcal{M}^{g}_{\{\alpha_{1},\ldots,\alpha_{I},\beta_{1},\ldots,\beta_{J},\gamma_{1},\ldots,\gamma_{L}\}}$$
such that the function $\Phi$ defined by equation (\ref{E-1}) satisfies that
\begin{enumerate}
\item  the singularities of angles $2\pi\alpha_{1},\ldots,2\pi\alpha_{I}$ are saddle points of $\Phi$;
\item the singularities of angles $2\pi\beta_{1},\ldots,2\pi\beta_{J}$ are  minimum points of $\Phi$,
and the singularities of angles $2\pi\gamma_{1},\ldots$, $2\pi\gamma_{L}$ are maximum points of $\Phi$.
\end{enumerate}

\begin{proof}
When $\sum\limits_{i=1}^{I}(\alpha_{i}-1)-2g=0$, it follows that $p+J+q+L=2$. Obviously, the claim holds true.\par

Suppose that when $\sum\limits_{i=1}^{I}(\alpha_{i}-1)-2g=k\geq0$ and there exist two nonnegative integers $p$ and $q$ such that
$p+J\geq1,q+L\geq1$ and
\begin{equation*}
\begin{cases}
(p+J)+(q+L)=\sum\limits_{i=1}^{I}(\alpha_{i}-1)+2-2g,\\
q-p=\sum\limits_{j=1}^{J}\beta_{j}-\sum\limits_{l=1}^{L}\gamma_{l},
\end{cases}
\end{equation*}
there exist a  compact Riemann surface $\mathcal{M}^{g}$  with genus $g$ and a reducible spherical conical metric ${\rm d}s^{2}$ on
$$\mathcal{M}^{g}_{\{\alpha_{1},\ldots,\alpha_{I},\beta_{1},\ldots,\beta_{J},\gamma_{1},\ldots,\gamma_{L}\}}$$
such that the function $\Phi$ defined by equation (\ref{E-1}) satisfies that
\begin{enumerate}
\item  the singularities of angles $2\pi\alpha_{1},\ldots,2\pi\alpha_{I}$ are saddle points of $\Phi$ and they lie on the same geodesic connected a minimum point and a maximum point of $\Phi$;
\item the singularities of angles $2\pi\beta_{1},\ldots,2\pi\beta_{J}$ are minimum points of $\Phi$,
and the singularities of angles $2\pi\gamma_{1},\ldots $, $2\pi\gamma_{L}$ are maximum points of $\Phi$.
\end{enumerate}

When $\sum\limits_{i=1}^{I}(\widetilde{\alpha}_{i}-1)-2g=k+1$ and
there exist two nonnegative integers $\widetilde{p}$ and $\widetilde{q}$ such that
$\widetilde{p}+\widetilde{J}\geq1,\widetilde{q}+\widetilde{L}\geq1$ and
\begin{equation*}
\begin{cases}
(\widetilde{p}+\widetilde{J})+(\widetilde{q}+\widetilde{L})=\sum\limits_{i=1}^{I}(\widetilde{\alpha}_{i}-1)+2-2g,\\
\widetilde{q}-\widetilde{p}=\sum\limits_{j=1}^{\widetilde{J}}\widetilde{\beta}_{j}-\sum\limits_{l=1}^{\widetilde{L}}\widetilde{\gamma}_{l},
\end{cases}
\end{equation*}
without loss of generality, suppose $\widetilde{\beta}_{1}>\widetilde{\gamma}_{1}$ (if $J=L=\frac{\sum\limits_{i=1}^{I}\widetilde{\alpha}_{i}-1+2-2g}{2}$ and $\widetilde{\beta}_{1}=\ldots=\widetilde{\beta}_{J}=\gamma_{1}=\ldots=\widetilde{\gamma}_{L}$, \textbf{Claim 2} says that the claim holds true), by the assumption,  for any $i$, there exist a compact Riemann surface $\widehat{\mathcal{M}}^{g}$ with genus $g$ and a reducible spherical conical metric $\widehat{{\rm d}s^{2}}$ on $$\widehat{\mathcal{M}}^{g}_{\{\widetilde{\alpha}_{1},\ldots,\widetilde{\alpha}_{i}-1,\ldots,\widetilde{\alpha}_{I},\widetilde{\beta}_{1}-\widetilde{\gamma}_{1},
\widetilde{\beta}_{2},\ldots,\widetilde{\beta}_{J},\widetilde{\gamma}_{2},\ldots,\widetilde{\gamma}_{L}\}}$$
such that the function $\widehat{\Phi}$ defined by equation (\ref{E-1}) satisfies that
\begin{enumerate}
\item  the singularities of angles $2\pi\widetilde{\alpha}_{1},\ldots, 2\pi(\widetilde{\alpha}_{i}-1),\ldots,2\pi\widetilde{\alpha}_{I}$ are saddle points of $\widehat{\Phi}$ and they lie on the same geodesic connected a minimum point and a maximum point of $\widehat{\Phi}$;
\item the singularities of angles $2\pi(\widetilde{\beta}_{1}-\widetilde{\gamma}_{1}),\ldots,2\pi\widetilde{\beta}_{\widetilde{J}}$ are minimum points of $\widehat{\Phi}$, and the singularities of angles $2\pi\widetilde{\gamma}_{2},\ldots,2\pi\widetilde{\gamma}_{\widetilde{L}}$ are maximum points of $\widehat{\Phi}$.
\end{enumerate}

 Since
$$\widetilde{\mathcal{M}}^{g}_{\{\widetilde{\alpha}_{1},\ldots,\widetilde{\alpha}_{I},\widetilde{\beta}_{1},\ldots,\widetilde{\beta}_{\widetilde{J}},\widetilde{\gamma}_{1},\ldots,\widetilde{\gamma}_{L}\}}
=\widehat{\mathcal{M}}^{g}_{\{\widetilde{\alpha}_{1},\ldots,\widetilde{\alpha}_{i}-1,\ldots,\widetilde{\alpha}_{I},\widetilde{\beta}-\widetilde{\gamma}_{1},
\widetilde{\beta}_{2},\ldots,\widetilde{\beta}_{J},\widetilde{\gamma}_{2},\ldots,\widetilde{\gamma}_{L}\}}+S^{2}_{\{\widetilde{\gamma}_{1}, \widetilde{\gamma}_{1}\}},$$ the claim is true.
\end{proof}

By \textbf{Claims 1,2 and 3}, the sufficiency of \textbf{Theorem \ref{Main-thm-5}} is established.

\textbf{Acknowledgments.} This work was partially completed during Wei's visit to the Institute of Geometry and Physics (IGP) in 2022. Wei expresses heartfelt gratitude to Professors Xiuxiong Chen, Qing Chen, Bing Wang, Meijun Zhu, and Xiaowei Xu for their gracious hospitality and valuable support. Wei is partially supported by the National Natural Science Foundation of China (Grant No. 12171140). Wu is partially supported by the Fundamental Research Funds for the Central Universities and the CAS Project for Young Scientists in Basic Research (Grant No. YSBR-001). Xu is supported in part by the National Natural Science Foundation of China (Grant Nos. 12271495, 11971450, and 12071449) and the CAS Project for Young Scientists in Basic Research (Grant No. YSBR-001).

\textbf{Declarations}

\textbf{Data Availability Statement}  This manuscript has no associated data.

\textbf{Competing interests} On behalf of all authors, the corresponding author states that there is no conflict of interest.

\noindent
Zhiqiang Wei\\
School of Mathematics and Statistics, Henan University, Kaifeng 475004 P. R. China\\
Center for Applied Mathematics of Henan Province, Henan University, Zheng zhou 450046 P. R. China\\
Email: weizhiqiang15@mails.ucas.edu.cn. ~or~10100123@vip.henu.edu.cn.\\
Yingyi Wu\\
School of Mathematical Sciences, University of Chinese Academy of Sciences, Beijing 100049, P. R. China\\
Email: wuyy@ucas.ac.cn\\
Bin Xu\\
CAS Wu Wen-Tsun Key Laboratory of Mathematics and School of Mathematical Sciences\\
University of Science and Technology of China, Hefei 230026, P. R. China\\
Email: bxu@ustc.edu.cn.

\end{document}